\pdfoutput=1
\documentclass[11pt]{article}
 \usepackage{graphicx}
\usepackage{amssymb}
\usepackage{amsthm}
\usepackage{epstopdf}
\usepackage{enumitem}
\usepackage{amsmath}
\usepackage{amsfonts}

\def\phi{{\varphi}}

\def\P{{\mathcal P}}

\DeclareSymbolFont{AMSb}{U}{msb}{m}{n}
\DeclareMathSymbol{\N}{\mathbin}{AMSb}{"4E}
\DeclareMathSymbol{\Z}{\mathbin}{AMSb}{"5A}
\DeclareMathSymbol{\R}{\mathbin}{AMSb}{"52}
\DeclareMathSymbol{\Q}{\mathbin}{AMSb}{"51}
\DeclareMathSymbol{\I}{\mathbin}{AMSb}{"49}
\DeclareMathSymbol{\C}{\mathbin}{AMSb}{"43}

\def\be{\begin{equation}}
\def\ber{\begin{eqnarray}}
\def\eer{\end{eqnarray}}

\def\vv{{\bf v}}

\def\beq{\begin{equation}}
\def\eeq{\end{equation}}
\newcommand{\E}{\mathbb{E}}
 \newcommand{\tens}{%
  \mathbin{\mathop{\otimes}}%
}

\begin{document}

\addtolength{\textheight}{0 cm} \addtolength{\hoffset}{0 cm}
\addtolength{\textwidth}{0 cm} \addtolength{\voffset}{0 cm}
\def\XX{\mathcal X}

\newcommand{\aaa}{\mathbb{A}}
\newcommand{\bb}{\mathbb{B}}
\newcommand{\cc}{\mathbb{C}}
\newcommand{\dd}{\mathbb{D}}
\newcommand{\ee}{\mathbb{E}}
\newcommand{\mm}{\mathbb{M}}
\newcommand{\rr}{\mathbb{R}}
\newcommand{\pp}{\mathbb{P}}
\newcommand{\qq}{\mathbb{Q}}
\newcommand{\ttt}{\mathbb{T}}
\newcommand{\zz}{\mathbb{Z}}

\def\vep{\varepsilon}
\def\<{\langle}
\def\>{\rangle}
\def\dsubset{\subset\subset}
\newcommand{\as}[1]{\begin{align*}#1\end{align*}}
\newcommand{\ald}[1]{\begin{aligned}#1\end{aligned}}

\newcommand{\Var}{\mathrm{Var}}
\newcommand{\mcal}[1]{\mathcal{#1}}

\def\AA{\mathcal A}
\def\BB{\mathcal B}
\def\CC{\mathcal C}
\def\DD{\mathcal D}
\def\FF{\mathcal F}
\def\GG{\mathcal G}
\def\EE{\mathcal E}
\def\JJ{\mathcal J}
\def\KK{\mathcal K}
\def\LL{\mathcal L}
\def\MM{\mathcal M}
\def\PP{\mathcal P}
\def\SS{\mathcal S}
\def\VV{\mathcal V}
\def\T{\mathcal T}
\def\d{\, {\rm d}}
 \newcommand{\lf}{\left}
\newcommand{\rt}{\right}
\newcommand{\vphi}{\varphi}
\newcommand{\no}{\nonumber}
\newcommand{\eq}[1]{\begin{equation}#1\end{equation}}
\newcommand{\eqs}[1]{\begin{equation*}#1\end{equation*}}

\newcommand{\ZZ}{\mathbb{Z}}
\newcommand{\Rm}{\mathbb{R}}
\newcommand{\RR}{\mathbb{R}}
\newcommand{\NN}{\mathbb{N}}
\newcommand{\sU}{\mathcal{U}}
\newcommand{\sF}{\mathcal{F}}
\newcommand{\sM}{\mathcal{M}}
\newcommand{\sS}{\mathcal{S}}
\newcommand{\mL}{\mathcal{L}}
\newcommand{\ac}{\hbox{\small ac}}
\newcommand{\mC}{\ensuremath{\mathcal{C}}}
\newcommand{\mU}{\ensuremath{\mathcal{U}}}
\newcommand{\mT}{\ensuremath{\mathcal{T}}}
\newcommand{\mS}{\ensuremath{\mathcal{S}}}
\newcommand{\mF}{\ensuremath{\mathcal{F}}}
\newcommand{\Nm}{\ensuremath{\mathbb{N}}}
\newcommand{\Zm}{\ensuremath{\mathbb{Z}}}
\newcommand{\Hm}{\ensuremath{\mathbb{H}}}
\newcommand{\mM}{\ensuremath{\mathcal{M}}}
\newcommand{\mK}{\ensuremath{\mathcal{K}}}
\newcommand{\mD}{\ensuremath{\mathcal{D}}}
\newcommand{\mA}{\ensuremath{\mathcal{A}}}
\newcommand{\mO}{\ensuremath{\mathcal{O}}}
\newcommand{\mI}{\ensuremath{\mathcal{I}}}
\newcommand{\mB}{\ensuremath{\mathcal{B}}}
\newcommand{\Tm}{\ensuremath{\mathbb{T}}}
\newcommand{\mE}{\ensuremath{\mathcal{E}}}
\newcommand{\vs}{\vspace{.5cm}}

\newcommand{\prf}[1]{\begin{proof}#1\end{proof}}
\def\qed {\mbox{}\hfill {\small \fbox{}} \\}  
\def\lto{\longrightarrow}
\def\lmto{\longmapsto}
\def\eq{\Longleftrightarrow}
\def\leq{\leqslant}
\def\geq{\geqslant}
\def \uK {K^+}
\def \oK {K^-}
\def \calT {\mathcal T}

\newtheorem{thm}{Theorem}[section]
\newtheorem{lem}[thm]{Lemma}
\newtheorem{cor}[thm]{Corollary}
\newtheorem{prop}[thm]{Proposition}
\newtheorem{defn}[thm]{Definition}
\newtheorem{rmk}[thm]{Remark}
\def\proof {\noindent{\sc{Proof. }}}
\def\qed {\mbox{}\hfill {\small \fbox{}} \\}  

\newcommand{\grad}{\operatorname{grad}}
\newcommand{\Leg}{\mathcal{L}}

\newcommand{\lbstoc}[0]{\underline{B}}
\newcommand{\ubstoc}[0]{\overline{B}^{\text{stoc}}}
\newcommand{\deriv}[2]{\ensuremath{\frac{d{#1}}{d{#2}}}}
\newcommand{\Id}[1]{\ensuremath{\boldsymbol{1}[#1]}}
\newcommand{\abs}[1]{\ensuremath{\left\lvert#1\right\rvert}}

\renewcommand{\P}{\ensuremath{\mathbb{P}}}

\newcommand{\bracket}[1]{\ensuremath{\left(#1\right)}}
\newcommand{\expect}[2][]{\ensuremath{\mathbb{E}_{#1}\left[#2\right]}}
\newcommand{\expcond}[2]{\ensuremath{\mathbb{E}\left[#1\middle\rvert #2\right]}}
\newcommand{\ball}[2]{\ensuremath{B(#1,#2)}}
\newcommand{\proj}[1]{\ensuremath{\text{proj}_{#1}}}
\newcommand{\pderiv}[2]{\ensuremath{\frac{\partial{#1}}{\partial{#2}}}}
\newcommand{\Cmik}[0]{\ensuremath{C}} %mikami's cost
%$\N$ $\mathbb{N}$

%\usepackage{mathtools}

\makeatletter
\DeclareRobustCommand\widecheck[1]{{\mathpalette\@widecheck{#1}}}
\def\@widecheck#1#2{%
    \setbox\z@\hbox{\m@th$#1#2$}%
    \setbox\tw@\hbox{\m@th$#1%
       \widehat{%
          \vrule\@width\z@\@height\ht\z@
          \vrule\@height\z@\@width\wd\z@}$}%
    \dp\tw@-\ht\z@
    \@tempdima\ht\z@ \advance\@tempdima2\ht\tw@ \divide\@tempdima\thr@@
    \setbox\tw@\hbox{%
       \raise\@tempdima\hbox{\scalebox{1}[-1]{\lower\@tempdima\box
\tw@}}}%
    {\ooalign{\box\tw@ \cr \box\z@}}}
\makeatother

\newcommand{\enum}[1]{\begin{enumerate}#1\end{enumerate}}
\newcommand{\items}[1]{\begin{itemize}#1\end{itemize}}
\newcommand{\lbl}[1]{\label{#1}}
\newcommand{\coro}[1]{\begin{corollary}#1\end{corollary}}
\newcommand{\rem}[1]{\begin{remark}#1\end{remark}}
\newcommand{\refn}[1]{(\ref{#1})}

%\addtocontents{toc}{\protect\setcounter{tocdepth}{1}}

\title{Mather Measures and Ergodic Properties of Kantorovich Operators associated to General Mass Transfers}
\author{Malcolm Bowles\thanks{This is part of the PhD dissertation of this author at the University of British Columbia.} \quad   and \quad Nassif  Ghoussoub\thanks{Both authors have been partially supported by a grant from the Natural Sciences and Engineering Research Council of Canada. } 
\\ \\
{\it\small Department of Mathematics,  The University of British Columbia}\\
{\it\small Vancouver BC Canada V6T 1Z2}\\
%{\small nassif@math.ubc.ca}\vspace{1mm}
}

\date{May 07, 2019 (revised on June 20, 2019)}
%\today{Revised}
\maketitle

\begin{abstract} We introduce and study the class of {\it linear transfers} between probability distributions and {\it the dual class of Kantorovich operators} between function spaces. Linear transfers can be seen as an extension of convex lower semi-continuous energies on Wasserstein space, of cost minimizing mass transports, as well as many other couplings between probability measures to which Monge-Kantorovich theory does not readily apply. Basic examples include {\it balayage of measures},  {\it martingale transports},  {\it optimal Skorokhod embeddings}, and  the {\it weak mass transports} of Talagrand, Marton, Gozlan and others. The class also includes various stochastic mass transports such as the {\it Schr\"odinger bridge} associated to a reversible Markov process, and the Arnold-Brenier variational principle for the incompressible Euler equations.  

We associate to most linear transfers, a critical constant, a corresponding {\it effective linear transfer} and {\it additive eigenfunctions} to their dual Kantorovich operators, that extend Man\'e's critical value,  Aubry-Mather invariant tori, and Fathi's weak KAM solutions for Hamiltonian systems. This amounts to studying the asymptotic properties of the nonlinear Kantorovich operators as opposed to classical ergodic theory, which deals with linear Markov operators. This allows for the extension of Mather theory to other settings such as its stochastic counterpart. 

We also introduce {\it the class of convex transfers}, which includes $p$-powers ($ p \geq 1$) of linear transfers, 
the logarithmic entropy, the Donsker-Varadhan information, optimal mean field plans, and certain free energies as functions of two probability measures, i.e., where the reference measure is also a variable.  Duality formulae for general transfer inequalities follow in a very natural way. This paper is an expanded version of a previously posted but not published work by the authors \cite{B-G1}.
  
\end{abstract} 
\newpage
\tableofcontents

\section{Introduction} Our main objective is to study the ergodic properties of {\it Kantorovich operators,} which are at the heart of the so-called weak KAM theory developed by Mather \cite{Mat}, Fathi \cite{Fa}, Aubry \cite{Au}, Man\'e \cite{Man} and many others. Consider two compact spaces $X$ and $Y$, and the corresponding spaces $C(Y)$ (resp., $USC (X)$) 
of continuous functions on $Y$ (resp., bounded above upper semi-continuous functions on $X$).  
A {\em backward Kantorovich operator} is a map (mostly non-linear) $T^-:C(Y) \to USC (X)$ verifying the following 3 properties:
 
 \begin{enumerate}[label=\alph*)]
 \item $T^-$ is {\it monotone}, i.e., $f_1\leq f_2$ in $C(Y)$, then $T^-f_1\leq T^-f_2$. 
 \item  $T^-$ is {\it a convex operator}, that is for any $\lambda \in [0, 1]$, $f_1, f_2$ in $C(Y)$, we have 
\begin{equation*}
\T^-(\lambda f_1+(1-\lambda)f_2)\leq \lambda T^-f_1+(1-\lambda)T^-f_2. 
\end{equation*}
\item $T^-$ is {\em affine on the constants}, i.e., for any $c\in \R$ and $f\in C(Y)$,  there holds $$T^-(f+c)=T^-f +c.$$ 
 \end{enumerate}

  {\em Forward Kantorovich operators} $T^+:C(X) \to LSC (Y)$ are those that verify (a), (c), and the {\em concave counterpart} of (b), that is

\begin{equation*}
T^-(\lambda f_1+(1-\lambda)f_2)\geq \lambda T^-f_1+(1-\lambda)T^-f_2,  
\end{equation*}
 where $LSC (Y)$ is the space of bounded below lower semi-continuous functions on $Y$.\\
  We shall say that $T^-$ (resp., $T^+$) is non-trivial if there is at least one function $f\in C(Y)$ (resp., $C(X)$) such that $T^-f\not \equiv -\infty$ (resp., $T^+f\not \equiv +\infty$).

 Kantorovich operators are important extensions of Markov operators and are ubiquitous in mathematical analysis and differential equations. They appear for example as the maps that associate to an initial state of a Hamilton-Jacobi equation the solution at a given time $t$, as general value functions in dynamic programming principles (\cite{FS} Section II.3), and also in the mathematical theory of image processing \cite{A-G-L}.  But as we shall see, their rich structure stems from their duality -via Legendre transform- with certain lower semi-continuous and convex functionals $\T$ on ${\mathcal M}(X)\times {\mathcal M}(Y)$, where  ${\mathcal M}(K)$ is the space of signed measures on a compact space $K$ equipped with the weak$^*$-topology in duality with $C(K)$. 
 
Indeed,  to any  map $T^-:C(Y) \to USC (X)$ (resp., $T^+:C(X) \to LSC (Y)$), one can associate a corresponding 
convex and lower semi-continuous functional  $\T_{T^-}$ (resp., $\T_{T^+}$) on ${\mathcal M}(X)\times {\mathcal M}(Y)$ via the following --possibly infinite-- expressions:

  If $(\mu, \nu)\in {\mathcal P}(X)\times {\mathcal P}(Y)$, where ${\cal P}(K)$ denotes the space of probability measures on $K$, then set
   \begin{equation}\label{back}
{\mathcal T}_{T^-}(\mu, \nu)= 
\sup\big\{\int_{Y}g\, d\nu-\int_{X}{T^-}g\, d\mu;\,  g \in C(Y)\big\}, 
\end{equation}
(resp.,
\begin{equation}\label{fore}
{\mathcal T}_{T^+}(\mu, \nu)= 
\sup\big\{\int_{Y}{T ^+}f\, d\nu-\int_{X}f\, d\mu;\,  f \in C(X)\}), 
 \end{equation} 

If $(\mu, \nu)\notin {\mathcal P}(X)\times {\mathcal P}(Y)$, then set ${\mathcal T}_{T^-}(\mu, \nu)=+\infty$ (resp., ${\mathcal T}_{T^+}(\mu, \nu)=+\infty$). \\

Dually, we introduce the following notions.

\begin{defn} \rm Let ${\mathcal T}: {\mathcal M}(X)\times {\mathcal M}(Y)  \to \R\cup \{+\infty\}$ be a bounded below functional with  a non-empty effective 
 domain $D(\T)$. 
 \begin{enumerate}
 \item We say that  $\T$ is {\em a backward (resp., forward) linear coupling}, if
\begin{equation}
D({\T})\subset {\mathcal P}(X)\times {\mathcal P}(Y), 
\end{equation}
 and 
 \begin{equation}
 \T=\T_{T^-}  \quad \hbox{(resp., $\T={\T}_{T^+}),$} 
 \end{equation} 
 for some $T^-:C(Y) \to USC (X)$ (resp., $T^+:C(X) \to LSC (Y)$).
\item We say that $\T$ is {\em a backward (resp., forward) linear transfer}, if it is a linear coupling with $T^-$ (resp., $T^+$) being backward (resp., forward) Kantorovich operators. 
 
 \end{enumerate}
\end{defn}

It is easy to see that in either case, $\T$ is then a proper, bounded below,  lower semi-continuous and convex functional on ${\mathcal M}(X)\times {\mathcal M}(Y)$. Moreover, if we 
 consider for each $\mu \in {\mathcal M}(X)$ (resp.,  $\nu \in {\mathcal M}(Y)$) the partial maps ${\mathcal T}_\mu$ on ${\mathcal P}(Y)$ (resp., ${\mathcal T}_\nu$ on ${\mathcal P}(X)$) given by $\nu \to {\mathcal T} (\mu, \nu)$ (resp., $\mu \to {\mathcal T} (\mu, \nu)$), their  Legendre transforms are then the following functionals on $C(Y)$ (resp., $C(X)$) defined by,
\[
{\mathcal T}^*_\mu (g)=\sup\{\int_Xg d\nu -{\mathcal T}_\mu(\nu);\mu \in {\mathcal P}(X) \}=\sup\{\int_X g d\nu -{\mathcal T}(\mu, \nu); \, \mu\in {\mathcal P}(X) \}, 
\]
and 
\[
{\mathcal T}^*_\nu (f)=\sup\{\int_Xf d\mu -{\mathcal T}_\nu(\mu);\mu \in {\mathcal P}(X) \}=\sup\{\int_X f d\mu -{\mathcal T}(\mu, \nu); \, \mu\in {\mathcal P}(X) \}, 
\]
respectively, since ${\mathcal T}_\nu$ and ${\mathcal T}_\mu$   are equal to $+\infty$ whenever $\mu$ and $\nu$ are not probability measures. Note also that 
 \begin{equation}\label{Lower}
\hbox{${\mathcal T}^*_\mu (g)\leq \int_XT^- g(x) \, d\mu(x)$ \quad for any $g\in C(Y)$,}
 \end{equation}
 (resp., 
  \begin{equation}\label{upper}
 \hbox{${\mathcal T}^*_\nu (g)\leq -\int_YT^+(-f)(x) \, d\nu(x)$ \quad for any $f\in C(X)$.}
\end{equation}
We shall later prove that we have equality if and only if $\T$ is a linear transfer.

Note that if ${\cal T}$ is a backward linear coupling with an operator $T^-$, then $\tilde {\cal T}(\mu, \nu):={\cal T}(\nu, \mu)$ is a forward linear coupling with the operator ${\tilde T}^+f=-T^-(-f)$. We shall therefore focus on the properties of backward linear couplings and transfers since their ``forward counterparts" could be derived from that relation.  There are however special characteristics to those that are simultaneously forward and backward linear transfers (see Sections 3 and 6). 
We shall say that {\it a coupling ${\cal T}$ is symmetric} if 
$$\hbox{${\cal T}(\nu, \mu):={\cal T}(\nu, \mu)$ for all $\mu, \nu \in {\mathcal P}(X)$.}
$$
Note that in this case,  
$T^+f=-T^-(-f)$.

The ``partial domain" of ${\mathcal T}$ will be denoted by
$$D_1(\calT)=\{\mu \in {\mathcal P}(X); \exists \nu\in {\mathcal P}(Y), (\mu, \nu)\in D(\calT)\},$$

The following characterization of linear transfers is the starting point of our analysis. 
  \begin{thm}\label{zero.1000} Let $\T:{\mathcal M}(X)\times {\mathcal M}(Y) $ be a functional such that  $D(\calT)\subset {\mathcal P}(X)\times {\mathcal P}(Y)$ and $\{\delta_x; x\in X\}\subset D_1 ({\mathcal T})$. 
 Then, the following are equivalent:
\begin{enumerate}
\item $\T$ is a backward linear transfer.

\item  There is a map $T: C(Y) \to USC(X)$ such that for each $\mu \in D_1(\calT)$, $\T_\mu$ is convex lower semi-continuous on ${\cal P}(Y)$ and 
 \begin{equation}\label{LTb}
\hbox{${\mathcal T}^*_\mu (g)=\int_XT g(x) \, d\mu(x)$ \quad for any $g\in C(Y)$}.
\end{equation}

\item  There exists a proper bounded below  lower semi-continuous function $c: X\times {\mathcal P}(Y)\to  \R \cup\{+\infty\}$ with $\sigma \to c(x, \sigma)$ convex  such that for any $(\mu, \nu)  \in {\mathcal M}(X) \times {\mathcal M}(Y)$, 
\begin{equation}\label{weak.0}
{\mathcal T}(\mu, \nu)=\left\{ \begin{array}{llll}
\inf_\pi\{\int_X c(x, \pi_x)\, d\mu(x); \pi \in {\mathcal K}(\mu, \nu)\},    \,  &\hbox{if $\mu, \nu \in {\mathcal P}(X)\times {\mathcal P}(Y)$,}\\
+\infty \quad &\hbox{\rm otherwise.}
\end{array} \right.
\end{equation}
 where $\mK(\mu,\nu)$ is the set of probability measures $\pi$ on $X\times Y$ whose marginal on $X$ (resp. on $Y$) is $\mu$ (resp., $\nu$) {\it (i.e., the transport plans)}, and $(\pi_x)_x$ is the disintegration of $\pi$ with respect to $\mu$. 
 \end{enumerate}
\end{thm}
\noindent This characterization makes a link between linear transfers and mass transport theory, and also explain the terminology we chose.  Indeed, 
the class of linear transfers contains all cost minimizing mass transports, that is functionals on ${\mathcal P}(X)\times {\mathcal P}(Y)$ of the form,
\begin{eqnarray}
{\mathcal T}_c(\mu, \nu):=\inf\big\{\int_{X\times Y} c(x, y)) \, d\pi; \pi\in \mK(\mu,\nu)\big\},
\end{eqnarray}
where $c(x, y)$ is a continuous cost function on the product measure space $X\times Y$. A consequence of the Monge-Kantorovich theory is that cost minimizing transports ${\mathcal T}_c$ are both forward and backward linear transfers with Kantorovich operators given for any $f\in C(X)$ (resp., $g\in C(Y)$) by 
\begin{equation}
T ^+_cf(y)=\inf_{x\in X} \{c(x, y)+f(x)\} \quad {\rm and} \quad T ^-_cg(x)=\sup_{y\in Y} \{g(y)-c(x, y)\}. 
\end{equation}
However, many couplings between probability measures cannot be formulated as optimal mass transportation problems, since they do not arise as cost minimizing problems associated to functionals $c(x, y)$ that assign a price for moving one particle $x$ to another $y$. Moreover, they are 
often not symmetric, meaning that the problem imposes a specific direction from one of the marginal distributions to the other. The notion of {\it transfers} between probability measures is therefore much more encompassing than mass transportation, yet is still amenable to --at least a one-sided version-- of the duality theory of Monge-Kantorovich \cite{V}.

The notion of linear transfer is general enough to encapsulate all bounded below convex lower semi-continuous functions on Wasserstein space and Markov operators,  but also the Choquet-Mokobodzki balayage theory \cite{Ch, Moko}, the deterministic version of optimal mass transport (e.g., Villani \cite{V}, Ambrosio-Gigli-Savare \cite{A-G-S}),  their  stochastic counterparts (Mikami-Thieulin \cite{M-T}), Barton-Ghoussoub \cite{B-G} and others), optimal Skorokhod embeddings (Ghoussoub-Kim-Pallmer \cite{G-K-P2, G-K-P3}), %optimal mean field plans \cite{Sav}, 
the Schr\"odinger bridge, and the Arnold-Brenier variational descriptions of the incompressible Euler equation. Linear transfers turned out to be essentially equivalent to the notion of {\it weak mass transports} recently developed  by Gozlan et al. \cite{GL, Go4}), and motivated by earlier work of Talagrand \cite{Ta1, Ta2}, Marton \cite{Ma1, Ma2} and others.

 This paper has two objectives.  First, it introduces the unifying concepts of {\it linear and convex mass transfers} and exhibits several examples that illustrate the potential scope of this approach. The underlying idea has been implicit in many related works and should be familiar to the experts. But, as we shall see, the systematic study of these structures add clarity and understanding, allow for non-trivial extensions, and open up a whole new set of interesting problems. In other words, there are by now enough examples that share common structural features that the situation warrants the formalization of their unifying concept.  
 The ultimate purpose is to extend many of the remarkable properties enjoyed by energy functionals on Wasserstein space and standard optimal mass transportations to a larger class of couplings that is stable under addition, convex combinations,  convolutions, and tensorizations. We exhibit the basic permanence properties of the convex cones of transfers, and extend several results known for mass transports including general duality formulas for inequalities between various transfers that extend the work of Bobkov-G\"{o}tze \cite{B-G}, Gozlan-Leonard \cite{GL}, Maurey \cite{Mau} and others. 

The second objective is to show that the approach of Bernard-Buffoni \cite{B-B1, B-B2} to the Fathi-Mather weak KAM theory (\cite{Fa} \cite{Mat}), which is based on optimal mass transport associated to a cost given by a generating function of a  Lagrangian, extend to transfers and therefore applies to other couplings, including the stochastic case.
 We do that by associating to any linear transfer a corresponding {\it effective linear transfer} in the same way that weak KAM theory associates an {\it effective Lagrangian} (and Hamiltonian) to many problems of the calculus of variations \cite{Fa, Evans2}. With such a perspective, Mather theory seems to rely on the ergodic properties of the nonlinear Kantorovich operators as opposed to classical ergodic theory, which deals with linear Markov operators.

  We shall focus here on probability measures on compact spaces, even though the right settings for most applications and examples are complete metric spaces, Riemannian manifolds, or at least $\R^n$. This will allow us to avoid the usual functional analytic complications, and concentrate on the algebraic aspects of the theory. The simple compact case will at least point to results that can be expected to hold and be proved --albeit with additional analysis and suitable hypothesis -- in more general situations. In the case of $\R^n$, which is the setting for many examples stated below, the right duality is between the space $Lip (\R^n)$ of all bounded and Lipschitz functions and the space of Radon measures with finite first moment. 
 
In Section 3, we study in detail the duality between Kantorovich operators and linear transfers. We actually associate to essentially any map $T: C(Y)\to USC (X)$ (resp., any convex functional $\T$ on ${\mathcal P}(X)\times {\mathcal P}(Y)$) an "optimal" Kantorovich map $\overline T$ (resp., linear transfer ${\overline \T}$) that can be seen as ``envelopes".  

  \begin{prop} {\bf (The transfer envelope of a correlation functional)} %\label{prop.three} 
 Let  ${\mathcal T}: {\mathcal P}(X)\times {\mathcal P}(Y)\to \R \cup\{+\infty\}$ be a bounded below lower semi-continuous functional that is convex  in each of the variables   such that $\{\delta_x; x\in X\}\subset D_1 ({\mathcal T})$.
  Then, there exists a functional ${\overline \T}\geq \T$ on ${\mathcal P}(X)\times {\mathcal P}(Y)$ that is the smallest backward linear transfer above $\T$. 
 \end{prop}

Dually, we say that {\em $T^-$ is proper at $x\in X$}, if
\begin{equation}\label{proper.0} 
\hbox{$\inf\limits_{\nu \in {\cal P}(Y)}\sup\limits_{g \in C(Y)}\big\{\int_{Y}g\, d\nu-T^-g (x)\big\}<+\infty.$}
\end{equation}
This then implies that $Tf(x)>-\infty$ for every $f\in C(Y)$, and translates into the condition that the associated coupling $\T$ is {\em proper} as a convex function in the following way:  
\begin{equation}\label{proper.1}
\delta_x \in D_1(\T):=\{\mu \in {\cal P}(X);\, \exists \nu \in {\cal P}(Y), \T^-(\mu, \nu)<+\infty\}.
\end{equation}

\begin{prop} {\bf (The Kantorovich envelope of a non-linear map)} Let $T:C(Y)\to USC (X)$ be a proper map.
Then, there exists ${\overline T}:C(Y)\to USC (X)$ that is 
 the largest Kantorovich operator below $T$ on $C(Y)$.
 \end{prop}
  In anticipation to the study of the ergodic properties of a Kantorovich operators, where we will need to consider iterates of $T$, we proceed to extend in Section 4 any Kantorovich operator $T: C(Y)\to USC (X)$ to a map from $USC(Y)$ into $USC(X)$ while keeping properties (a), (b) and (c) that characterize Kantorovich operators. 
 
 In section 5, we exhibit a large number of (basic) examples of linear transfers which do not fit in standard mass transport theory.
 The various optimal {\it martingale mass transports} and {\it weak mass transports} of Marton, Gozlan and collaborators are examples of  one-directional linear transfers.  However, what motivated us to develop the concept of transfers are the stochastic mass transports, which do not minimize a given cost function between point particles, since the cost of transporting a Dirac measure to another is often infinite.

In Section 6, we show that the class of  linear transfers has remarkable permanence properties under various operations. The most important one for our study is the   stability under inf-convolution: If  ${\mathcal T}_1$ (resp., ${\mathcal T}_2$) are backward linear transfers on  ${\mathcal P}( X_1)\times {\mathcal P}(X_2)$ (resp., ${\mathcal P}( X_2)\times {\mathcal P}(X_3)$), then their {\it inf-convolution}  
 \begin{equation}
{\mathcal T}(\mu, \nu):={\mathcal T}_{1}\star{\mathcal T}_{2}(\mu, \nu)=
\inf\{{\mathcal T}_{1}(\mu, \sigma) + {\mathcal T}_{2}(\sigma, \nu);\, \sigma \in {\mathcal P}(X_2)\}
\end{equation}
is a backward linear transfer on ${\mathcal P}( X_1)\times {\mathcal P}(X_3)$.
  This leads to an even richer class of transfers, such as {\it the ballistic stochastic optimal transport}, {\it broken geodesics of transfers}, and projections onto certain subsets of Wasserstein space.   
  
  In anticipation to the extension of Mather theory, and motivated by the work of Bernard-Buffoni \cite{B-B1}, we study in Section 7 those  linear transfers that satisfy the {\em triangular inequality}, 
  \begin{equation}\label{triangle}
  \T(\mu, \nu)\leq \T(\mu, \sigma)+\T(\sigma, \nu) \quad \hbox{for all $\mu, \nu, \sigma \in {\cal P}(X)$},
  \end{equation}
  as well as the $\T$-Lipschitz functionals on the set ${\cal A}=\{\mu \in {\cal P}(X); \T(\mu, \mu)=0\}$.

 \begin{thm} Let $\T$ is a backward linear transfer on $\mcal{P}(X) \times \mcal{P}(X)$ with $T^-$ as a Kantorovich operator. Assume that $\T$ satisfies (\ref{triangle}) and that for all $\mu, \nu \in {\cal P}(X)$,  
   \begin{equation}
  \T(\mu, \nu)=\inf\{ \T(\mu, \sigma)+\T(\sigma, \nu); \, \sigma \in {\cal A}\}.
  \end{equation}
  The following then hold:
  \begin{enumerate}
  \item  A functional $\Phi$ on ${\cal A}$ is $\T$-Lipschitz if and only if there exists a  function $f\in C(X)$ such that 
   \begin{equation}
  \Phi (\mu)=\int_X f d \mu =\int_XT^-f d \mu \quad \hbox{for all $\mu \in  {\cal A}$}.
  \end{equation}
 \item  If $\T$ is also a forward transfer with $T^+$ as a Kantorovich operator, then 
   \begin{equation}
  \Phi (\mu)=\int_X f d \mu =\int_XT^-f d \mu=\int_XT^+\circ T^-f d \mu \quad \hbox{for all $\mu \in {\cal A}$}.
  \end{equation}
 \end{enumerate}
   \end{thm}
\noindent   We note that the functions $\psi_0=T^-f$ and $\psi_1=T^+\circ T^-f$, are {\em conjugate} in the sense  that $\psi_0=T^-\psi_1$ and $\psi_1=T^+\psi_0$.
   
 In Sections 8-10 we associate to any given linear transfer $\T$, a distance-like transfer $\T_\infty$, by exploiting the ergodic properties of the corresponding Kantorovich operators. For each $n\in \N$, we let ${\cal T}_n= {\cal T}\star {\cal T}\star ....\star {\cal T}$ be the transfer obtained from a backward linear transfer $\T$ by iterating its convolution $n$-times. The Kantorovich operator associated to $\T_n$ is given by the $n$-th iterate $(T^-)^n$ of the Kantorovich operator $T^-$ associated to $\T$. We will be interested in the limiting behavior of $\T_n$ and $(T^-)^n$ as $n$ goes to infinity. The following identifies a critical constant associated to a given linear transfer.

   \begin{thm}\label{One.1} Suppose $\T$ is a backward linear transfer on $\mcal{P}(X) \times \mcal{P}(X)$ and let $T:=T^-$ be its backward Kantorovich operator. Assume
    \begin{equation}\label{properly}
  \hbox{$\T(\mu_0, \mu_0) <+\infty$ for some probability measure $\mu_0$.}
  \end{equation}
   \begin{enumerate}
\item   Then, there exists a finite constant $c(\T)$  such that 
\begin{equation}
c(\T):=  \sup_{n} \frac{1}{n}\inf_{\mu,\nu \in \mcal{P}(X)}\T_n(\mu,\nu)=\inf_n\frac{1}{n}\inf_{\mu \in \mcal{P}(X)}\T_n(\mu,\mu).
\end{equation}
It will be called the ``Man\'e constant" associated to $\T$. 
\item It is also characterized by
\begin{equation}
c(\T)=\inf\limits_{\mu \in \mcal{P}(X)}\T(\mu,\mu),
\end{equation} 
and the probabilitiy distributions where the infimum is attained will be called ``Mather measures" for $\T$.
\item Moreover, $c(\T)$ is the unique constant for which there may be  $u\in C(X)$ such that 
\begin{equation}
Tu + c = u.  
\end{equation}
 Such a function $u$ will be called a {\it ``backward weak KAM solution"} for $T$. 
 \end{enumerate}
  \end{thm}
\noindent Similar definitions can be made for forward linear transfers.  Actually, when $\T$ is continuous on  $\mcal{P}(X) \times \mcal{P}(X)$ for the Wasserstein metric, much more can be said since we should be able to associate to $\T$ an {\it idempotent transfer} $\T_\infty$, i.e., one that verify $\T\star \T=\T$, in which case its corresponding Kantorovich map $T_\infty$  is idempotent for the composition operation (i.e., $T_\infty^2=T_\infty$), while its range correspond to all weak KAM solutions for $\T$.  The most known ones are the Monge optimal mass transport or more generally, the Rubinstein-Kantorovich mass transports, where the cost $c(x, y)$ is a distance on a metric space. In reality, many more examples satisfy this property, such as transfers induced by convex energies with $0$ as an infimum, the balayage transfer, and certain optimal Skorokhod embeddings in Brownian motion.  The following shows that one can associate such an idempotent transfer under equi-continuity conditions on $\T$.

  \begin{thm} \label{cont} Let $\T$ be a backward linear transfer on $\mcal{P}(X) \times \mcal{P}(X)$
  that is  continuous  for the Wasserstein metric, 
  and let $T:=T^-: C(X)\to C(X)$ be the corresponding backward Kantorovich operator.
  Then, there exist a Man\'e critical value $ c = c(\T)\in \R$ and an idempotent backward linear transfer $\T_\infty$ such that if  
  $T_\infty$ is its corresponding idempotent Kantorovich operator, then the following hold:   
 
 \begin{enumerate}

\item For every $f\in C(X)$ and $x\in X$, $\lim\limits_{n\to +\infty} \frac{T^nf(x)}{n}=-c(\T)$:
 
\item  $\T_\infty$ is the largest linear transfer envelope below $\liminf_n \T_n$ and  
 $\T_\infty= (\T - c) \star \T_\infty$;
 \item $T\circ T_\infty f + c = T_\infty f$ for all $f \in C(X)$, that is $u:=T_\infty f$ is a backward weak KAM solution.
\item The set ${\cal A}:=\{\mu \in {\cal P}(X); {\cal T}_\infty(\mu, \mu)=0\}$ is non-empty and for every 
$\mu, \nu \in {\cal P}(X)$, we have 
\begin{equation}
{\cal T}_\infty(\mu, \nu)=\inf\{ {\cal T}_\infty(\mu, \sigma)+{\cal T}_\infty(\sigma, \nu), \sigma \in {\cal A}\}, 
\end{equation}
and the infimum on ${\cal A}$ is attained. 
  \item  The Man\'e constant $c(\T)=\inf\{{\cal T}(\mu, \mu); \mu\in {\cal P}(X)\}$ is attained by a probability $\bar{\mu}$ in ${\cal A}$.
  \item If $\T$ is also a forward transfer, then similar results hold for the forward operator $T^+$. Moreover, 
   the associated effective transfer $\T_\infty$ can then be expressed as 
\begin{eqnarray}\label{KAM.duals}
{\mathcal T}_\infty(\mu, \nu)=\sup\big\{\int_{X}f^+\, d\nu-\int_{X}f^-\, d\mu;\, (f^-, f^+)\in {\cal I}\big\},
\end{eqnarray}
where 
\begin{align*}
{\cal I}=\big\{(f^-, f^+); &\hbox{ $f^-$ (resp., $f^+$) is a backward (resp., forward) weak KAM solution}\\
& \quad \hbox{and $\int_Xf^-d\mu=\int_Xf^+d\mu $ for all  $\mu\in {\cal A}$\big\}.}
\end{align*}
\end{enumerate}
\end{thm}
  By analogy with the weak KAM theory of Mather-Aubry-Fathi --briefly described in the next paragraph-- we shall say that $\T_\infty$ (resp., $T^\infty$) is {\em the effective transfer} or the {\em generalized Peierls barrier} (resp., {\em effective Kantorovich operator}) associated to $\T$.     The set ${\mathcal A}$ is the analogue of the {\em projected Aubry set}, and 
   \[
   \mathcal{D}:= \{ (\mu,\nu) \in \mathcal{P}(X)\times\mathcal{P}(X)\,:\, \T(\mu,\nu) + \T_\infty(\nu,\mu) = c(\T)\}
   \]
  can be seen as a {\em generalized Aubry set} \cite{Fa}.

As mentioned above, the effective transfer ${\mathcal T}_\infty$ is obtained by an infinite  inf-convolution process, while $T_\infty$ is obtained by an infinite iteration procedure, which lead to fixed points (additive eigenfunctions) for such a non-linear operator. The same procedure actually applies for any semi-group of backward linear transfers (for the convolution operation) and the corresponding semi-group of Kantorovich maps (for the composition operation). This will be established in Section 7 for an equicontinuous semi-group of backward linear transfers.

In Section 9, we deal with the case of a general linear transfer, where we do not assume continuity of $\T$, but that the corresponding Kantorovich operator $T$ maps $C(X)$ to $USC (X)$.  
We then consider the following measure of the oscillation of the iterates of $\T$:
\begin{equation}
K(n):=\inf\limits_{\mu \in {\cal P}(X)}\T_n (\mu, \mu)-\inf\limits_{\mu, \nu \in {\cal P}(X)}\T_n (\mu, \nu).
\end{equation} 
Note that Theorem \ref{One.1} already asserts that $\frac{K(n)}{n}$ decreases to zero, but we shall need a slightly stronger condition to prove in section 8 the existence of weak KAM solutions.  
 
  \begin{thm}\label{One} Let $\T$ be a backward linear transfer on $\mcal{P}(X) \times \mcal{P}(X)$ such that its corresponding Kantorovich operator maps $C(X)$ to $USC(X)$. Assume (\ref{properly}) and the following two conditions:  
  \begin{equation}\label{T1}
   \sup_{x\in X}\inf_{\sigma \in {\mathcal P}(X)} \T(x, \sigma) <+\infty, 
   \end{equation}
   and 
     \begin{equation}\label{stronger.0}
     \liminf_n K(n) <+\infty.
     \end{equation}
     \begin{enumerate}
     \item Then, there exists a backward weak KAM solution  for $\T$ at the level $c:=c(\T)$.
     \item The Man\'e constant $c$ is unique in the following sense
     \begin{align}
c(\T)&=\sup \{d\in \R; \hbox{\rm there exists $u\in USC(X)$ with $Tu+d  \leq u$}\} \\
&=\inf \{d\in \R; \hbox{\rm there exists $v\in USC(X)$ with $Tv+d  \geq v$}\}. \nonumber
\end{align}
\end{enumerate}
  \end{thm}
 Note that  (\ref{T1}) merely states that the function $T1$ is bounded below, while (\ref{properly}) yields that $T 1$ is not identically $-\infty$. 
  This will allow us to prove the following.
 
 \begin{thm}  
 Let $\T$ be a backward linear transfer that is also bounded above on ${\cal P}(X)\times {\cal P}(X)$, then 
 \begin{equation}
\frac{\T_n(\mu, \nu)}{n} \to c \quad \hbox{uniformly on ${\cal P}(X)\times {\cal P}(X)$.}
\end{equation}
Moreover, there exists an idempotent  operator $T_\infty: C(X)\to USC (X)$ such that for each $f\in C(X)$,  $T_\infty f$ is a backward weak KAM solution for $\T$. 
\end{thm} 
In Section 10, we use a regularization procedure to show that many of the conclusions in Theorem \ref{cont} can hold  for transfers that are neither necessarily continuous nor bounded. This holds for example when the following condition is satisfied.
\begin{equation}
\inf_{\mu \in \mcal{P}(X)}\T(\mu,\mu)=\inf_{\mu, \nu \in \mcal{P}(X)}\T(\mu,\nu),
\end{equation} 
which holds in many situations. This will allow us to prove the following general result.
\begin{thm}
Let $\T$ be a backward linear transfer on $\mcal{P}(X)\times \mcal{P}(X)$, where $X$ is a bounded domain in $\R^n$. Then, for every $\lambda \in (0, 1)$, there exists a convex function $\phi$, a constant $c\in \R$ and a function $g\in USC(X)$
% such that the linear transfer given by
%\begin{equation}
%\tilde{\T}_\lambda (\mu,\nu) := \T(\mu, \lambda (\nabla \phi_\lambda)_\#\mu + (1-\lambda)\nu)
%\end{equation}
%is such that 
%\[
%%\inf\limits_{\mu. \nu \in {\cal P}(X)} {\tilde \T}_\lambda (\mu, \nu)=\inf\limits_{\mu \in {\cal P}(X)}{\tilde \T}_\lambda(\mu, \mu).
%\] 
%In particular, ${\tilde \T}_\lambda$ admits weak KAM solutions, that is there exists 
%$g\in USC(X)$ and  
such that 
\begin{equation}
T^-g   +c= \lambda \, g (\nabla \phi) +(1-\lambda)\, g.
\end{equation}
\end{thm}
\noindent Note that if $\phi$ is the quadratic function, then $g$ is a weak KAM solution for $T$.
 
To make the connection with Mather-Aubry-Fathi theory, consider $\T_t$ to be the cost minimizing transport 
\begin{equation}
\T_t(\mu,\nu) = \inf\{\int_{M\times M} c_t(x,y)\d\pi(x,y)\,;\, \pi \in \mcal{K}(\mu,\nu)\},
\end{equation}  
where \begin{equation}
c_t(x,y) := \inf\{\int_{0}^{t}L(\gamma(s), \dot{\gamma}(s))\d s\,;\, \gamma \in C^1([0,t];M), \gamma(0)=x, \gamma(t)=y\},
\end{equation} 
for some given (time-independent) \textit{Tonelli Lagrangian} $L$ possessing suitable regularity properties on a compact state space $M$.  
The backward Kantorovich operators associated to $\T_t$ are nothing but the Lax-Oleinik semi-group $S_t^-$, $t > 0$, defined as
\begin{equation}
S_t^- u(x) := \inf\{ u(\gamma(0)) + \int_{0}^{t}L(\gamma(s), \dot{\gamma}(s))\d s\,;\, \gamma \in C^1([0,t]; M), \gamma(t) = x\}.
\end{equation}
Recall from \cite{Fa} that a function $u \in C(M)$ is said to be a \textit{negative weak KAM solution} if for some $c\in \R$, we have 
\begin{equation}S_t^-u + ct = u \quad \hbox{for all $t \geq 0$},
\end{equation}
these solutions are then given by any function in the range of the effective Kantorovich map associated to $(S_t^-)_t$. 
Actually, these solutions were obtained this way by Bernard and Buffoni \cite{B-B1, B-B2}, who capitalized on the fact that in this case, the transfers $(\T_t)_t$ are actually given by optimal mass transports associated to the cost $c_t$, and that the Lax-Oleinik semi-groups are obtained via Monge-Kantorovich theory. 
 This general asymptotic theory applies to both the linear setting such as the heat semi-group and to non-linear contexts including the Schr\"odinger bridge.
  It also applies to settings where transfers are neither given by optimal transport problems nor are they not continuous on Wasserstein space. 

In section 11, we apply the general theory to the following semi-group of stochastic optimal mass transports:
Let $(\Omega, \mcal{F}, \P)$ be a complete probability space with normal filtration $\{\mcal{F}_t\}_{t \geq 0}$, and define $\mcal{A}_{[0,t]}$ to be the set of continuous semi-martingales $X: \Omega \times [0,t] \to M$ such that there exists a Borel measurable drift $\beta: [0,t] \times C([0,t]) \to \R^d$ for which
\enum{
\item $\omega \mapsto \beta (s,\omega)$ is $\mcal{B}(C([0,s]))_{+}$-measurable for all $s \in [0,t]$, where $\mcal{B}(C([0,s]))$ is the Borel $\sigma$-algbera of $C[0,s]$.
\item $W (s) := X(s) - X(0) - \int_{0}^{s}\beta (s')\d s'$ is a $\sigma(X(s)\,;\, 0 \leq s \leq t)$  
is an $M$-valued Brownian motion. 
}
For each $\beta$, we shall denote the corresponding $X$ by $X^\beta$ in such a way that 
\begin{equation}
d X^\beta(t)=\beta (t) dt + d W(t).
\end{equation}
 The stochastic transport from $\mu \in \mcal{P}(M)$ to $\nu\in \mcal{P}(M)$ on the interval $[0,t]$, $t > 0$, is then defined as 
\begin{equation}
\T_{t}(\mu,\nu) := \inf\lf\{\E \int_{0}^{t} L(X^\beta (s), \beta (s))\d s\,;\, X^\beta (0) \sim \mu, X^\beta (t) \sim \nu, X^\beta \in \mcal{A}_{[0,t]}\rt\}.
\end{equation}
Note that these couplings do not fit in the Monge-Kantorovich framework as they are not optimal mass transportations that correspond to a cost function between two states, but they are  backward linear transfers according to our definition thanks to the work of Mikami-Tieullin \cite{M-T}.  In this case, they only have backward Kantorovich operators given by the stochastic Lax-Oleinik operator,
\begin{equation}
S_{t} f(x) := \sup_{X \in \mcal{A}_{[0,t]}}\lf\{\E\lf[\lf(f(X(t)) - \int_{0}^{t} L(X(s),\beta_X(s,X))\d s\rt)|X(0) = x\rt]\rt\}, 
\end{equation}
in such a way that 
\begin{equation}
{\mathcal T}_t(\mu,  \nu)=\sup\big\{\int_{M}u(y)\, d\nu(y)-\int_{M}S_tu(x)\, d\mu(x); u\in C(M)\big\}. 
\end{equation}
In addition, for each end-time $T>0$, $ u(t, x)=S_{T-t}u(x)$ is a viscosity solution to the following backward Hamilton-Jacobi-Bellman equation 
\begin{align}
\begin{cases}\frac{\partial u}{\partial t}(t,x) + \frac{1}{2}\Delta u (t,x) + H(x,\nabla u(t,x)) &= 0,\quad \hbox{on $[0,T)\times M$} \lbl{HJB-time}\\
\hfill u(T,x) &= u(x) \quad \hbox{on $M$}.  
\end{cases}
\end{align}
The existing of corresponding {\it stochastic weak KAM solutions} (i.e., fixed points for $u\to S_tu+ct$) will then be viscosity solutions of second order stationary Hamilton-Jacobi-Bellman equation  
\begin{equation}\label{HJB-stat0}
\frac{1}{2}\Delta u(x) + H(x,\nabla u (x)) = c,\quad x \in M. 
\end{equation}
We shall consider the case of a torus, already studied by Gomez \cite{Gom}, and capitalize on his work to show that   just like in the deterministic case, the Man\'e constant $c$,  for which there exists a backward weak KAM solution is unique and is connected to a stochastic analogue of Mather's problem via
\begin{equation}\label{mather0}
c =\inf\{\T_1(\mu, \mu); \mu \in {\cal P}(M)\}= \inf\{\int_{TM} L(x,v)\d m(x,v); m \in {\cal N}_0(TM)\},
\end{equation}
where ${\cal N}_0(TM)$ is the set of probability measures $m$ on phase space that verify for every $\vphi \in C^{1,2}([0,1]\times M)$, 
\begin{equation}
\int_{[0,1]}\int_{TM}\lf[\partial_t \vphi(x,t) + v\cdot \nabla \vphi(x,t) + \frac{1}{2}\Delta \vphi(x,t)\rt]d m(x,v)\d t = \int_{TM}[\vphi(x,1) - \vphi(x,0)]\d m(x,v).
\end{equation}
 The {\em stochastic Mather measures} are those that are minimizing Problem (\ref{mather0}).

 In section 12, we introduce  
a natural and  richer family of transfers: the class of {\it convex transfers.}  
\begin{defn} \rm 
A proper convex and weak$^*$ lower semi-continuous functional ${\mathcal T}: {\mathcal M}(X)\times {\mathcal M}(Y) \to \R\cup \{+\infty\}$ is said to be a {\it backward convex coupling} (resp., {\it forward convex coupling}), if there exists a family of maps $T^-_i: C(Y)\to USC (X)$ (resp., $T^+_i: C(X)\to LSC(Y)$) such that:\\ %for $g\in C(Y)$ (resp., $f\in C(X)$)
If $(\mu, \nu)\in {\mathcal P}(X)\times {\mathcal P}(Y)$, then 
 \begin{equation}
{\mathcal T}(\mu, \nu)=\sup\big\{\int_{Y}g(y)\, d\nu(y)-\int_{X}{T_i^-}g(x)\, d\mu(x);\,  g \in C(Y), i\in I\big\},
\end{equation}
(resp.,
  \begin{equation}
{\mathcal T}(\mu, \nu)= 
\sup\big\{\int_{Y}{T_i^+}f(y)\, d\nu(y)-\int_{X}f(x)\, d\mu(x);\,  f \in C(X), i\in I \big\}, 
\end{equation}
If $(\mu, \nu)\notin {\mathcal P}(X)\times {\mathcal P}(Y)$, then ${\mathcal T}(\mu, \nu)+\infty$. 
 \end{defn}
\noindent In other words,  
 \begin{equation}\label{supT}
{\mathcal T}(\mu, \nu)=\sup_{i\in I}{\mathcal T}_i(\mu, \nu), 
\end{equation}
where each ${\mathcal T}_i$ is a linear transfer on ${\mathcal P}(X)\times {\mathcal P}(Y)$ induced by each $T_i^-$ (resp., $T_i^+$).  
Note that we do not assume in general that each $T_i^-$ (resp., $T_i^+$) is a Kantorovich operator. 
 Typical examples are $p$-powers (for $ p \geq 1$) of a linear transfer, which will then be a convex couplings in the same direction.   More generally, for any convex increasing real function $\gamma$ on $\R^+$ and any linear backward (resp., forward) transfer, the map $\gamma ({\cal T})$ is a backward (resp., forward) convex coupling.  Actually,  in this case, each of the associated $\T_i$ can be taken to be a linear transfer.

Note that a convex coupling $\T$ of the form (\ref{supT}) only implies that for $g\in C(Y)$ (resp., $f\in C(X)$), 
\begin{equation}
\hbox{${\mathcal T}_\mu^*(g)\leq \inf\limits_{i\in I}\int_XT_i ^-g(x) \, d\mu(x)$\quad  and \quad ${\mathcal T}_\nu^*(f)\leq \inf\limits_{i\in I}\int_Y {-T_i ^+}(-f)(y) \, d\nu(y)$}.
\end{equation}
We therefore introduce the following stronger notion. 
\begin{defn} \rm Say that $\T$ is a backward convex transfer (resp.,  forward  convex transfer) if for $g\in C(Y)$ (resp., $f\in C(X)$), 
 \begin{equation}
\hbox{${\mathcal T}^*_\mu (g)=\inf\limits_{i\in I}\int_XT_i ^-g(x) \, d\mu(x)$} \quad  \hbox{(resp.,  ${\mathcal T}^*_\nu (f)=\inf\limits_{i\in I}\int_Y {-T_i ^+}(-f)(y) \, d\nu(y)$)}. 
\end{equation}
\end{defn}
\noindent Again, the $T_i's$ are not necessarily Kantorovich maps, i.e., they don't correspond to Legendre transforms of  linear transfers $\T_i's$, however, the map $g\to  \inf\limits_{i\in I}\int_XT_i ^-g(x) \, d\mu(x)$ does in this case possess the properties of a Legendre transform. 
We give an example in Section 12 of  a convex coupling that is not a convex transfer. 

 Typical examples of convex backward transfers include {\it generalized entropies} of the following form, but as a function of both measures, i.e., including the reference measure, 
 \begin{equation}
 {\cal T}(\mu, \nu)=\int_X \alpha (\frac{d\nu}{d\mu})\,  d\mu, \quad \hbox{if $\nu<<\mu$ and $+\infty$ otherwise,}
 \end{equation}
whenever $\alpha$ is a strictly convex lower semi-continuous superlinear real-valued function on $\R^+$. 

The {\em Donsker-Varadhan information} is defined 
as 
\begin{equation}
{\cal I}(\mu, \nu):=\begin{cases}\EE(\sqrt{f}, \sqrt{f}), \ \ &\text{ if }\ \mu=f\nu, \sqrt{f}\in\dd(\EE)\\
+\infty, &\text{ otherwise,}
\end{cases}
\end{equation}
where $\EE$ is a Dirichlet form  with domain $\dd(\EE)$ on
$L^2(\nu)$. It  is another example of a backward completely convex transfer, since it can also be written as
\begin{equation}
{\cal I}(\mu, \nu)=\sup\{\int_X f\, d\nu-\log \|P_1^f\|_{L^2(\mu)};\,  f\in C(X)\}, 
\end{equation}
 where $P_t^f$ is an associated (Feynman-Kac) semi-group of operators on $L^2(\mu)$.

 The important example of the {\it logarithmic entropy} 
 \begin{equation}
 {\cal H}(\mu, \nu)=\int_X \log (\frac{d\nu}{d\mu})\,  d\nu, \quad \hbox{if $\nu<<\mu$ and $+\infty$ otherwise,}
\end{equation}
is of course one of them, but it is much more as we  now focus on  a remarkable subset of the class of convex transfers: the class of {\it entropic transfers}, defined as follows:

\begin{defn} \rm  Let $\alpha$ (resp., $\beta$) be a convex increasing (resp., concave increasing) real function on $\R$, and let ${\mathcal T}: {\mathcal P}(X)\times {\mathcal P}(Y) \to \R\cup \{+\infty\}$ be a proper (jointly) convex and weak$^*$ lower semi-continuous functional. We say that 
\begin{itemize}
\item ${\mathcal T}$ is a {\it $\beta$-entropic backward transfer}, if there exists a map $T ^-: C(Y) \to USC (X)$ such that for each $\mu \in D_1(\calT)$, the Legendre transform of ${\mathcal T}_\mu$ on ${\mathcal M}(Y)$ satisfies:
\begin{equation}\hbox{${\mathcal T}^*_\mu (g)=\beta \left(\int_XT ^-g(x) \, d\mu(x)\right)$ \quad for any $g\in C(Y)$}.
\end{equation}

\item $\calT$ is an {\it $\alpha$-entropic forward transfer}, if there exists a map $T ^+: C(X)\to LSC (Y)$ such that for each $\nu \in D_2(\calT)$, the Legendre transform of  ${\mathcal T}_\nu$ on ${\mathcal M}(X)$ satisfies: 
\begin{equation}\hbox{${\mathcal T}^*_\nu (f)=-\alpha \left(\int_Y {T ^+}(-f)(y) \, d\nu(y)\right)$ \quad for any $f\in C(X)$. }
\end{equation} 
  \end{itemize}
 \end{defn}
\noindent So, if ${\mathcal T}$ is an $\alpha$-entropic forward transfer on $X\times Y$, then for any probability measures $(\mu, \nu) \in D(\T)$, % 
we have
 \begin{equation}
{\mathcal T}(\mu, \nu)= 
\sup\big\{\alpha\left(\int_{Y}{T ^+}f(y)\, d\nu(y)\right)-\int_{X}f(x)\, d\mu(x);\,  f \in C(X) \big\},
\end{equation}
while if ${\mathcal T}$ is a $\beta$-entropic backward transfer, then 
 \begin{equation}
{\mathcal T}(\mu, \nu)=\sup\big\{\int_{Y}g(y)\, d\nu(y)-\beta \left(\int_{X}{T ^-}g(x)\, d\mu(x)\right);\,  g \in C(Y)\big\}.
\end{equation}
Again, the associated maps $T^-$ and $T^+$ are not necessarily Kantorovich maps, 
 however, the map $g\to \beta \left(\int_XT ^-g(x)\, d\mu(x)\right)$ and $f\to \alpha \left(\int_XT ^+f(x)\, d\nu(x)\right)$ inherit special (convexity and lower semi-continuity) properties from the fact that they are Legendre transforms. 

We observe in Section 12 that entropic transfers are completely convex transfers. A typical example is  of course the logarithmic entropy, since it can be written as 
\begin{equation}
 {\cal H}(\mu, \nu)=\sup\{\int_X f\, d\nu-\log(\int_Xe^{f}\, d\mu);\,  f\in C(X)\}, 
\end{equation}
making it a $\log$-entropic backward transfer. 
 More examples of $\alpha$-entropic forward transfers and $\beta$-entropic backward transfers can be obtained by convolving entropic transfers  with linear transfers of the same direction. 
 
 In section 13, we show how the concepts of linear and convex transfers lead naturally to more transparent proofs and vast extensions of many well known duality formulae for transport-entropy inequalities, such as Maurey-type inequalities of the following type \cite{Mau}: \\
 Given linear transfers $\T_1, \T_2$, entropic transfers ${\mathcal H}_1, {\mathcal H}_2$ and a convex transfer ${\cal F}$,  find a reference pair $(\mu, \nu)\in {\cal P}(X_1)\times {\cal P}(X_2)$ such that 
\begin{equation}
{\mathcal F}(\sigma_1, \sigma_2) \leq \lambda_1 {\mathcal T}_1\star {\mathcal H}_1(\sigma_1, \mu)+\lambda_2 {\mathcal T}_2\star {\mathcal H}_2( \sigma_2, \nu) \quad \hbox{for all $(\sigma_1, \sigma_2)\in {\mathcal P}(X_1) \times {\mathcal P}(X_2)$.}
\end{equation}
This is then equivalent to the non-negativity of an expression of the form $
\tilde {\mathcal E}_1\star (-{\mathcal T})\star  {\mathcal E}_2,$ which could be obtained from the following dual formula:
\begin{equation}
\tilde {\mathcal E}_1\star (-{\mathcal F}) \star  {\mathcal E}_2 \, (\mu, \nu)=\inf\limits_{i\in I} \inf\limits_{f\in C(X_3)} \left\{\alpha_1 \big(\int_{X_1} E_1^+\circ F_{i}^-f\, d\mu)+ \alpha_2 (\int_{X_2}E_2^+(f)\, d\nu)\right\}, 
\end{equation}
where 
${\mathcal F}$ is a convex backward  transfer on $Y_1\times Y_2$ with Kantorovich family $(F^-_i)_{i\in I}$, ${\mathcal E}_1$ (resp., ${\mathcal E}_2$) is a forward $\alpha_1$-transfer on $Y_1\times X_1$ (resp., a  forward $\alpha_2$-transfer on $Y_2\times X_2$) with Kantorovich operator  $E_1^+$ (resp., $E_2^+$).

 \section{First examples of linear mass transfers} 

The class of linear transfers is quite large and ubiquitous in analysis. 

\subsection{Convex energies on Wasserstein space are linear transfers}

The class of linear transfers is a natural extension of the convex energies on Wasserstein space. \\

\noindent {\bf Example 2.1: Convex energies}

If $I: {\mathcal P}(Y) \to \R$ is a  bounded below convex weak$^*$-lower semi-continuous functions  on ${\mathcal P}(Y)$. One can then associate a  backward linear transfer 
\begin{equation}
\T(\mu, \nu)=I(\nu)\quad  \hbox {for all  $(\mu, \nu) \in {\mathcal P}(X) \times {\mathcal P}(Y),$}
\end{equation} 
 in such a way that the corresponding Kantorovich map is $T^-: C(Y) \to \R \subset C(X)$ is $T^-f(x)=I^*(f)$ for every $x\in X$.

For example, if $I$ is the linear functional $I(\nu)=\int_YV(y)\, d\nu (y)$, where $V$ is a lower semi-continuous potential on $Y$, then for every $x\in X$, 
$$T^-f(x)=\sup_{y\in Y} (f(y)-V(y)).$$

If $I$ is the relative entropy with respect to Lebesgue measure, that is $I(\nu)=\int_Y \log \frac{d\nu}{dy} dy$ when $\nu$ is absolutely continuous with respect to Lebesgue measure and $+\infty$ otherwise, then it induces a linear transfer with backward Kantorovich map being for all $x$,
$$T^-f(x)=\log \int_Ye^f\, dy.$$ 

The same holds for the variance functional $I(\nu):=-{\rm var} (\nu):=|\int_Yy\ d\nu|^2-\int_Y|y|^2\, d\nu(y)$, where the associated Kantorovich map is given by 
\[
T ^-f(x)=\sup\{ {\widehat {f+q}}(z)-|z|^2; z\in Y\},
\]
 where $q$ is the quadratic function $q(x)=\frac{1}{2}|x|^2$ and $\hat g$ is the concave envelope of the function $g$. See (\ref{var}) below. 

\subsection{Mass transfers with positively homogenous Kantorovich operators}

    To any Markov operator, i.e., bounded linear positive operator $T: C(Y) \to C(X)$ such that $T1=1$, one can associate a backward linear transfer in the following way:
\begin{equation}
{\mathcal T}_T(\mu, \nu)=\left\{ \begin{array}{llll}
0 \quad &\hbox{if $T^*(\mu)=\nu $}\\
+\infty \quad &\hbox{\rm otherwise,}
\end{array} \right.
\end{equation}
 where $T^*:{\mathcal M}(X) \to {\mathcal M}(Y)$ is the adjoint operator. It is then easy to see that $T^-=T$ is the corresponding backward Kantorovich map. If now $\pi_x= T^*(\delta_x)$, then one can easily see that $T^-f(x)=\int_Yf(y)d\, \pi_x(y)$ and that 
$${\mathcal T}_T(\mu, \nu)=0 \quad \hbox{if and only if \quad $\nu (B)=\int_X\pi_x(B)\, \, d\mu (x)$ for any Borel $B\subset Y$. }
$$
Conversely, any probability measure $\pi$ on $X\times Y$ induces a forward and backward linear transfer in the following way:
\begin{equation} \label{plan}
{\mathcal I}_\pi (\mu, \nu)=\left\{ \begin{array}{llll}
0 \quad &\hbox{if $\mu=\pi_1$ and $\nu=\pi_2.$}\\
+\infty \quad &\hbox{\rm otherwise,}
\end{array} \right.
\end{equation}
where $\pi_1$ (resp., $\pi_2$) is the first (resp., second) marginal of $\pi$. In this case,
\begin{equation}
T^-f(x)=\int_Yf(y)d\, \pi_x(y) \quad \hbox{and \quad $T^+f(y)=\int_Xf(x)d\, \pi_y(x)$,}
\end{equation}
where $(\pi_x)_x$ (resp., $(\pi_y)_y$) is the disintegration of $\pi$ with respect to $\pi_1$ (resp.,  $\pi_2$). Note however, that we don't necessarily have here that $x\to \pi_x$ is weak$^*$-continuous, that is T maps $L^1(Y, \pi_2) \to L^1(X, \pi_1)$ and not necessarily $C(Y)$ to $C(X)$.\\

%Needless to say in all these examples, the recession operator coincides with the Kantorovich operator. \\

\noindent {\bf Example 2.2: The prescribed push-forward transfer}
    
 If $\sigma$ is a continuous map from $X$ to $Y$, then 
   \begin{equation}
{\mathcal I}_\sigma (\mu, \nu)=\left\{ \begin{array}{llll}
0 \quad &\hbox{if $\sigma_\#\mu=\nu $}\\
+\infty \quad &\hbox{\rm otherwise,}
\end{array} \right.
\end{equation}
is a backward linear transfer with Kantorovich operator given by $T^-f=f\circ \sigma$. \\

{\it The identity transfer} corresponds to when $X=Y$ and $\sigma (x)=x$, in which case the corresponding Kantorovich operators are the identity map, that is $T^+f=T^-f=f$.\\

 \noindent {\bf Example 2.3: The prescribed Balayage transfer}
 
  Given a convex cone of continuous functions ${\mathcal A} \subset C(X)$, where $X$ is a compact space, one can define an order relation between probability measures $\mu, \nu$ on $X$, called the  ${\mathcal A}$-balayage, in the following way. 
 \[
 \mu \prec_{\mathcal A} \nu \quad \hbox{ if and only if \quad $\int_X\phi \, d\mu \leq \int_X\phi \, d\nu$ for all $\phi$ in  ${\mathcal A}$.}
 \]
 Suppose now that $T:C(X)\to C(X)$ is a Markov operator such that  $\delta_x \prec_{\mathcal A} \pi_x:=T^*(\delta_x)$ for all $x\in X$, we will then call it -- as well as its associated transfer ${\mathcal T}_T$ -- an ${\cal A}$-dilation.  Similarly, a probability measure $\pi$ on $X\times X$ is {\it an ${\cal A}$-dilation} if $\delta_x \prec_{\cal A} \pi_x$, where $(\pi_x)_x$ is the disintegration of $\pi$ with respect to its first marginal $\pi_1$. 
To each ${\cal A}$-dilation $\pi$, one can define a backward linear transfer as above.\\

 \noindent {\bf Example 2.4: The prescribed Skorokhod transfer}
 
  Writing $Z\sim \rho$ if $Z$ is a random variable with distribution $\rho$, and letting $(B_t)_t$ denote Brownian motion, and ${\mathcal S}$  the corresponding class of --possibly randomized-- stopping times.  
   For a fixed $\tau\in {\mathcal S}$, one can associate a backward linear transfer in the following way:
  \begin{equation}
{\mathcal T}_{\tau}(\mu, \nu)=\left\{ \begin{array}{llll}
0 \quad &\hbox{if $B_0\sim \mu$ and $B_\tau \sim\nu$.}\\
+\infty \quad &\hbox{\rm otherwise.}
\end{array} \right.
\end{equation}  
Its backward Kantorovich operator is then $T^-f(x)=\mathbb{E}^{x}[f(B_\tau)]$, where the expectation is with respect to Brownian motion satisfying $B_0=x$.

\subsection{Optimal linear transfers with zero cost}

Let ${\mathcal C}$ be a class of positive bounded linear operators $T$  from $C(Y)\to C(X)$ such that $T1=1$. We can then consider the following correlation,
  \begin{equation}
{\mathcal T}_{\mathcal C}(\mu, \nu)=\left\{ \begin{array}{llll}
0 \quad &\hbox{if there exists $T\in {\mathcal C}$ with $T^*(\mu)=\nu$}\\
+\infty \quad &\hbox{\rm otherwise.}
\end{array} \right.
\end{equation} 
In other words, 
 \begin{equation}
{\mathcal T}_{\mathcal C}(\mu, \nu)=\inf\{{\mathcal T}_T (\mu, \nu); T\in {\mathcal C}\}.
 \end{equation}
We now give a few interesting examples, where ${\mathcal T}_{\mathcal C}$ is again a  linear mass transfer.\\

 \noindent {\bf Example 2.5: The null transfer} 
  
  This is simply the map ${\mathcal N}(\mu, \nu)=0$ for all probability measures $\mu$ on $X$ and $\nu$ on $Y$. 
  It is easy to see that it is both a backward and forward linear transfer with Kantorovich operators, 
 \begin{equation}
 \hbox{$T^-f\equiv \sup_{y\in Y}f(y)$ \quad and \quad $T^+f\equiv\inf_{x\in X}f(x)$.}
 \end{equation}
 Note that 
 \begin{eqnarray*}
 {\mathcal N}(\mu, \nu)&=&\inf\{{\mathcal T}_\sigma (\mu, \nu); \sigma \in C(X;Y)\}\\
 &=&\inf\{{\mathcal T}_\pi (\mu, \nu); \pi \hbox{ is a transfer plan on $X\times Y$}\}\\
  &=&\inf\{{\mathcal T}_T (\mu, \nu); T \in C(Y. X), \hbox{$T$ positive and $T1=1$}\}, 
 \end{eqnarray*}
 where ${\mathcal I}_\sigma$ and ${\mathcal I}_\pi$ are the push-forward transfers defined in Example 2.2. This is a particular case, i.e., when the cost is trivial,  of a relaxation result of Kantorovich (e.g., see Villani \cite{V}).\\

 \noindent {\bf Example 3.6: The Balayage transfer} 
  
 Let ${\mathcal A}$ be a proper closed convex cone in $C(X)$, and define now {\em the balayage transfer} ${\cal B}$ on ${\cal P}(X)\times {\cal P}(X)$ via
 \begin{equation}
{\mathcal B}(\mu, \nu)=\left\{ \begin{array}{llll}
0 \quad &\hbox{if $\mu \prec_{\mathcal A} \nu$}\\
+\infty \quad &\hbox{\rm otherwise.}
\end{array} \right.
\end{equation} 
A generalized version of a Theorem of Strassen \cite{St} yields the following relations:

\begin{prop} Assume the cone ${\mathcal A}$ separates the points of $X$ and that it is stable under finite suprema. Then, for any two probability measures $\mu, \nu$ on $X$, the following are equivalent:
\begin{enumerate}
\item $\mu \prec_{\mathcal A} \nu$.
\item There exists an $\cal A$-dilation $\pi$ on $X\times X$ such that $\mu=\pi_1$ and $\nu=\pi_2$.
\end{enumerate}
\end{prop}
\noindent From this follows that 
\begin{equation}
{\mathcal B}(\mu, \nu)=\inf\{{\mathcal B}_\pi(\mu, \nu);  \hbox{$\pi$ is an $\cal A$-dilation}\}.
\end{equation}
Moreover, a generalization of Choquet theory developed by Mokobodoski and others \cite{Moko} yields that for every $\mu \in {\mathcal P}(X)$, we have 
\[
\sup \{\int_X f\ d\sigma;\, \mu \prec_{\mathcal A}\sigma \}=\int_X{\hat f}\, d\mu, 
\]
where   
\[
{\hat f}(x)=\inf \{ g(x); g\in -{\mathcal A},\,  g\geq f \, \hbox{on $X$}\}=\sup\{ \int_Xf d\sigma;\, \epsilon_x  \prec_{\mathcal A} \sigma\}.
\]
It follows that 
$
{\mathcal B}_\mu^*(f)=\int_X{\hat f}\, d\mu$,  which means that ${\mathcal B}$ is a backward linear transfer whose Kantorovich operator is $T^-f={\hat f}$.

${\mathcal B}$ is also a forward linear transfer 
with a forward Kantorovich operator is $T^+f=\check f$, where 
\[
{\check f}(x)=\sup \{ h(x); h\in {\mathcal A},\,  h \leq f \, \hbox{on $X$}\}=\inf\{ \int_Xf d\sigma;\, \epsilon_x  \prec_{\mathcal A} \sigma\}.
\]
%Again, the respective recession operators coincide with the Kantorovich operators.
\begin{itemize}
\item A typical example is when $X$ is a convex compact space in a locally convex topological vector space and ${\mathcal A}$ is the cone of continuous convex functions. In this case, $T^-f=\hat f$ (resp., $T^+f=\check f$)  is the concave (resp., convex) envelope of $f$, and which was the context of the original Choquet theory. 

\item If $X$ is a bounded subset of a normed space $(E, \|\cdot\|)$, then ${\mathcal A}$ can be taken to be the cone of all norm-Lipschitz convex functions.  

\item If $X$ is an interval of the real line, then one can consider ${\mathcal A}$ to be the cone of increasing functions.

\item If $X$ is a pseudo-convex domain of $\C^n$, then one can take ${\mathcal A}$ to be the cone of Lipschitz plurisubharmonic functions (see \cite{GM}). In this case, if $\phi$ is a Lipschitz function, then the Lipschitz plurisubharmonic envelope of $\phi$, i.e., the largest Lipschitz PSH function below $\phi$ is given by the formula
\[
{\check \phi}(x)=\inf\{\int_0^{2\pi}\phi (P(e^{i\theta})\, \frac{d\theta}{2\pi}; P:\C\to X \,  \hbox{polynonial with}\,  P(0)=x\}.
\] 
Note that $\hat \phi=-\check \psi$, where $\psi=-\phi$. 
\end{itemize}

 \noindent {\bf Example 2.7: The Skorokhod transfer}
 
  Again, letting ${\mathcal S}$ be the class of --possibly randomized-- Brownian stopping times, and  
     define 
 \begin{equation}
{\mathcal SK}(\mu, \nu)=\left\{ \begin{array}{llll}
0 \quad &\hbox{if $B_0\sim \mu$ and $B_\tau \sim\nu$ for some $\tau \in {\mathcal S}$, }\\
+\infty \quad &\hbox{\rm otherwise.}
\end{array} \right.
\end{equation}  
The following is a  classical result of Skorokhod. See, for example \cite{G-K-L2} for a proof in higher dimension.
\begin{prop} Let ${\mathcal A}$ be the cone of Lipschitz subharmonic functions on a domain $\Omega$ in $\R^n$. Then, the following are equivalent for two probability measures $\mu$ and $\nu$ on $\Omega$.
\begin{enumerate}
\item $\mu \prec_{\mathcal A} \nu$ (i.e, $\mu$ and $\nu$ are in subharmonic order).
\item There exists a stopping time $\tau \in {\mathcal S}$ such that $B_0\sim \mu$ and $B_\tau \sim\nu$. 
\end{enumerate}
\end{prop}
This means that ${\mathcal SK}$ is a backward linear transfer with Kantorovich operator given by $T^-f=f_{**}$, which is the smallest Lipschitz superharmonic function above $f$. This can also be written as  $T^-f=J_f$, where $J_f(x)$ is a viscosity solution for the heat variational inequality,
 	\begin{equation}
		 \max\left\{f(x)-J(x), \Delta J(x)\right\}=0.
	\end{equation}
Another representation for $J_f$ is given 
  by  the following  dynamic programming principle,
	\begin{align} 
		J_f(x) := \sup_{\tau \in \mathcal S}\mathbb{E}^{x}\Big[f(B_\tau)\Big].
	\end{align}

 \subsection{Mass transfers minimizing a transport cost between two points}

The examples in this subsection correspond to cost minimizing transfers, where a cost $c(x,y)$ of moving state $x$ to $y$  is given. \\

   \noindent {\bf Example 2.8: Monge-Kantorovich transfers} 
  
Any proper, bounded below, function $c$ on $X\times Y$ determines a backward and forward linear transfer. This is Monge-Kantorovich theory of optimal transport. One associates the map ${\calT}_c$ on  ${\mathcal P}(X)\times  {\mathcal P}(Y)$ to be the optimal mass transport  between two probability measures $\mu$ on $X$ and $\nu$ on  $Y$, that is \begin{eqnarray}
{\mathcal T}_c(\mu, \nu):=\inf\big\{\int_{X\times Y} c(x, y)) \, d\pi; \pi\in \mK(\mu,\nu)\big\},
\end{eqnarray}
where $\mK(\mu,\nu)$ is the set of probability measures $\pi$ on $X\times Y$ whose marginal on $X$ (resp. on $Y$) is $\mu$ (resp., $\nu$) {\it (i.e., the transport plans)}. Monge-Kantorovich theory readily yields  that ${\mathcal T}_c$ is a linear transfer. Indeed, 
if we define the operators 
\begin{equation}
T ^+_cf(y)=\inf_{x\in X} \{c(x, y)+f(x)\} \quad {\rm and} \quad T ^-_cg(x)=\sup_{y\in Y} \{g(y)-c(x, y)\}, 
\end{equation}
for any $f\in C(X)$ (resp., $g\in C(Y)$), then Monge-Kantorovich duality yields that for any probability measures $\mu$ on $X$ and $\nu$ on  $Y$, we have 
\begin{eqnarray*}
{\mathcal T}_c(\mu, \nu)&=& 
\sup\big\{\int_{Y}T ^+_cf(y)\, d\nu(y)-\int_{X}f(x)\, d\mu(x);\,  f \in C(X)\big\}\\
&=&\sup\big\{\int_{Y}g(y)\, d\nu(y)-\int_{X}T _c^-g(x)\, d\mu(x);\,  g \in C(Y)\big\}.
\end{eqnarray*}
This means that the Legendre transform $({\mathcal T}_c)^*_\mu (g)=\int_{X}T _c^-g(x)\, d\mu(x)$ and $T_c^-$ is the corresponding backward Kantorovich operator. Similarly,  $({\mathcal T}_c)^*_\nu (f)=-\int_Y T ^+_c(-f)(y) \, d\nu(y)$ on $C(X)$ and $T ^+_c$ is the corresponding forward Kantorovich operator. See for example Villani \cite{V}. \\

\noindent {\bf Example 2.9: The trivial Kantorovich transfer}
 
  Any pair of functions $c_1\in USC(X)$, $c_2\in LSC (Y)$ defines trivially a linear transfer via
$$
{\cal T}(\mu, \nu)=\int_Yc_2\, d\nu -\int_X c_1\, d\mu.
$$
The Kantorovich operators are then
$T^+f=c_2+\inf (f-c_1)$  
and  $T^-g=c_1+\sup(g-c_2)$.\\

 \noindent {\bf Example 2.10: The Kantorovich-Rubinstein transport}
 
  If $d: X\times X\to \R$ is a lower semi-continuous metric on $X$, then 
 \begin{equation}
  {\cal T}(\mu,\nu) = \|\nu-\mu\|^*_\mathrm{Lip}:=\sup\left\{\int_X u \, d(\nu-\mu); u
\mathrm{\ measurable}, \|u\|_\mathrm{Lip}\leq 1\right\}
 \end{equation}
 is a linear transfer, where here    $
\|u\|_\mathrm{Lip}:=\sup_{x\not =y}\frac{|u(y)-u(x)|}{d(x,y)}$. The corresponding forward Kantorovich operator is then the Lipschitz regularization $T^+f(x)=\inf\{f(y)+d(y,x); y\in X\}$, while $T^-f(x)=\sup\{f(y)-d(x,y); y\in X\}$.  Note that $T^+\circ T^-f=T^-f$. \\

 \noindent {\bf Example 2.11: The Brenier-Wasserstein distance} \cite{Br}
 
 We mention this important example even though it is not in a compact setting.  If $c(x, y)=\langle x, y\rangle$ on $\R^d\times \R^d$, and  $\mu, \nu$ are two probability measures of compact support on $\R^d$, then 
 \begin{eqnarray*}
{\cal W}_2(\mu, \nu)=\inf\big\{\int_{\R^d\times \R^d} \langle x,y\rangle \, d\pi; \pi\in \mK(\mu,\nu)\big\}.
\end{eqnarray*}
Here, the Kantorovich operators are  
\begin{equation}T ^+f(x) 
=-f^*(-x) \quad \hbox{and \quad $T ^-g(y)=  (-g)^*(-y)$,}
\end{equation}
 where $f^*$ is the convex Legendre transform of $f$. \\

 \noindent {\bf Example 2.12:  Optimal transport for a cost given by a generating function} (Bernard-Buffoni \cite{B-B1})
 
  This important example links the Kantorovich backward and forward operators with the forward and backward Hopf-Lax operators that solve first order Hamilton-Jacobi equations. Indeed, on a given compact manifold $M$, consider the cost:
\begin{equation}\label{BB}
c^L(y,x):=\inf\{\int_0^1L(t, \gamma(t), {\dot \gamma}(t))\, dt; \gamma\in C^1([0, 1), M);  \gamma(0)=y, \gamma(1)=x\},
\end{equation}
where $[0, 1]$ is a fixed time interval, and $L: TM \to \R\cup\{+\infty\}$ is a given Tonelli Lagrangian that is convex in the second variable of the tangent bundle $TM$. If now $\mu$ and $\nu$ are two probability measures on $M$, then 
 \begin{eqnarray*}
{\mathcal T}_{L}(\mu, \nu):=\inf\big\{\int_{M\times M} c^L(y, x) \, d\pi; \pi\in \mK(\mu,\nu)\big\}
 \end{eqnarray*}
is a linear transfer with forward Kantorovich operator given by 
 $T ^+_1f(x)=V_f(1, x)$, where
  $V_f(t, x)$ being the value functional
\begin{equation}\label{value.1}
V_f(t,x)=\inf\Big\{f(\gamma (0))+\int_0^tL(s,\gamma (s), {\dot \gamma}(s))\, ds; \gamma \in C^1([0, 1), M);   \gamma(t)=x\Big\}.
\end{equation}
Note that $V_f$ is --at least formally-- a solution for the Hamilton-Jacobi equation 
\begin{eqnarray}\label{HJ.0} 
\left\{ \begin{array}{lll}
\partial_tV+H(t, x, \nabla_xV)&=&0 \,\, {\rm on}\,\,  [0, 1]\times M,\\
\hfill V(0, x)&=&f(x). %\nonumber
\end{array}  \right.
 \end{eqnarray}
 Similarly, the backward Kantorovich potential is given by  
 $T ^-_1g(y)=W_g(0, y),$ $W_g(t, y)$ being the value functional
 \begin{equation}\label{value.2}
W_g(t,y)=\sup\Big\{g(\gamma (1))-\int_t^1L(s,\gamma (s), {\dot \gamma}(s))\, ds; \gamma \in C^1([0, 1), M);   \gamma(t)=y\Big\},
\end{equation}
 which is a solution for the backward Hamilton-Jacobi equation 
\begin{eqnarray}\label{HJ.0} 
\left\{ \begin{array}{lll}
\partial_tW+H(t, x, \nabla_xW)&=&0 \,\, {\rm on}\,\,  [0, 1]\times M,\\
\hfill W(1, y)&=&g(y). %\nonumber
\end{array}  \right.
 \end{eqnarray}

  \section{Envelopes and representation of Linear transfers}

\noindent The following relates mass transfers with the optimal weak transports of Gozlan et al. \cite{Go4}. 

\subsection{Representation of linear transfers as weak transports}

 \begin{thm}\label{prop.zero} 
 Let  ${\mathcal T}:{\mathcal M}(X)\times {\mathcal M}(Y)\to \R \cup\{+\infty\}$ be a functional  
such that $\{\delta_x; x\in X\}\subset D_1 ({\mathcal T})$. Then, the following are equivalent:
\begin{enumerate}
\item $\T$ is a backward linear transfer.

 \item  There is a map $T: C(Y) \to USC(X)$, such that for each $\mu \in D_1(\calT)$, $\T_\mu$ is convex lower semi-continuous on ${\cal P}(Y)$ and 
 \begin{equation}\label{LTb}
\hbox{${\mathcal T}^*_\mu (g)=\int_XT g(x) \, d\mu(x)$ \quad for any $g\in C(Y)$}.
\end{equation}

\item There exists a bounded below  lower semi-continuous function $c: X\times {\mathcal P}(Y)\to  \R \cup\{+\infty\}$ with $\sigma \to c(x, \sigma)$ convex  such that for any pair $(\mu, \nu)  \in {\mathcal P}(X) \times {\mathcal P}(Y)$,
\begin{equation}\label{weak.0}
{\mathcal T}(\mu, \nu)=\left\{ \begin{array}{llll}
\inf_\pi\{\int_X c(x, \pi_x)\, d\mu(x); \pi \in {\mathcal K}(\mu, \nu)\},    \,  &\hbox{if $\mu, \nu \in {\mathcal P}(X)\times {\mathcal P}(Y)$,}\\
+\infty \quad &\hbox{\rm otherwise.}
\end{array} \right.
\end{equation}
  where $(\pi_x)_x$ is the disintegration of $\pi$ with respect to $\mu$.

\end{enumerate} 

\end{thm}
The proof of this theorem will be split in Propositions \ref{prop.one}, \ref{prop.two} and \ref{prop.three}, where we can provide more details about the needed conditions.
The first establishes the easy equivalence between (1) and (2).
    \begin{prop}\label{prop.one} 1) If $\T$ is a backward linear transfer with Kantorovich operator $T$, then 
  \begin{equation}\label{LTB}
\hbox{${\mathcal T}^*_\mu (g)=\int_XT g(x) \, d\mu(x)$ \quad for any $g\in C(Y)$}.
\end{equation}
2) Conversely, if $\T$ satisfies (2) in the above Theorem, then $\T$ is a backward linear transfer and $T$ is a Kantorovich operator.

  \end{prop}
\noindent{\bf Proof:} 1) Since
\begin{equation}
{\mathcal T}(\mu, \nu)=\left\{ \begin{array}{llll}
\sup\big\{\int_{Y}g\, d\nu-\int_{X}{T^-}g\, d\mu;\,  g \in C(Y)\big\}\,  &\hbox{if $\mu, \nu \in {\mathcal P}(X)\times {\mathcal P}(Y)$,}\\
+\infty \quad &\hbox{\rm otherwise.}
\end{array} \right.
\end{equation}
%it follows that  ${\mathcal T}^*_\mu (g)=
we have that ${\mathcal T}_\mu  \geq \Gamma_{_{T, \mu}}^*$, where $\Gamma_{_{T, \mu}}$ is the convex lower semi-continuous function on $C(Y)$ defined by $\Gamma_{_{T, \mu}}(g)=\int_{X}{T }g(x)\, d\mu(x)$ since $T$ is a Kantorovich operator.  Moreover, ${\mathcal T}_\mu  = \Gamma_{_{T, \mu}}^*$ on the probability measures on $Y$.  If now $\nu$ is a positive measure with $\lambda:=\nu (Y) >1$, then 
\begin{align*}
\Gamma_{_{T, \mu}}^*(\nu)&=\sup\big\{\int_{Y}g(y)\, d\nu(y)-\int_{X}{T }g(x)\, d\mu(x);\,  g \in C(Y)\big\}\\
&\geq n\lambda -\int_XT(n)\d\mu\\
&=n(\lambda -1)-\int_XT(0)\d\mu,
\end{align*}
where we have used property (3) to say that $T(n)=n+T(0)$. Hence $\Gamma_{_{T, \mu}}^*(\nu)=+\infty$. A similar reasoning applies to when $\lambda<1$, and it follows that  $\T_\mu = \Gamma_{_{T, \mu}}^*$ and therefore $(\T_\mu)^*  = \Gamma_{_{T, \mu}}^{**}=\Gamma_{_{T, \mu}}$ since the latter is convex and lower semi-continuous on $C(Y)$.

2) Conversely, it is clear that since $\T_\mu$ is convex lower semi-continuous, we have for any $\mu \in D_1(\calT)$,
\[ 
{\mathcal T}(\mu, \nu)=\T_\mu(\nu)=(\T_\mu)^{**}(\nu)=\sup\big\{\int_{Y}g\, d\nu-\int_{X}{T^-}g\, d\mu;\,  g \in C(Y)\big\}.
\] 
Moreover, $Tg(x)=(\T_{\delta_x})^{*}(g)$, which easily implies that $T$ is a Kantorovich operator.

\subsection{Linear transfers and Kantorovich operators as envelopes} 

 We now associate to any convex lower semi-continuous functional on ${\mathcal P}(X)\times {\mathcal P}(Y)$ a backward and a forward linear transfer.  
This is closely related to the work of Gozlan et al. \cite{Go4}, who introduced the notion of {\it weak transport}. These are cost minimizing transport plans, where cost functions between two points are replaced by  {\it generalized costs} $c$ on $X\times {\mathcal P}(Y)$, 
 where $\sigma \to c(x, \sigma)$ is convex and lower semi-continuous. We now show that this notion is essentially equivalent to the notion of backward linear transfer, at least in the case where Dirac measures belong to the first partial effective domain of the map $\T$, that is when $\{\delta_x; x\in X\}\subset D_1 ({\mathcal T})$. We shall prove the following.
 
   \begin{prop}\label{prop.two} 
Let $c: X\times {\mathcal P}(Y)\to  \R \cup\{+\infty\}$ be a bounded below,  lower semi-continuous function such that $\sigma \to c(x, \sigma)$ is convex, and define for any pair $(\mu, \nu)  \in {\mathcal M}(X) \times {\mathcal M}(Y)$, the functional
 \begin{equation}\label{weak.0}
{\mathcal T}_c(\mu, \nu)=\left\{ \begin{array}{llll}
\inf_\pi\{\int_X c(x, \pi_x)\, d\mu(x); \pi \in {\mathcal K}(\mu, \nu)\},    \,  &\hbox{if $\mu, \nu \in {\mathcal P}(X)\times {\mathcal P}(Y)$,}\\
+\infty \quad &\hbox{\rm otherwise.}
\end{array} \right.
\end{equation}
 where $(\pi_x)_x$ is the disintegration of $\pi$ with respect to $\mu$. 
 
 Then, $\T_c$ is a backward linear transfer with Kantorovich operator 
 \begin{equation}\label{Kminus}
 T_c ^-g(x)=\sup\{\int_Yg(y)\, d\sigma (y)- c(x, \sigma); \sigma \in {\mathcal P}(Y)\}.
 \end{equation} 
\end{prop}
\noindent{\bf Proof:} We first compute the Legendre transform of the functional $(\T_c)_\mu$. Since $\T_c$ is $+\infty$ outside of the probability measures, we can write
  \begin{eqnarray*}
 ({\mathcal T}_c)^*_\mu (g)&=&\sup\{\int_Yg\, d\nu -  {\mathcal T}_c(\mu, \nu); \nu\in {\mathcal P}(Y)\}\\
 &=&\sup\{\int_Yg(y)\, d\nu (y) -  \int_Xc(x, \pi_x)\, d\mu(x); \nu\in {\mathcal P}(Y), \pi \in {\mathcal K}(\mu, \nu)\}\\
&=&\sup\{\int_X\int_Yg(y) d\pi_x(y)\, d\mu (x) -  \int_Xc(x, \pi_x)\, d\mu(x);  
\pi \in {\mathcal K}(\mu, \nu)\}\\
&\leq &\sup\{\int_X\int_Yg(y) d\sigma(y)\, d\mu (x) -  \int_Xc(x, \sigma)\, d\mu(x); \sigma \in {\mathcal P}(Y)\}\\
&\leq &\int_X\{\sup\limits_{\sigma \in {\mathcal P}(Y)}\{\int_Yg(y) d\sigma(y) - c(x, \sigma)\, d\mu\}\}\\
&=&\int_XT_c^-g(x) d \mu (x).
 \end{eqnarray*}
On the other hand, use your favorite selection theorem to find a measurable selection $x\to {\bar \pi}_x$ from $X$ to ${\mathcal P}(Y)$ such that
$
T_c^-g(x)=\int_Yg(y) d {\bar \pi}_x(y) - c(x, {\bar \pi}_x) \,\, \hbox{for every $x\in X$.}
$
It follows that 
 \begin{eqnarray*}
({\mathcal T}_c)^*_\mu (g)&=&\sup\{\int_Yg\, d\nu -  {\mathcal T}_c(\mu, \nu); \nu\in {\mathcal P}(Y)\}\\
&= &\sup\{\int_Yg\, d\nu - \int_X c(x, \pi_x) d\mu(x);\,\,  \nu\in {\mathcal P}(Y), \pi \in {\cal K}(\mu, \nu)\}.
\end{eqnarray*}
Let ${\bar \nu} (A)=\int_X {\bar \pi}_x(A)\ d\mu (x)$. Then, ${\bar \pi}(A\times B)=\int_A{\bar \pi}_x(B)\, d\mu $ belongs to ${\cal K}(\mu, {\bar \nu})$, hence 
\begin{eqnarray*}
({\mathcal T}_c)^*_\mu (g)&\geq &\int_Yg\, d{\bar \nu} - \int_X c(x, {\bar \pi}_x) d\mu(x)\\
&= & \int_Y\int_X g(y) d{\bar \pi}_x(y)\ d\mu (x) - \int_X c(x, {\bar \pi}_x) d\mu(x)\\
&=& \int_X\{\int_Yg(y) d {\bar \pi}_x(y) - c(x, {\bar \pi}_x)\}\, d\mu (x)\\
&=& \int_XT^-_cg(x) d \mu (x), 
 \end{eqnarray*}
 hence   $({\mathcal T}_c)^*_\mu (g)=\int_XT^-_cg(x) d \mu (x)$. 
 
 We now show that $\T_\mu$ is convex. For that let $\nu=\lambda \nu_1+(1-\lambda)\nu_2$ and find $(\pi^1_x)_x$ and  $(\pi^2_x)_x$ in ${\cal P}(Y)$ such that 
 $$\hbox{$\int_X\pi_x^id\mu(x)=\nu_i$ \quad and \quad $\int_X c(x, \pi^i_x)\, d\mu(x) \leq {\mathcal T}_c(\mu, \nu_i) +\epsilon$ for $i=1, 2$.}
 $$
  It is clear that the plan defined by 
 $\pi (A\times B):=\int_A(\lambda \pi^1_x (B)+(1-\lambda)\pi^2_x(B)) d\mu (x)$ belongs to ${\cal K}(\mu, \nu)$ and therefore, using the convexity of $c$ in the second variable, we have
 \begin{align*}
 \T_c(\mu, \nu)&\leq \int_X c(x, \pi_x)\, d\mu(x) \leq \int_X \lambda c(x, \pi^i_x)\, d\mu(x) + \int_X (1-\lambda) c(x, \pi^i_x)\, d\mu(x)\\
 & \leq \lambda {\mathcal T}_c(\mu, \nu_1) + (1-\lambda) {\mathcal T}_c(\mu, \nu_2) + \epsilon.
 \end{align*}
 It follows that for every $\nu \in {\cal P}(Y)$, 
 \[
 \T_\mu (\nu)=({\mathcal T}_c)_\mu^{**}=\sup\{\int_Yf(y) d\nu-\int_X T^-_cf d\mu; f\in C(Y)\}.
 \]
 Moreover, it is easy to see that $T_c$ satisfies properties a), b) and c) of a Kantorovich operator. We can therefore conclude that $\T$ is a backward linear transfer. This establishes that 3) implies 1) in Theorem \ref{prop.zero}. \qed

That  (1) implies 3) in Theorem \ref{prop.zero} will follow from the following general result.
  \begin{prop}\label{prop.three} 
 Let  ${\mathcal T}: {\mathcal P}(X)\times {\mathcal P}(Y)\to \R \cup\{+\infty\}$ be a bounded below lower semi-continuous functional that is convex  in each of the variables %, lower semi-continuous 
% functional  
such that $\{\delta_x; x\in X\}\subset D_1 ({\mathcal T})$, and  
consider ${\overline \T}$ to be the backward linear transfer associated to $c(x, \sigma)=\T(\delta_x, \sigma)$ by the previous proposition, and let 
\begin{equation}
{\tilde \T}(\mu, \nu):=\int_X \T^{} (x, \nu) d \mu (x).
\end{equation}
 Then,  $\T \leq \overline {\cal T} \leq {\tilde \T}$, and  
 \begin{enumerate}
 \item $\overline {\cal T}$ is the smallest backward  linear transfer greater than $\T$.
 \item $\overline {\cal T}$ is the largest backward linear transfer smaller than  ${\tilde \T}$.
 \end{enumerate}
 \end{prop}
\noindent{\bf Proof:} Note that 
\begin{equation}\label{Kminus}
 T ^-g(x)=\sup\{\int_Yg(y)\, d\sigma (y)-{\mathcal T}(\delta_x, \sigma); \sigma \in {\mathcal P}(Y)\}, 
 \end{equation} 
and therefore for each  $x\in X$, we have for each $g\in C(Y)$
\begin{equation}
T^-g(x)=({\overline \T}_{\delta_x})^*(g)=(\T^{}_{\delta_x})^*(g). 
\end{equation}
To show that ${\overline {\T}} \leq {\tilde  \T}$, 
write for an arbitrary $\mu \in {\cal P}(X)$,  
\begin{align*}
({\overline {\mathcal T}}_\mu)^*(g)&=\int_XT^-_cg(x) d \mu (x)\\
&=\int_X\sup_\sigma\{\int_Y g\, d\sigma -\T^{} (x, \sigma); \sigma\in {\mathcal P}(Y)\} d \mu (x)\\
&\geq \sup_\sigma\{ \int_Y g\, d\sigma -\int_X\T^{} (x, \sigma) d \mu (x); \sigma\in {\mathcal P}(Y)\}\\
&= \sup\{\int_Yg\, d\sigma -  {\tilde \T} (\mu, \sigma); \sigma\in {\mathcal P}(Y)\}\\
&=({\tilde \T}^{}_\mu)^* (g),
\end{align*}
hence ${\overline \T}\leq {\tilde \T}$ since both of them are convex in the second variable.\\
Note that $\T (\delta_x, \nu)= {\tilde \T}(\delta_x, \nu)$, hence 
 if   ${\cal S} \leq {\tilde \T}$ and ${\cal S}$ is a backward linear transfer with $S^-$ as a Kantorovich operator, then  
  \[
 S^-g(x)=(S_{\delta_x})^*(g) \geq (\tilde \T^{}_{\delta_x})^*(g)=({\overline \T}_{\delta_x}^*(g)=T^-g(x), 
 \]
and therefore ${\cal S} \leq {\overline \T}$. It follows that ${\overline \T}$ is the greatest backward linear transfer smaller than ${\tilde \T}$.

To show that $\T \leq {\overline \T}$, note that since $\T$ is jointly convex and lower semi-continuous, then 
 for each $f\in C(Y)$, the functional 
$$\mu \to (\T_\mu)^*(f):=\sup\{\int_Yfd\sigma -\T(\mu, \sigma); \sigma \in {\mathcal P}(Y)\}$$
is upper semi-continuous and concave. It follows from Jensen's inequality that 
\[
(\T_\mu)^*(f) \geq \int_X(\T_{\delta_x})^*(f) d\mu (x)=\int_XT^-f (x) d\mu (x), 
\]
hence 
$$\T (\mu, \nu)=(\T_\mu)^{**}(\nu) \leq \sup\{\int_Yfd\nu -\int_XT^-f d\mu; f\in C(Y)\}={\overline \T}(\mu, \nu). $$
If now ${\cal S} \geq \T$, then ${\overline {\cal S}}\geq {\overline \T}$, and if ${\cal S}$ is a linear transfer, then 
${\cal S}={\overline {\cal S}}\geq {\overline \T}$, and therefore $\overline {\cal T}$ is the smallest backward  linear transfer greater than $\T$.
\qed

  \begin{rmk} \rm A similar construction can be done to associate a forward linear transfer $\overline \T_+$ to a given functional $\T$ on ${\mathcal P}(X)\times {\mathcal P}(Y)$ provided $\{\delta_y; y\in Y\}\subset D_2 ({\mathcal T})$. Note that one can then define $\T$ as a backward (resp., forward) linear transfer if $\T={\overline \T}_-$ (resp.,  if $\T={\overline \T}_+$).
  
  \end{rmk}

 \begin{rmk} \rm Any  lower continuous convex functional ${\mathcal T}$ on ${\mathcal P}(X)\times {\mathcal P}(Y)$  that is finite on the set of Dirac measures 
 gives rise to a backward and forward optimal mass transport  ${\mathcal T}_c(\mu, \nu)$ associated to the cost function  $c(x,y)=\calT (\delta_x, \delta_y)$. It is then easy to see that  
 \[
\hbox{
${\calT}_{\delta_x}^*(g)=\sup\{\int_Y g d\nu -{\mathcal T}(\delta_x, \nu); \, \nu\in {\mathcal P}(Y) \}\geq \sup\{g(y) -c(x,y); y\in Y\}=T _c^-g(x)
$, }
\]
hence %{\mathcal T}_c(\mu, \nu)., 
  \begin{equation}\label{compare.00}
{\mathcal T}_c(\mu, \nu) \geq {\overline \T}(\mu, \nu)  \geq {\mathcal T}(\mu, \nu). 
\end{equation}
However, the inequality (\ref{compare.00}) is often strict.
 Moreover,  transfers need not be  defined on Dirac measures, a prevalent situation in stochastic transport problems. 
\end{rmk}

Dually, we give the following characterization of Kantorovich operators, which in particular, yields a uniqueness statement for the duality between them and linear transfers.  

\begin{thm} Let $T:C(Y)\to USC (X)$ be a map such that for every $x\in X$, 
\begin{equation}\label{proper}
\hbox{$\inf\limits_{\nu \in {\cal P}(Y)}\sup\limits_{g \in C(Y)}\big\{\int_{Y}g\, d\nu-T g (x)\big\}<+\infty.$}
\end{equation}
Then, there exists a Kantorovich operator ${\overline T}$ such that 
\begin{enumerate}
\item $ {\overline T}$ is the largest Kantorovich operator smaller than $T$ on $C(Y)$.
\item $\overline T$ can be written as 
\begin{equation}
{\overline T}f(x)=\sup\limits_{\sigma \in {\cal P}(Y)}\inf\limits_{g\in C(Y)}\{\int_Y (f-g)\, d\sigma + T g(x)\}.
\end{equation}
 
\end{enumerate}
\end{thm}
\begin{proof}  Consider the functional on ${\cal M}(X)\times {\cal M}(Y)$ given by 
 \begin{equation}\label{speciale}
{\mathcal T}(\mu, \nu)=\left\{ \begin{array}{llll}
\sup\big\{\int_{Y}g\, d\nu-\int_{X}{T}g\, d\mu;\,  g \in C(Y)\big\}\,  &\hbox{if $\mu, \nu \in {\mathcal P}(X)\times {\mathcal P}(Y)$,}\\
+\infty \quad &\hbox{\rm otherwise.}
\end{array} \right.
\end{equation}
Note that $\T$ is bounded below, convex, lower semi-continuous functional
and condition (\ref{proper}) means that $\{\delta_x; x\in X\}\subset D_1 ({\mathcal T})$. 
 Hence Proposition \ref{prop.two} applies to yield a backward linear transfer ${\overline \T}$ with a corresponding backward Kantorovich operator defined as 
${\overline T}f(x)=\sup\limits_\sigma \{\int fd\sigma - {\mathcal T}(\delta_x, \sigma)\}$.
Note now that 
\begin{align*}
(\T_{\delta_x})^*f&={\overline T}f(x)\\
&=\sup\limits_\sigma \{\int fd\sigma - {\mathcal T}(\delta_x, \sigma)\}\\
&=\sup\limits_\sigma \inf\limits _g \{\int fd\sigma - \int gd\sigma + Tg(x) \}\\
&\leq \inf\limits _g \sup\limits_\sigma \{\int fd\sigma - \int gd\sigma + Tg(x) \}\\
&=\inf\limits _g \{\sup (f -g) + Tg(x) \}\\
&=Tf (x), 
\end{align*}
If now $S$ is a Kantorovich map such that $S\leq T$, then 
\begin{align*}
{\overline T}f(x)&=\sup\limits_\sigma \{\int fd\sigma - {\mathcal T}(\delta_x, \sigma)\}\\
&=\sup\limits_\sigma \inf\limits _g \{\int fd\sigma - \int gd\sigma + Tg(x) \}\\
&\geq \sup\limits_\sigma \inf\limits _g \{\int fd\sigma - \int gd\sigma + Sg(x) \}\\
&= \inf\limits _g \sup\limits_\sigma \{\int fd\sigma - \int gd\sigma + Sg(x) \}\\
&=\inf\limits _g \{\sup (f -g) + Sg(x) \}\\
&=\inf\limits _g \{S[\sup (f -g) + g](x) \}\\
&\geq Sf (x).
\end{align*}
where the last three steps used the fact that $S$ satisfies properties (a), (b) and (c) of a Kantorovich operator. \qed
\end{proof}

\subsection{Powers and recessions of linear transfers}

\begin{prop}\label{sup} Let $\T$ be a convex coupling on ${\mathcal P}(X)\times {\mathcal P}(Y)$ of the form
\begin{equation}
\T (\mu, \nu):=\sup_{i\in I} \T_i(\mu, \nu)
\end{equation}
where for each $i\in I$, 
$\T_i(\mu, \nu)=\sup\limits_ {f\in C(Y)}\{\int_Yf d\nu -\int_X T_if d\mu\} $ 
for some map $T_i: C(Y)\to USC(X)$. Assume $\{\delta_x; x\in X\}\subset D_1(\T)$ and consider the envelope ${\overline \T}$ of $\T$ and the corresponding Kantorvich operator $\overline T$. Then,
\begin{enumerate}
\item $\overline T$ is given by the formula
\begin{equation}
{\overline T}f(x)=\sup\limits_{\sigma \in {\cal P}(Y)}\inf\limits_{g\in C(Y)}\{\int_Y (f-g)\, d\sigma +\inf\limits_i T_i g(x)\}.
\end{equation}
and therefore satsifies ${\overline T}f \leq \inf_i T_if$ on $C(Y)$. 
\item If each $T_i$ is a Kantorovich operator, then ${\overline T}f = \inf_i T_if$ if and only if $f\to \inf\limits_i T_if(x)$ is convex .
 \end{enumerate}
 \end{prop}  
\begin{proof} Note first that 
\begin{align*}
\T (\mu, \nu)=\sup_i \T_i(\mu, \nu)&=\sup\limits_{i}\sup\limits_ {f\in C(Y)}\{\int_Yf d\nu -\int_X T_if d\mu\}\\
&=\sup\{\int_Yf d\nu -\inf\limits_i\int_X T_if d\mu; f\in C(Y)\}
\end{align*}
and $(\T_\mu)^*(f)\leq \inf\limits_{i}\int_X T_i f d\mu$. Moreover, $f\to (\T_\mu)^*(f)$ is the convex envelope of $f\to \inf\limits_{i}\int_X T_i f d\mu$. 
If now ${\overline T}$ is the Kantorovich operator for the envelope $\overline \T$, then 
\[
{\overline T}f(x)=\sup\limits_{\sigma \in {\cal P}(Y)}\inf\limits_{g\in C(Y)}\{\int_Y (f-g)\, d\sigma +\inf\limits_i T_i f\}, 
\]
and consequently, 
${\overline T}f(x)=(\T_{\delta_x})^*(f) \leq \inf_i T_if(x)$.
 
If $f\to \inf\limits_i T_if(x)$ is convex and lower semi-continuous for any $x\in X$, then the envelope property of $f\to {\overline T}f(x)$ yield that  ${\overline T}f (x)=\inf_i T_if(x)$. Note that if each $T_i$ is a Kantorovich operator, then $f\to \inf_i T_if$ satisfies properties (a) and (c) of a Kantorovich operator but not necessarily the convexity assumption (b). 

\end{proof}

\begin{cor} Let ${\cal T}$ be a  backward linear transfer with  Kantorovich operator $T^-$.
\begin{enumerate}
\item If $\alpha: \R^+\to \R$ is a convex increasing function on $\R$, then $\alpha ({\cal T})$ is a backward convex transfer, whose envelope ${\overline{\alpha (\T)}}$ has a Kantorovich operator  equal to 
 \begin{equation}\label{mir}
T_\alpha^- f= \inf\limits_{s>0}\{sT^-(\frac{f}{s})+\alpha^\oplus(s)\},
\end{equation}
where
%\begin{equation}
 $ \alpha^{\oplus}(t)=\sup\{ts-\alpha(s); s\geq 0\}$.
 
\item  In particular,  if  ${\mathcal W}_c(\mu, \nu):=\inf\big\{\int_{X\times Y} c(x, y)) \, d\pi; \pi\in \mK(\mu,\nu)\big\}$   is the Monge-Kantorovich transport associated to a cost $c$, and $p>1$,   
then 
 \[
{\cal W}_c^p(\mu, \nu)\leq {\overline {\cal W}_c^p}(\mu, \nu)= 
\inf\limits_\pi\{\int_X {\cal W}_c^p(\delta_x, \pi_x)\, d\mu(x); \pi \in {\mathcal K}(\mu, \nu)\}.     
\]

\end{enumerate}
\end{cor}
  \noindent {\bf Proof:}
 It suffices to note that $\alpha (t)=\sup\{ts-\alpha^{\oplus} (s); s\geq 0\}$, hence
\begin{eqnarray*}
\alpha ({\mathcal T}(\mu, \nu))
&=& 
\sup\big\{s\int_{Y}f\, d\nu -s\int_{X}{T^-}f\, d\mu-\alpha^\oplus(s);\, s\in \R^+,  f \in C(Y) \big\}\\
&=& 
\sup\big\{\int_{Y}h\, d\nu -s\int_{X}{T^-}(\frac{h}{s})\, d\mu-\alpha^\oplus(s);\, s\in \R^+,  h \in C(Y) \big\}.
 \end{eqnarray*}
Therefore $\alpha (\T)$ is a convex coupling and its envelope ${\overline{\alpha (\T)}}$ has a Kantorovich operator 
 \begin{equation}
T_\alpha^- f\leq \inf\limits_{s>0}\{sT^-(\frac{f}{s})+\alpha^\oplus(s)\}.
\end{equation}
Note however that for each $s>0$, $T_s^-f:=sT^-(\frac{h}{s})+\alpha^\oplus(s)$ is a backward Kantorovich operator. Moreover, the function 
$(s, f) \to sT^-(\frac{f}{s})+\alpha^\oplus(s)$ is jointly convex on $\R^+\times C(Y)$, hence the infimum in $s$ is convex in $f$ and therefore we have equality in (\ref{mir}).  
\qed
 \begin{cor} \label{recession}Let ${\cal T}$ be a  backward linear transfer with  Kantorovich operator $T^-$. Then, the functional 
 \begin{equation}
{\T_f}(\mu, \nu)=\left\{ \begin{array}{llll}
0 \quad &\hbox{if $\T(\mu, \nu) <+\infty$}\\
+\infty \quad &\hbox{\rm otherwise.}
\end{array} \right.
\end{equation} 
is a backward linear transfer with Kantorovich operator equal to 
 \begin{equation}
  T^-_rf(x)=\lim\limits_{\lambda \to +\infty} \frac {T^-(\lambda f)(x)}{\lambda}.
  \end{equation}
\end{cor}
  \noindent {\bf Proof:}	 
  Given any bounded below, lower semi-continuous and convex  functional ${\mathcal T}: {\mathcal P}(X)\times {\mathcal P}(Y)\to \R \cup\{+\infty\}$ such that $\{\delta_x; x\in X\}\subset D_1 ({\mathcal T})$,  we can consider the Kantorovich operator that generates its backward linear transfer envelope $\bar \T$, that is 
   \begin{equation}\label{Kminus}
 T ^-g(x)=\sup\{\int_Yg(y)\, d\sigma (y)-{\mathcal T}(\delta_x, \sigma); \sigma \in {\mathcal P}(Y)\},
 \end{equation} 
and its corresponding recession function    
  \begin{equation*}
  T^-_rf(x)=\lim\limits_{\lambda \to +\infty} \frac {T^-(\lambda f)(x)}{\lambda}.
  \end{equation*}
  It is then clear that 
  $T^-_r$ is a Kantorovich operator and 
  \begin{equation} 
  T^-_rf(x)=\sup\{\int_Yf d\sigma; \sigma \in {\mathcal P}(Y), \T(\delta_x, \sigma) <+\infty\}.
   \end{equation}
  The corresponding linear transfer 
    \begin{equation}
   {\mathcal T}_r(\mu, \nu)=\sup\big\{\int_{Y}g(y)\, d\nu (y)-\int_{X}{T_r^-}g(x)\, d\mu(x);\,  g \in C(Y)\big\}.
  \end{equation}
  Since $T_r^-$ is positively homogenous, its associated transfer $\T_r$ can only take the values $0$ and $+\infty$. It is also clear that $\T_r$ is the envelope of $\T_f$, and therefore $\T_f \leq {\bar \T_f}=\T_r$. 
We now show that  if $\T$ is a linear transfer, then  $\T_r \leq \T_f$. Indeed, assume that $\T_f(\mu, \nu)=0$, then $\T(\mu, \nu)<+\infty$, hence for every $f\in C(X)$ we have 
\[
\int_X T^- (tf) d\mu \geq t\int_X f\ d\nu -\T(\mu, \nu),
\]
hence by dividing by $t$ and letting $t\to \infty$, we get from the monotone convergence theorem that $\int_XT_rfd\mu \geq \int_Xf d\nu$ and hence 
$\T_r(\mu, \nu) \leq 0=\T_f(\mu, \nu)$.

  \begin{rmk} \rm Note that the above shows that for a general backward linear transfer $\T$ with Kantorovich operator $T^-$ and Recession operator $T^-_r$, we have 
  \begin{equation}
  \T(\mu, \nu) <+\infty \quad \hbox{\rm if and only if \quad $\int_XT_rfd\mu \geq \int_Xf d\nu$ for every $f\in C(Y)$.}
  \end{equation}
  The latter condition can be seen as a {\em generalized order condition} between $\mu$ and $\nu$ that extends the notion of convex order. Indeed, if $\T$ is the balayage transfer, then $T^-f=T^-_rf=\hat f$, which is the concave envelope of $f$, and the condition does coincide with the convex order between measures. 
  \end{rmk}
  
 \section{Extension of Kantorovich operators} % from $C(X)$ to $USC_\sigma (X)$}
 
 In order to study the ergodic properties of a Kantorovich operator $T: C(X) \to USC (X)$, one needs to iterate it and therefore it is necessary to extend it to an operator $T: USC(X) \to USC (X)$ and eventually to $T: USC_\sigma(X) \to USC_\sigma (X)$ with the same properties (a), (b), (c) of a Kantorovich operator. 
 
 In order to define such an extension, we assume that $T$ is proper so that we can associate a linear transfer $\T$ on ${\mathcal P}(X)\times {\mathcal P}(Y)$ in such a way that 
 \begin{equation}\label{formula}
 Tf(x)=\sup\{\int_Y f d\nu -{\mathcal T}(\delta_x, \nu); \, \nu\in {\mathcal P}(Y) \} \quad \hbox{for every $f\in C(Y)$}.
 \end{equation}
 We shall then extend $T$ in such a way that (\ref{formula}) holds for every $f\in USC_\sigma(Y)$. Properties (a), (b) and (c) will then follow.
 
 \subsection{Extension of Kantorovich operators from $C(Y)$ to $USC_\sigma (Y)$}

 \begin{thm} Let $\T$ be a backward linear transfer such that  $\{\delta_x; x\in X\}\subset {\mathcal D}_1(\T)$, and let $T:C(Y) \to USC_\sigma(X)$ be the associated Kantorovich operator. 
 \begin{enumerate}
 \item For $f \in USC(Y)$,  define $T^{\widehat {\,}} f (x) := \inf\{ T g(x)\,;\, g \in C(Y), \, g \geq f \}$, then
 \begin{equation}\label{express.1}
\hbox{$T^{\widehat {\,}} f (x)=\sup\{\int_Y f d\nu -{\mathcal T}(\delta_x, \nu); \, \nu\in {\mathcal P}(Y) \}  
$,}
\end{equation}
and $T^{\widehat {\,}} $ maps $USC(Y)$ to $USC_\sigma (X)$.  

Moreover, if $T:C(Y) \to USC(X)$, then $T^{\widehat {\,}} $ maps $USC(Y)$ to $USC (X)$.
\item For $f \in USC_{\sigma}(Y)$, define $T^{\widehat {\,}}_{\widecheck {\,}} f := \sup\{ T^{\widehat {\,}} g\,;\, g \in USC(Y)\,,  g \leq f\}$, then 
 \begin{equation}\label{express.2}
T^{\widehat {\,}}_{\widecheck {\,}} f (x)=\sup\{\int_Y f d\nu -{\mathcal T}(\delta_x, \nu); \, \nu\in {\mathcal P}(Y) \},  
\end{equation}
and $T^{\widehat {\,}}_{\widecheck {\,}}$ maps $USC_{\sigma}(Y)$ to $USC_{\sigma}(X)$. % with the same properties (1)-(4)  in Proposition \ref{basicK}. 
\end{enumerate} 
 \end{thm}
\begin{proof} 1)
It is clear that for any $g \in C(Y)$, $g \geq f$, 
\[
\sup\{ \int_Y f\d\sigma - \T(\delta_x, \sigma)\,;\, \sigma \in \mcal{P}(Y)\}\leq \sup\{ \int_Y g\d\sigma - \T(\delta_x, \sigma)\,;\, \sigma \in \mcal{P}(Y)\} 
= T g(x).
\]
Therefore $\sup\{ \int_Y f\d\sigma - \T(\delta_x, \sigma)\,;\, \sigma \in \mcal{P}(Y)\}\leq \inf\{ T g(x)\,;\, g \in C(Y), g \geq f\} = T^{\widehat {\,}} f(x)$. 

On the other hand, let $g_n \searrow f$ be a decreasing sequence of continuous functions. Then, 
\as{
T^{\widehat {\,}} f(x) \leq T g_n(x) = \sup\{\int_Y g_n\d\sigma - \T(\delta_x, \sigma)\,;\, \sigma \in \mcal{P}(Y)\}
= \int_Y g_n \d \sigma_n - \T(\delta_x, \sigma_n),
}
for some probability measure $\sigma_n$. Consider an increasing subsequence $n_k$ so that $\sigma_{n_k} \to \bar{\sigma}$. Then for any $j \leq k$,
$
T^{\widehat {\,}}f (x) \leq \int_Y g_{n_j} \d\sigma_{n_k} - \T(\delta_x, \sigma_{n_k})
$
where we have used the fact that $g_{n_k} \leq g_{n_j}$ whenever $j \leq k$. For this fixed $j$, we have that $g_{n_j} \in C(Y)$ and so $\int g_{n_j}\d\sigma_{n_k} \to \int g_{n_j}\d\bar{\sigma}$ as $k \to \infty$. Hence we obtain
$$
T^{\widehat {\,}}f (x) \leq \lim_{k \to \infty}\int_Y g_{n_j}\d\sigma_{n_k} - \liminf_{k \to \infty}\T(\delta_x, \sigma_{n_k})\\
\leq \int_Y g_{n_j}\d\bar{\sigma} - \T(\delta_x, \bar{\sigma}).
$$
Finally we take a limit as $j \to \infty$ to obtain $T^{\widehat {\,}} f(x) \leq \sup\{ \int_Y f\d\sigma - \T(\delta_x, \sigma)\,;\, \sigma \in \mcal{P}(Y)\}$. 
It follows that $T^{\widehat {\,}} f$ satisfies (\ref{express.1}) and therefore $T^{\widehat {\,}} f\in USC_\sigma$. Note that $T^{\widehat {\,}}  f$ is bounded above since $T^{\widehat {\,}}f(x) \leq \sup_{y\in Y} f(y) - m_{\T}$,
where $m_{\T}$ is a lower bounded for $\T$.

If now $T:C(Y) \to USC(X)$, then $T^{\widehat {\,}} $ is in $USC (X)$ by its definition.

\noindent 2) For $f \in USC_{\sigma}(Y)$, we use the first part to write   
for any $g \in USC(Y)$, $g \leq f$, 
\as{
\sup\{ \int f\d\sigma - \T(\delta_x, \sigma)\,;\, \sigma \in \mcal{P}(Y)\} \geq \tilde{T}g(x)  
}
and so it is greater than $T^{\widehat {\,}}_{\widecheck {\,}}  f(x)$. On the other hand, for an increasing $g_n \nearrow f$, $g_n \in USC(Y)$, 
\as{
T^{\widehat {\,}}_{\widecheck {\,}} f (x) \geq T^{\widehat {\,}}  g_n(x)  \geq \int g_n\d\sigma - \T(\delta_x, \sigma),\quad \text{for any $\sigma \in \mcal{P}(Y)$}.
}
By the monotone convergence of $g_n$ to $f$, we may take the limit as $n \to \infty$ in the above inequality, and conclude
\eqs{
T^{\widehat {\,}}_{\widecheck {\,}} f (x) \geq \int f\d\sigma - \T(\delta_x, \sigma)\quad \text{for any $\sigma \in \mcal{P}(Y)$},
}
whereby taking the supremum in $\sigma$ yields $T^{\widehat {\,}}_{\widecheck {\,}} f (x) \geq \sup\{ \int f\d\sigma - \T(\delta_x, \sigma)\,;\, \sigma \in \mcal{P}(Y)\}$, and we are done showing that $T^{\widehat {\,}}_{\widecheck {\,}}$ maps $USC_{\sigma}(Y)$ to $USC_{\sigma}(X)$, while satisfying (\ref{express.2}).
\\

The following continuity properties of $T$ along monotone sequences of $USC(Y)$ and $USC_\sigma(Y)$ will be crucial for Sections 8 and 9.
 
\begin{lem}  \label{monotone}
Let $\T$ be a backward linear transfer as above, and let $T$ denote its corresponding Kantorovich operator, extended to $USC_\sigma(Y)$. 
\begin{enumerate}
\item If $f_n \in USC(Y)$, $f \in USC_\sigma(Y)$ with $f_n \searrow f$, then $\lim_{n \to \infty} T f_n = Tf$. 
\item If $f_n \in USC_\sigma(Y)$, $f \in USC_\sigma(Y)$ with $f_n \nearrow f$, then $\lim_{n \to \infty}T f_n = Tf$.
\end{enumerate}
\end{lem}
\prf{1) By monotonicity, $T f \leq \liminf_{n} Tf_n$. On the other hand let $\sigma_n$ achieve the supremum  in the definition of  $T f_n (x)$, i.e.,  
\eqs{
T f_n (x) = \int f_n \d\sigma_n - \T(\delta_x, \sigma_n).
}
Extract an increasing subsequence $n_k$ so that $\limsup_{n} T f_n (x) = \lim_{k}T f_{n_k}(x)$ and $\sigma_{n_k} \to \bar{\sigma}$. Then as before, we have $T f_{n_k} (x) \leq \int f_{n_j}\d\sigma_{n_k} - \T(\delta_x, \sigma_{n_k})$ for fixed $j \leq k$. As $f_{n_j} \in USC(Y)$, it follows that $\limsup_{n} T f_n (x) \leq \int f_{n_j}\d \bar{\sigma} - \T(\delta_x, \bar{\sigma})$. Then we let $j \to \infty$ and use monotone convergence.

2) Again, by monotoncity, $T f \geq \limsup_{n} T f_n(x)$. On the other hand, $T f_n(x) \geq \int f_n \d\sigma - \T(\delta_x, \sigma)$ for all $\sigma$. Hence by monotone convergence, $\liminf_{n} T f_n (x) \geq \int f \d \sigma - \T(\delta_x, \sigma)$ for all $\sigma$. Taking the supremum over $\sigma$ yields $\liminf_{n} T f_n (x) \geq T f(x)$.
}
\rmk{We note that In general, the operator $T$ cannot be extended to the class $C_{\delta\sigma\delta}(Y)=USC_{\sigma\delta}(Y)$, and the continuity property of $T$  in item (1) of the above lemma cannot be extended to sequence $f_n \in USC_\sigma(Y)$.
}
\end{proof}

 \begin{cor} Let  ${\T}: {\mathcal P}(X)\times {\mathcal P}(Y)\to \R \cup\{+\infty\}$ be a backward linear transfer
such that $\{\delta_x; x\in X\}\subset D_1 ({\mathcal T})$. Then, 
\begin{enumerate} 
\item For any $(\mu, \nu)\in {\mathcal P}(X)\times {\mathcal P}(Y)$, we have
\begin{align*}
 {\mathcal T}(\mu, \nu)&=\sup\big\{\int_{Y}g(y)\, d\nu(y)-\int_{X}{T ^-}g(x)\, d\mu(x);\,  g \in LSC(Y)\big\}\\
 &=\sup\big\{\int_{Y}g(y)\, d\nu(y)-\int_{X}{T ^-}g(x)\, d\mu(x);\,  g \in USC(Y)\big\}\\
 &=\sup\big\{\int_{Y}g(y)\, d\nu(y)-\int_{X}{T ^-}g(x)\, d\mu(x);\,  g \in USC_\sigma(Y)\big\}.
\end{align*}
\item The Legendre transform formula (\ref{LTb}) for $\T_\mu$, which holds for continuous functions on $Y$, also holds for $g\in USC_\sigma(Y)$, that is
\begin{equation}\label{extLT}
\hbox{${\mathcal T}^*_\mu (g):=\sup\{\int_Yg\d\sigma -\T(\mu. \sigma); \sigma \in {\mathcal P}(Y)\}=\int_XT ^-g \, d\mu $ \, for all $g\in USC_\sigma(Y)$}.
\end{equation}

\end{enumerate}
\end{cor}
\begin{proof} It is clear that $ {\mathcal T}(\mu, \nu)$ is smaller than all the expressions on its right. It is also clear that it suffices to show that 
$${\mathcal T}(\mu, \nu) \geq \sup\big\{\int_{Y}g(y)\, d\nu(y)-\int_{X}{T ^-}g(x)\, d\mu(x);\,  g \in USC_\sigma(Y)\big\}.$$
For that recall from Section 2 that for any $g\in USC_\sigma(Y)$ we have for every $x\in X$,
\begin{equation}\label{ext}
 T^-g(x)=\sup\{\int_Yg(y)\, d\sigma (y)-\T(x, \sigma); \sigma \in {\mathcal P}(Y)\}.
\end{equation}
Take now any $\pi \in {\mathcal K}(\mu, \nu)$ and its disintegration $(\pi_x)_x$ in such a way that $\nu (A)=\int_X\pi_x(A)\d\mu (x)$, then 
\[
 T^-g(x)\geq \int_Yg(y)\, d\pi_x (y)-\T(x, \pi_x),  
 \]
hence,
\begin{align*}
\int_{Y}g(y)\, d\nu(y) -\int_{X}{T ^-}g(x)\, d\mu(x)  &\leq \int_{Y}g(y)\, d\nu(y) -\int_X\int_Yg(y)\, d\pi_x (y)\, d\mu (x)+\int_X \T(x, \pi_x) \, d\mu (x)\\ &\leq \int_X \T(x, \pi_x) \, d\mu. 
\end{align*}
It follows from Theorem \ref{prop.zero} that 
\begin{equation}\label{oneside1}
 \int_{Y}g(y)\, d\nu(y) -\int_{X}{T ^-}g(x)\, d\mu(x) \leq \T (\mu, \nu),
\end{equation}
and (1) is done.

For (2) note first that (\ref{oneside1}) yields that $\int_XT ^-g(x) \, d\mu(x) \geq {\mathcal T}^*_\mu (g)$. On the other hand, assume $g\in USC(Y)$ and use (\ref{ext}) to find a measurable selection $x\to {\bar \pi}_x$ from $X$ to ${\mathcal P}(Y)$ such that
\[
T^-g(x)=\int_Yg(y) d {\bar \pi}_x(y) - c(x, \bar \pi_x) \quad \hbox{for every $x\in X$,}
\]
 and let $\sigma (A)=\int_X\bar \pi_x(A)\, d\mu (x)$, then 
\begin{align*}
\int_XT ^-g(x) \, d\mu(x)&=\int_X \int_Yg(y) d {\bar \pi}_x(y)\, d\mu (x) - \int_Zc(x, \pi_x)\, d\mu (x)\\ 
&\leq \int_Yg(y)\, d\sigma (y)-\T(\mu , \sigma), 
\end{align*}
hence $\int_XT ^-g(x) \, d\mu(x) \leq {\mathcal T}^*_\mu (g)$.
Note now that (\ref{extLT}) carries through increasing limits, hence it also holds for $g\in USC_\sigma(Y)$.\qed
\end{proof}
\subsection{Conjugate functions for bi-directional transfers}

Suppose now that $\T$ is both a backward and forward transfer with Kantorovich operators $T^-$ and $T^+$. We have the following notion motivated by the theory of mass transport.

\begin{defn} Say that a pair of functions $(f_1, f_2)\in USC(Y) \times LSC(X)$ are conjugate if:
\begin{equation}
\hbox{$T^-f_1=f_2$\quad  and \quad $T^+f_2=f_1$.}
\end{equation}
\end{defn} 
The following proposition shows in particular that for any function $g\in C(Y)$, the couple $(T^-g, T^+\circ T^-g)$ form a conjugate pair. 

\begin{prop}\label{few} Suppose  ${\mathcal T}: {\mathcal P}(X)\times {\mathcal P}(Y) \to \R\cup \{+\infty\}$   is both a forward and backward linear transfer, 
and that $\{(\delta_x, \delta_y); (x, y)\in X\times Y\} \subset D({\calT})$. Assume that $T^-:C(Y)\to USC(X)$ and that $T^+:C(X)\to LSC(Y)$, then for any $g\in C(Y)$  
(resp.,$f\in C(X)$)  
\begin{equation}\label{compare}
\hbox{$T ^+\circ T ^-g(y) \geq g(y)$ for $y\in Y,$ \quad  \quad $T ^-\circ T ^+f(x) \leq f(x)$ for $x\in X$,}
\end{equation}
and
 \begin{equation} 
 \hbox{$T ^-\circ T ^+\circ T ^-g=  T ^-g$ \quad and \quad $T ^+\circ T ^-\circ T ^+f = T ^+f.$}
 \end{equation} 
 In particular, 
  \begin{eqnarray}
{\mathcal T}(\mu, \nu)&=&\sup\big\{\int_{Y}T ^+\circ T ^-g(y)\, d\nu(y)-\int_{X}T^-g\, d\mu(x);\,  g \in C(Y)\big\}\label{tau-convex}\\
&=&\sup\big\{\int_{Y}T^+f(y)\, d\nu(y)-\int_{X}T^-\circ T^+f\, d\mu(x);\,  f \in C(X)\big\}.
\end{eqnarray}
\end{prop}

\noindent{\bf Proof:}  Note that $USC(X)\subset LSC_\delta(X)$, hence for  $\nu \in {\mathcal P}(Y)$, 
\begin{eqnarray*}
\int_YT ^+\circ T ^-g\, d\nu&=&-\T_\nu^*(-T ^-g)\\  
&=& -\sup\{-\int_XT_{\delta_x}^*(g)\, d\mu (x)-{\calT}(\mu, \nu); \mu \in {\mathcal P}(X)\}\\
&=& \inf \{\int_XT_{\delta_x}^*(g)\, d\mu (x)+{\calT}(\mu, \nu); \mu \in {\mathcal P}(X)\}\\
&\geq & \inf \{\int_XT_{\delta_x}^*(g)\, d\mu (x)+\int_Yg\, d\nu - \int_XT_{\delta_x}^*(g)\, d\mu; \mu \in {\mathcal P}(X)\}\\
&=&\int_Yg\, d\nu.
\end{eqnarray*}
 The last item follows from the above and the monotonicity property  of the Kantorovich operators.

\section{Linear transfers which are not mass transports}

We now give examples of linear transfers, which do not fit in the framework of Monge-Kantorovich theory. 
  
\subsection{Linear transfers associated to weak mass transports}

Weak mass transportations also arise from the work of Marton, who extended the work of Talagrand. The paper of Gozlan et al. \cite{Go4} exhibit many examples of which we single out the following. \\

 \noindent {\bf Example 4.2: Marton transports are backward linear transfers} (Marton  \cite{Ma1, Ma2})

These are transports of the following type: 
 \begin{equation}
 {\mathcal T}_{\gamma, d}(\mu, \nu)=\inf\left\{\int_X \gamma \left(\int_Y d(x,y)d\pi_x(y)\right)\, d\mu (x); \pi\in {\cal K}(\mu, \nu)\right\},
 \end{equation}
 where $\gamma$ is a convex function on $\R^+$ and $d:X\times Y \to \R$ is a lower semi-continuous functions. Marton's weak transfer correspond to $\gamma (t)=t^2$ and $d(x, y)=|x-y|$, which in probabilistic terms reduces to  
 \begin{equation}
 {\mathcal T}_2(\mu, \nu)=\inf\left\{\ensuremath{\mathbb{E}}[\ensuremath{\mathbb{E}}[|X-Y|\, |Y]^2]; X \sim \mu, Y\sim \nu \right\}. 
 \end{equation}
 This is a backward linear transfer with Kantorovich potential 
 \[
 T ^-f(x)=\sup\left\{ \int_Yf(y) d\sigma (y) - \gamma \left(\int_Y d(x, y)\, d\sigma (y)\right); \ \sigma \in {\cal P}(Y)\right\}.
 \]
 We now give applications to transfers that are mostly dependent on the barycenter of the measures involved. 
% \subsection{Barycentric cost functionals and martingale transports}  
 
 \begin{prop} \label{bar} Let $\T$ be a backward linear transfer on ${\cal P}(X) \times {\cal P} (Y)$, where $Y$ is convex compact such that for some lower semi-continuous functional $c:X\times Y\to \R$, we have 
 $$\T (x, \sigma)=c(x, \int_Y y\, d\sigma (y)) \quad \hbox{for all $x\in X$ and $\sigma \in {\cal P} (Y)$},$$ 
 where $\int_Y y\, d\sigma (y)$ denotes the barycenter of $\sigma$. Then, for every $f\in C(Y)$, 
 \[
 T^- f(x)=\sup\{{\hat f}(y)-c(x,y); y\in Y\},
 \]
 where $ \hat f$ is the concave envelope of $f$, i.e., the smallest concave usc function above $f$. 
  \end{prop} 
  \noindent {\bf Proof:} Note that $z$ is the barycenter of a probability measure $\sigma$ if and only if $\delta_z\prec_{\cal C} \sigma$ where $\cal C$ is the cone of convex functions.  Write now
  \begin{align*}
   T^-f(x)&=\sup\{\int_Y f\, d\sigma - c(x, \int_Y y\, d\sigma (y)); \sigma  \in {\cal P} (Y)\}  \\
   &=\sup_{z\in Y} \sup \{\int f\, d\sigma - c(x, y); \sigma  \in {\cal P} (Y), \delta_y\prec \sigma\}  \\
    &=\sup_{z\in Y} \{{\hat f}(z) - c(x, z)\}.
  \end{align*}

 \noindent {\bf Example 4.3:  A barycentric cost function} (Gozlan et al. \cite{Go4})
  
Consider the (weak) transport
  \begin{equation}
  {\cal T}(\mu, \nu)=\inf\left\{\int_X \|x-\int_Y y d\pi_x(y)\|\, d\mu (x); \pi\in {\cal K}(\mu, \nu)\right\}.
 \end{equation}
Again, this is a backward linear transfer, with Kantorovich potential 
 \[
 T ^-f(x)=\sup\{ {\hat f}(y) -\|y-x\|; y\in \R^n\}, 
 \]
where $ \hat f$ is the concave envelope of $f$.  

Note that the same holds if one uses other cones for balayage, such as the cone of subharmonic or plurisubharmonic functions.\\

\noindent {\bf Example 4.4: The variance functional}\label{var}

If the transfer is given by the variance functional 
$$\T(\mu, \nu):=I(\nu)=-{\rm var} (\nu):=|\int_Yy\ d\nu|^2-\int_Y|y|^2\, d\nu(y),$$ 
then, by letting $q$ be the quadratic function $q(y)=|y|^2$, we have
\begin{align*}
T ^-f(x)=&\sup\{ \int_Y f\, d\sigma - |\int_Yy\ d\sigma|^2+\int_Y|y|^2\, d\sigma(y); \sigma \in {\cal P}(Y)\}\\
 =&\sup\{ \int_Y (f +q)\, d\sigma - |\int_Yy\ d\sigma|^2; \sigma \in {\cal P}(Y)\}\\
 =&S^-(f+q)(x),
\end{align*} 
where  $S^-$ is the Kantorovich operator associated to the transfer ${\cal S}(\mu, \nu):=|\int_Yy\ d\sigma|^2$, which only depends on the barycenter and therefore $S^-g=\sup\{ {\hat g}(z)-|z|^2; z\in Y\}$. It follows that 
\[
T ^-f=\sup\{ {\widehat {f+q}}(z)-|z|^2; z\in Y\}.
\]

  Cost minimizing mass transport with additional constraints give examples of one-directional linear transfers. We single out the following: \\

 \noindent {\bf Example 4.5:  Martingale transports are backward linear transfers} 
  
  Martingale transports are $\cal C$-dilations where $\cal C$ is the cone of convex continuous functions on $\R^n$. 
   If $c:\R^d\times \R^d \to \R$ is a continuous cost function, then define the weak cost as 
  \begin{equation}
{\tilde c} (x, \sigma)=\left\{ \begin{array}{llll}
\int_{\R^d} c(x,y) \,d\sigma (y) \quad  &\hbox{if $ \delta_x \prec_C\sigma$,}\\
+\infty \quad &\hbox{\rm if not.}
\end{array} \right.
\end{equation}
  The corresponding martingale transport is then 
  \[
 {\mathcal T}_M(\mu, \nu)=\inf\{\int_{\R^d\times \R^d} {\tilde c}(x,\pi_x) \,d\ \mu; \pi\in {\cal K}(\mu,\nu)\}. 
  \]
  Equivalently, if $\mu, \nu$ are two probability measures we then consider  $MT(\mu,\nu)$ to be the subset of  ${\mathcal K}(\mu, \nu)$ consisting of {\it the martingale transport plans}, that is the set of probabilities $\pi$
on $\R^d \times \R^d$ with marginals $\mu$ and $\nu$, such that for $\mu$-almost $x\in\R^d$, 
the component $\pi_x$ of its disintegration $(\pi_x)_x$ with respect to $\mu$, i.e. $d\pi(x,y)=d\pi_x(y)d\mu(x)$, has its barycenter at $x$. As mentioned above, 
\begin{equation}
\hbox{$MT(\mu,\nu)\neq \emptyset$ if and only if $\mu \prec_{\cal C}\nu$.}
\end{equation}
 One can also use the probabilistic notation, which amounts to 
minimize 
$\E_{\rm P} \,c(X,Y)$
over all martingales $(X,Y)$ on a probability space $(\Omega, {\mathcal F}, P)$ into $\R^d \times \R^d$ (i.e. $E[Y|X]=X$) with laws $X \sim \mu$ and $Y \sim \nu$ (i.e., $P(X\in A)=\mu (A)$ and $P(Y\in A)=\nu (A)$ for all Borel set $A$ in $\R^d$). Note that in this case, the disintegration of $\pi$ can be written as the conditional probability   $\pi_x (A) =  \P {Y\in A|X=x}$. 
 
The martingale transport can be written as 
\begin{equation}
{\mathcal T}_M(\mu, \nu)=\left\{ \begin{array}{llll}
\inf\{\int_{\R^d\times \R^d} c(x,y) \,d\pi(x,y); \pi\in MT(\mu,\nu)\} \quad &\hbox{if $\mu\prec_C\nu$}\\
+\infty \quad &\hbox{\rm if not.}
\end{array} \right.
\end{equation}
This s a backward linear transfer with a backward Kantorovich operator  given by
\[
\hbox{$T _M^-f(x)=\hat{f}_{c, x}(x)$, where $\hat{f}_{c, x}$ is the concave envelope of the function $f_{c, x}:y\to f(y)-c(x, y)$.}
\]
  See Henri-Labord\`ere \cite{HL} and Ghoussoub-Kim-Lim \cite{G-K-L1} for higher dimensions.\\
 
  \noindent {\bf Example 4.6: Schr\"odinger bridge} (Gentil-Leonard-Ripani  \cite{GLR})

Let $M$ be a compact Riemannian manifold and fix some reference non-negative measure $R$ on path space $\Omega=C([0,1], M)$. Let $(X_t)_t$ be a random process on $M$ whose law is $R$, and denote by $R_{01}$ the joint law of the initial position $X_0$ and the final position $X_1$, that is $R_{01}=(X_0, X_1)_\#R$. For example (see \cite{GLR}), assume $R$ is the reversible Kolmogorov continuous Markov process associated with the generator $\frac{1}{2}(\Delta -\nabla V\cdot \nabla)$ and the initial measure $m=e^{-V(x)}dx$ for some function $V$.

For probability measures $\mu$ and $\nu$ on $M$, define
\begin{equation}\label{schrodtransfer}
\T_{R_{01}}(\mu,\nu) := \inf\{ \int_{M} \mathcal{H}(r_1^x, \pi_x)d \mu(x)\,;\, \pi \in \mathcal{K}(\mu,\nu),\, d \pi(x,y) = d \mu( x) d\pi_x(y)\}
\end{equation}
where $d R_{01}(x, y) = d m(x)d r_1^x(y)$ is the disintegration of $R_{01}$ with respect to its initial measure $m$. By Theorem \ref{zero.0}, $\T_{R_{01}}$ is a backward linear transfer (corresponding to the weak cost $c(x,p) = \mathcal{H}(r_1^x, p)$). Its Kantorovich operator is given by 
 \[
 T ^-f(x)=\log E_{R^x}e^{f(X_1)}=\log S_1(e^f)(x),
 \]
 where $(S_t)$ is the  semi-group associated to $R$.

The transfer (\ref{schrodtransfer}) is associated to the maximum entropy formulation of the Schr\"odinger bridge problem in the following way: Define the entropic transportation cost between $\mu$ and $\nu$ via the formula 
 \begin{equation}
 {\cal S}_R(\mu, \nu)=\inf\{\int_{M \times M}\log(\frac{d\pi}{dR_{01}})\, d\pi; \pi \in {\cal K}(\mu, \nu)\}.
 \end{equation} 
 Then, under appropriate conditions on $V$ (e.g., if $V$ is uniformly convex), then 
 \[
 {\cal T}_{R_{01}}(\mu, \nu)={\cal S}_R(\mu, \nu)-\int_{M}\log(\frac{d\mu}{dm})\, d\mu.
 \] 
  Note that when $V=0$, the process is Brownian motion with Lebesgue measure as its initial reversing measure, while when $V(x)=\frac{|x|^2}{2}$, $R$ is the path measure associated with the Ornstein-Uhlenbeck process with the Gaussian as its initial reversing measure. 
  
  \subsection{One-sided transfers associated to stochastic mass transport}

Let $M$ be a manifold (compact manifold or $\R^n$) and consider a Lagrangian on phase space $L: TM \to [0,\infty)$.
 Let $(\Omega, \mcal{F}, \P)$ be a complete probability space with normal filtration $\{\mcal{F}_t\}_{t \geq 0}$, and define $\mcal{A}_{[0,t]}$ to be the set of continuous semi-martingales $X: \Omega \times [0,t] \to M$ such that there exists a Borel measurable drift $\beta: [0,t] \times C([0,t]) \to \R^d$ for which
\enum{
\item $\omega \mapsto \beta (s,\omega)$ is $\mcal{B}(C([0,s]))_{+}$-measurable for all $s \in [0,t]$, where $\mcal{B}(C([0,s]))$ is the Borel $\sigma$-algbera of $C[0,s]$.
\item $W (s) := X(s) - X(0) - \int_{0}^{s}\beta (s')\d s'$ is a $\sigma(X(s)\,;\, 0 \leq s \leq t)$  
$M$-valued Brownian motion. 
}
For each $\beta$, we shall denote the corresponding $X$ by $X^\beta$ in such a way that 
\begin{equation}
d X^\beta(t)=\beta (t) dt + d W(t).
\end{equation}

\noindent {\bf Example 4.7: Stochastic mass transport between two probability measures}

Consider the following functional $\T: \mcal{P}(M)\times \mcal{P}(M) \to \R \cup \{+\infty\}$ defined for any pair of probability measures $\mu_0$ and $\mu_1$ on $M$ via the formula:
\begin{equation}\label{stoctrans}
\T (\mu_0,\mu_1) := \inf\lf\{\E \int_{0}^{1} L(X^\beta (s), \beta (s))\d s\,;\, X(0) \sim \mu_0, X(1) \sim \mu_1, X \in \mcal{A}_{[0,1]}\rt\},
\end{equation}
This stochastic transport does not fit in the standard optimal mass transport theory since it does not originate in the optimization according to a cost between two deterministic states. However, it 
 still enjoy a dual formulation (first proven in Mikami-Thieullin \cite{M-T} for the space $\R^d$) that permits it to be realised as a backward linear transfer. In fact, by introducing the operator $T_t: C(M) \to USC(M)$ via the formula
\begin{equation}\label{stochasticop}
T_{t} f(x) := \sup_{X \in \mcal{A}_{[0,t]}} \lf\{\E \lf[f(X(t))|X(0) = x\rt] - \E\lf[\int_{0}^{t} L(X(s),\beta_X(s,X))\d s|X(0) = x\rt]\rt\},
\end{equation}
then the duality relation between $\T$ and $T_t$ can be readily detailed. Indeed, 
an adaptation of the proofs of Mikami-Thieullin \cite{M-T} yields the following.  

\begin{prop}\label{backwardlintrans} Under suitable conditions on $L$ (for example if $L(x, \beta)=\frac{1}{2}|\beta|^2$),  the following assertions hold:
\begin{enumerate}
\item $\T$ is a backward linear transfer on $\mcal{P}(M)\times \mcal{P}(M)$ with Kantorovich operator $T_1$.  is the unique viscosity solution of 
\begin{equation}\label{timedep.0}
\frac{\partial u}{\partial t}(t,x) + \frac{1}{2}\Delta_x u(t,x) + H(x, \nabla_x u(t,x)) = 0,\quad (t,x) \in [0,1)\times M, 
\end{equation}
with $u(1,x) = f(x)$.

\item In particular, for any pair of probability measures $\mu_0$ and $\mu_1$ on $M$, we have  
\begin{equation}
\T (\mu_0,\mu_1)=\sup\{\int_M u(1, x)d\mu_1(x) -\int_M u(0, x) d\mu_0(x); u(t, x) \hbox{ solution of (\ref{timedep.0})}\}.
\end{equation}
\end{enumerate} 
\end{prop}

\noindent {\bf Example 4.8: Stochastic mass transport with fixed distribution at all time} 
 
 Suppose now $\mu_0 \in {\cal P}(M)$ and $\mu \in {\cal P}([0, 1] \times M)$. If the latter has Lebesgue measure as a first marginal, then we can disintegrate it and write it as  $d \mu =d\mu_t \, dt$, where $\mu_t$ is a probability measure on $M$. Consider the following functional on ${\cal P}(M)\times {\cal P}([0, 1] \times M)$, 
 \begin{eqnarray}\label{stoctrans}
\T (\mu_0, \nu)&:=&\T (\mu_0, (\mu_t)_{t>0})\nonumber \\
 &:= &\inf\lf\{\E \int_{0}^{1} L(X^\beta(s), \beta (s))\d s\,;\, X \in \mcal{A}_{[0,1]}, X(t) \sim \mu_t \,\, \forall t\in [0,1] \rt\},
\end{eqnarray}
if the first marginal of $\mu$ is Lebesgue measure and $+\infty$ otherwise.

\begin{prop}\label{backwardlintrans} Under suitable conditions on $L$ (for example if $L(x, \beta)=\frac{1}{2}|\beta|^2$),  the following assertions hold:
\begin{enumerate}
\item $\T$ is a backward linear transfer on ${\cal P}(M)\times {\cal P}([0, 1] \times M)$ with corresponding Kantorovich operator $T: C([0,1]\times M)\to C(M)$ defined for any $f \in C([0,1]\times M)$ as $u_f(0, x)$, where $u_f$ is a bounded continuous viscosity solution of the following Hamilton-Jacobi equation,
\begin{equation}\label{timedep.10}
\frac{\partial u}{\partial t}(t,x) + \frac{1}{2}\Delta_x u(t,x) + H(x, \nabla_x u(t,x)) + f(t, x)= 0,\quad (t,x) \in [0,1)\times M, 
\end{equation}
with $u_f(1,x) = 0$.

\item In particular, for any probability measures $\mu_0 \in {\cal P}(M)$ and $\nu \in {\cal P}([0, 1] \times M)$, we have  
\begin{equation}
\T (\mu_0,\mu)=\sup\{\int_0^1\int_M f(t, x)d\mu_t(x) dt-\int_M u_f(0, x) d\mu_0; u_f \hbox{ solves (\ref{timedep.10})}\}.
\end{equation}
 
\end{enumerate} 
\end{prop}

\noindent {\bf Example 4.9: The Arnold-Brenier variational principle for the incompressible Euler equation}
 
In \cite{Br1,Br2,Br3}, Brenier proposed several relaxed versions of the Arnold geodesic formulation of the incompressible Euler equation. We describe the following model which, strictly speaking is not stochastic, yet we include it in this section for comparison purposes. 
 
For  a smooth domain $D$ in $\R^d$ consider the space 
 $$H^1_t(\R^d)=\{\xi \in L^2([0,1], \R^d) \, \hbox{such that  $\dot\xi \in L^2([0,1], \R^d)$}\} 
 $$
 and denote by $H^1_t(D)$ the subset of $H^1_t(\R^d)$ consisting of those paths valued in $D$. \\
 For any $(s, t)\in [0, 1]^2$, we consider the projections  $\pi_{s, t}: C([0,1]; D) \to D\times D$ (resp., $\pi_{t}: C([0,1]; D) \to D)$ defined by $\pi_{s, t}f=(f(s), f(t))$ (resp., $\pi_tf=f(t)$). 
 
  For $\mu \in {\cal P}(C([0,1]; D))$, we denote by $\mu_{s,t}:=(\pi_{s,t})_\#\mu$ $\mu_t:=(\pi_t)_\#\mu$. Similarly, for 
  %its projection on the $t$-ccordinate and
  $\nu \in {\cal P}(D\times D)$, we denote by $\nu_0$ and $\nu_1$ its first (resp., second) marginal on $D$. 
  
If $\lambda$ is the normalized Lebesgue measure on $D$, we consider the functional
%  \begin{equation}
  \begin{align*}
   \T(\mu, \nu)=
\begin{cases}
\inf\{ \E_\mu\int_0^1\frac{1}{2}| \dot \xi|^2 dt;  &\hbox{if %$\nu_0=\nu_1=
$\mu_t=\lambda,\, \forall t \in [0, 1]$ and $\mu_{0, 1}=\nu$}\\
+\infty &{\rm otherwise}.
 \end{cases}
\end{align*}
 \begin{prop}\label{backwardlintrans}   
 The following assertions hold:
\begin{enumerate}
\item $\T$ is a backward linear transfer on ${\cal P}(C([0,1]; D)) \times  {\cal P}(D\times D)$ with corresponding Kantorovich operator $T: C([0,1]\times D)\to C(D\times D)$ defined for any $f \in C([0,1]\times D)$ as 
\begin{equation}
Tf (x, y)=\inf\{\int_0^1[\frac{1}{2}| \dot \xi|^2 -f(t, \xi_t)]dt; \xi \in H_t^1(D),\, \xi (0)=x, \xi (1)=y\}. 
\end{equation}
\item $T_f(x, y)=u_f(0, x, y)$, where $u_f$ is a bounded continuous viscosity solution of the Hamilton-Jacobi equation,
\begin{equation}\label{timedep.1}
\frac{\partial u}{\partial t}(t,x,y) + \frac{1}{2}|\nabla u(t, x, y)|^2 + f(t, x)= 0,\quad (t,x,y) \in [0,1)\times D\times D, 
\end{equation}
with $u_f(1,x,y) = 0$.
 \end{enumerate} 
\end{prop}

   \subsection{Transfers associated to optimally stopped stochastic transports}
 
 In dimension $d\geq 2$, there are many different types of martingales. If one chooses those that essentially follow a Brownian path, then we have the following linear transfers. \\

  \noindent {\bf Example 4.10:  Optimal subharmonic Martingale transfers} (Ghoussoub-Kim-Palmer \cite{G-K-P4})  
  
 Confining the problem to a convex bounded domain $O$ in $\R^d$, then if $(\mu, \nu)$ are in {\it subharmonic order,} i.e. $\mu \prec_{SH}\nu$, where $SH$ is the cone of subharmonic functions on $O$, 
we set, 
\begin{align}\label{eqn:KantoS}
	\mathcal{P}_c (\mu, \nu)= \inf_{\pi \in \mathcal{BM}(\mu,\nu)}\int_{O \times O} c(x, y)\pi(dx,dy), 
\end{align}
where each $\pi \in \mathcal{BM}(\mu,\nu)$ is a probability measure on $O \times O$ with marginals $\mu$ and $\nu$, satisfying
$
	\delta_x\prec_{SH} \pi_x \ {\rm for\ }\mu{\rm-a.e.}\ x, 
$
where $\pi_x$ is the disintegration of $\pi(dx,dy)=\pi_x(dy)\mu(dx)$. Otherwise, set $\mathcal{P}_c (\mu, \nu)=+\infty.$

By a remarkable theorem of Skorokhod \cite{Sko},  
such transport plans $\pi$ can be seen as joint distributions of $(B_0,B_\tau)\sim \pi$, where $B_0 \sim \mu$, $B_\tau \sim \nu$ and $\tau$ is a possibly randomized stopping time for the Brownian filtration. See for example \cite{G-K-L2}. The above problem associated to a cost $c$ can then be formulated as 
\begin{align}\label{eqn:primal}
	\mathcal{P}_c(\mu, \nu) =  \inf_{\tau} \Big\{\mathbb{E} \big[ c(B_0, B_\tau)\big]; \ B_0 \sim \mu \quad \& \quad B_\tau \sim \nu\Big\}, 
\end{align}
where $(B_t)_t$ is Brownian motion starting with distribution $\mu$ and ending at a stopping time $\tau$ such that $B_\tau$ realizes the distribution $\nu$. \\
In \cite{G-K-P4} it is shown that $\mathcal{P}_c$ is a backward linear transfer with a backward Kantorovich map given by $T^-f(x)=J_f(x,x)$, where 
\begin{align}\label{eqn:J-psi} J_f (x, y)= \sup_{\tau \leq \tau_O}  \mathbb{E} \big[ \psi (B^y_\tau) - c(x, B^y_\tau)  \big], 
 \end{align}
and $\tau_O$ is the first exit time of the set $O$.  Under some regularity assumptions on $f$ and $c$, and for each fixed $x\in \overline{O}$, the function $y\mapsto J_f(x,y)$ is the unique viscosity solution to the obstacle problem for $u\in C(\overline{O})$:
\begin{align*}
\begin{cases}
			u(y)\geq  f(y)-{ c}(x,y),\ {\rm for}\ y\in O,\\
			u(y)=  f(y)-{ c}(x,y)\ {\rm for\ }y\in \partial O,\\
			\Delta u(y)\leq 0\ {\rm for }\ y\in O,\\
			\Delta u(y)= 0\ {\rm whenever }\ u(y)> f(y)-{ c}(x,y),
		\end{cases}
\end{align*}
as well as the unique minimizer of the variational problem
\begin{align*}
	\inf\Big\{\int_O \big|\nabla u \big|^2dy;\ u\geq f- c(x,\cdot),\,  u\in H^1(O)\}.  
\end{align*}

 \noindent {\bf Example 4.11:  Optimally stopped stochastic transport} \cite{G-K-P1, G-K-P4}
 
  Given a Lagrangian $L:[0,1]\times \R^d\times \R^d \rightarrow\mathbb{R}$, consider the optimal stopping problem
 \begin{equation}
    {\mathcal T}_L(\mu,\nu)=\inf\left\{ \expect{\int_0^\tau L(t,X(t),\beta_X(t,X(t)))\,dt};X(0)\sim \mu, \tau \in {\cal S},  X_\tau\sim\nu,X(\cdot)\in \mathcal{A}\right\},  
\end{equation}
where  ${\cal S}$ is the set of possibly randomized stopping times, and $\mathcal{A}$ is the class of processes defined in Section 4.3.
 In this case, ${\mathcal T}_L$ is a backward linear transfer with Kantorovich potential given by 
$T ^-_{L}f={\hat V}_f(0, \cdot)$, where 
\begin{equation}
  {\hat V}_f(t,x)=\sup_{X\in {\mathcal A}}\sup_{T\in {\cal S}}\left\{\expcond{f(X(T))-\int_t^T L(s,X(s),\beta_X(s,X))\,ds}{X(t)=x}\right\},
\end{equation}
which is --at least formally-- a solution ${\hat V}_f(t, x)$ of  the quasi-variational Hamilton-Jacobi-Bellman inequality, 
	\begin{align} \label{eqn:HJB_0}
		\min\left\{\begin{array}{r} V_f(t,x)-f(x),  -\partial_t V_f(t,x)-{H}\big(t,x,\nabla V_f(t,x)\big)-\frac{1}{2}\Delta V_f(t,x)\end{array}\right\}=0.
	\end{align}
	
In Section 9, we shall deal in detail with optimal stochastic transports as a semi-group of backward  linear transfers in conjunction with a stochastic Mather theory. 
 
  \section{Operations on linear mass transfers}  
 
 Denote by ${\cal LT}_- (X \times Y)$ (resp., ${\cal LT}_+ (X \times Y)$) the class of backward (resp., forward) linear transfers on $X\times Y$.

  \begin{prop} The class ${\cal LT}_- (X \times Y)$ (resp., ${\cal LT}_+(X \times Y)$) is a convex subcone in the cone of convex weak$^*$ lower continuous functions on ${\mathcal P}(X)\times {\mathcal P}(Y)$. 
 
 \begin{enumerate}
 
 \item {\bf (Scalar multiplication)} If $a\in \R^+\setminus \{0\}$ and ${\mathcal T}$ is a backward linear  
transfer with Kantorovich operator $T ^-$, 
then the transfer $(a{\mathcal T})$ defined by
$(a{\mathcal T})(\mu, \nu)=a{\mathcal T}(\mu, \nu)$ is also a backward linear 
transfer with Kantorovich operator on $C(Y)$ defined by,
\begin{equation}
T _a^-(f)=aT ^-(\frac{f}{a}).
\end{equation}

\item {\bf (Addition)} If ${\mathcal T}_1$ and ${\mathcal T}_2$ are backward linear transfers  on $X\times Y$ with Kantorovich operator $T_1^-$, $T_2^-$ respectively, and such that $X\subset D({\mathcal T}_1)\cap D({\mathcal T}_2)$, then the sum  defined as 
\begin{equation}
 {\mathcal T}_1 \oplus  {\mathcal T}_2) (\mu, \nu):=\inf\{ \int_X\big\{ {\mathcal T}_1 (x, \pi_x)+  {\mathcal T}_2 (x, \pi_x)\big\}\, d\mu (x); \pi \in {\mathcal K}(\mu, \nu)\}
 \end{equation}
 is a backward linear transfer  on $X\times Y$, with Kantorovich operator given on $C(Y)$ by 
\begin{eqnarray*}
T^-f(x)&=&\sup\{ \int_Y f\,d\sigma -\T_1(x, \sigma)-\T_2(x, \sigma); \sigma \in {\mathcal P}(Y)\}\nonumber\\
&=&\inf\{T_1^-g(x)+T_2^-(f-g)(x); g\in C(Y)\}.
\end{eqnarray*}

\end{enumerate}
\end{prop}

\subsection{Convolution and tensor products of transfers}
 
 \begin{defn} Consider the following operations on transfers. 

 \begin{enumerate}
\item {\bf (Dual Sum)} If ${\mathcal T}_1$ and ${\mathcal T}_2$ are backward linear transfers  on $X\times Y$ with Kantorovich operator $T_1^-$, $T_2^-$ respectively, and such that $X\subset D({\mathcal T}_1)\cap D({\mathcal T}_2)$, then 
$\T_1\square \T_2$ is defined as the transfer whose Kantorovich operator is $T_1+T_2$, that is
\begin{equation}
\T_1\square \T_2 (\mu, \nu)=\sup\{\int_Y f d\nu -\int_X(T_1f +T_2f) d\mu; f\in C(Y)\}
\end{equation}

\item  ({\bf Inf-convolution}) Let $X_1, X_2, X_3$ be 3 spaces, and suppose ${\mathcal T}_1$ (resp., ${\mathcal T}_2$) are functionals on  ${\mathcal P}( X_1)\times {\mathcal P}(X_2)$ (resp., ${\mathcal P}( X_2)\times {\mathcal P}(X_3)$). The {\it convolution} of ${\mathcal T}_1$ and ${\mathcal T}_2$ is the functional on ${\mathcal P}( X_1)\times {\mathcal P}(X_3)$ given by
 \begin{equation}
{\mathcal T}(\mu, \nu):={\mathcal T}_{1}\star{\mathcal T}_{2}=
\inf\{{\mathcal T}_{1}(\mu, \sigma) + {\mathcal T}_{2}(\sigma, \nu);\, \sigma \in {\mathcal P}(X_2)\}.
\end{equation}

\item ({\bf Tensor product}) If ${\mathcal T}_1$ (resp., ${\mathcal T}_2$) are functionals  on ${\mathcal P}( X_1)\times {\mathcal P}(Y_1)$ (resp., ${\mathcal P}( X_2)\times {\mathcal P}(Y_2)$) such that $X_1\subset D({\mathcal T}_1)$ and $X_2\subset D({\mathcal T}_2)$, then the tensor product of  ${\mathcal T}_1$ and ${\mathcal T}_2$ is the functional on ${\mathcal P}(X_1\times X_2)\times {\mathcal P}(Y_1\times Y_2)$ defined by:
\begin{eqnarray*}
 {\mathcal T}_1 \tens  {\mathcal T}_2 (\mu, \nu)=\inf\left\{ \int_{X_1\times X_2}\big({\mathcal T}_1 (x_1, \pi_{x_1,x_2})+{\mathcal T}_2 (x_2, \pi_{x_1,x_2})\big)\, d \mu (x_1, x_2); \pi \in {\mathcal K}(\mu, \nu)\right\}.
\end{eqnarray*}
\end{enumerate}
%\end{defn}
 Similar statements hold for ${\cal LT}_+ (X \times Y)$.
 \end{defn}

The following easy proposition is important to what follows.

\begin{prop} \label{inf.tens} The following stability properties hold for the class of backward linear transfers.
\begin{enumerate}

\item If ${\mathcal T}_1$ (resp., ${\mathcal T}_2$) is a backward linear transfer  on  $X_1\times  X_2$ (resp., on $X_2\times X_3$) with Kantorovich operator $T _1^-$ (resp., $T _2^-$), then ${\mathcal T}_{1}\star{\mathcal T}_{2}$ is also a backward linear transfer  on  $ X_1\times X_3$ with Kantorovich operator equal to $T _1^-\circ T _2^-$. 

\item If ${\mathcal T}_1$ (resp., ${\mathcal T}_2$) is a backward linear transfer  on $X_1\times Y_1$ (resp., $X_2\times Y_2$) such that $X_1\subset D({\mathcal T}_1)$ and $X_2\subset D({\mathcal T}_2)$, then $ {\mathcal T}_1 \tens  {\mathcal T}_2$
  is a backward linear transfer  on $(X_1\times X_2)\times (Y_1\times Y_2)$, with   Kantorovich operator given by 
\begin{equation}
T ^-g(x_1, x_2)=\sup\{\int_{Y_1\times Y_2} f(y_1, y_2) d\sigma(y_1, y_2) - {\mathcal T}_1(x_1, \sigma_1)-{\mathcal T}_2(x_2, \sigma_2);\,  \sigma \in {\mathcal K}(\sigma_1, \sigma_2)  \}.
 \end{equation} 
 Moreover, \begin{equation}\label{tens}
    {\mathcal T}_1 \tens  {\mathcal T}_2(\mu,\nu_1\otimes\nu_2)
 \leq
 \mathcal{T}_{1}(\mu_1,\nu_1)+\int_{X_1}\mathcal{T}_{2}(\mu_2^{x_1},\nu_2)\,d\mu_1(x_1), 
\end{equation}
where  $d\mu(x_1,x_2)=d\mu_1(x_1)d\mu_2^{x_1}(x_2).$
 
\end{enumerate}

\end{prop}
Note that a similar statement holds for forward linear transfers, modulo order reversals. For example,   if ${\mathcal T}_1$ and ${\mathcal T}_2$) are forward linear transfer, then  ${\mathcal T}_{1}\star{\mathcal T}_{2}$ is a  forward linear transferon  $ X_1\times X_3$ with Kantorovich operator equal to $T _2^+\circ T _1^+$.  \\

 \noindent{\bf Proof:}  For 1), we note first that if ${\mathcal T}_{1}$ (resp., ${\mathcal T}_{2}$) is jointly convex and weak$^*$-lower semi-continuous on ${\mathcal P}(X_1)\times {\mathcal P}(X_2)$ (resp., ${\mathcal P}(X_2)\times {\mathcal P}(X_3)$),  
 then both  $({\mathcal T}_{1}\star{\mathcal T}_{2})_\nu: \mu \to ({\mathcal T}_{1}\star{\mathcal T}_{2})(\mu, \nu)$ and  $({\mathcal T}_{1}\star{\mathcal T}_{2})_\mu: \nu \to ({\mathcal T}_{1}\star{\mathcal T}_{2})(\mu. \nu)$ are convex and weak$^*$-lower semi-continuous. We now calculate their Legendre transform. For $g\in C(X_3)$,
 \begin{eqnarray*}
 ({\mathcal T}_{1}\star{\mathcal T}_{2})_\mu^*(g)&=&\sup\limits_{\nu \in {\mathcal P}(X_3)}\sup\limits_{\sigma \in  {\mathcal P}(X_2)} 
 \left\{\int_{X_3} g\, d\nu -{\mathcal T}_{1}(\mu, \sigma) - {\mathcal T}_{2}(\sigma, \nu)\right\}\\
&=&  
\sup\limits_{\sigma \in  {\mathcal P}(X_2)} 
 \left\{({\mathcal T}_{2})_\sigma^* (g)-{\mathcal T}_{1}(\mu, \sigma) \right\}\\
&=& \sup\limits_{\sigma \in  {\mathcal P}(X_2)} 
 \left\{\int_{X_2} T _2^-(g)\, d\sigma-{\mathcal T}_{1}(\mu, \sigma) \right\}\\
 &=&({\mathcal T}_{1})_\mu^* (T _2^-(g))\\
 &=& \int_{X_1} T _1^-\circ T _2^-g\, d\mu.
 \end{eqnarray*} 
  In other words, 
$
 {\mathcal T}_{1}\star{\mathcal T}_{2}(\mu, \nu)=
\sup\big\{\int_{X_3}g(x)\, d\nu(x)-\int_{X_1} T _1^-\circ T _2^-g\, d\mu;\,  f\in C(X_3) \big\}
$.\\

 2) follows immediately from the last section since we are defining the tensor product as a generalized cost minimizing transport, where the cost ion $X_1\times X_2 \times {\mathcal P}(Y_1\times Y_2)$ is simply,
 \[
 {\mathcal T}((x_1, x_2),  \pi)= {\mathcal T}_1(x_1, \pi_1)+{\mathcal T}_2(x_1, \pi_2), 
 \]
where $\pi_1, \pi_2$ are the marginals of $\pi$ on $Y_1$ and $Y_2$ respectively. ${\mathcal T}_1 \tens  {\mathcal T}_2$ is clearly its corresponding backward transfer  with $T ^-$ being its Kantorovich operator. 

More notationally cumbersome but straightforward  is how to write the Kantorovich operators of the tensor product $T ^-g(x_1, x_2)$ in terms of $T_1^-$ and $T_2^-$, in order to establish (\ref{tens}). 

\rmk{\rm Note that if $\T$ is any backward linear transfer on $X\times Y$, and $\T_\sigma$ is the one induced by a point transformation  $\sigma: Z\to X$, then one can easily check that for $\mu \in {\cal P}(Z)$ and $\nu \in {\cal P}(Y)$, we have
$
\T_\sigma \star \T (\mu, \nu)=\T(\sigma_\# \mu, \nu),
$
and its backward Kantorovich operator is given by ${\tilde T}f (z)= (T^-f)(\sigma (z))$. 
Similarly, if $\tau: Z\to Y$, $\mu \in {\cal P}(X)$ and $\nu \in {\cal P}(Z)$, then 
$
\T \star \T_\tau  (\mu, \nu)=\T(\mu, \tau_\# \nu), 
$
hence 
\[
\T_\sigma \star \T \star \T_\tau  (\mu, \nu)=\T(\sigma_\#\mu, \tau_\# \nu).
\]
 
\subsection{Hopf-Lax formulae and projections  in Wasserstein space}

By an obvious induction on the convolution property enjoyed by linear transfers, one can immediately show the following.

 \begin{prop}\label{n-conv} Let $X_0, X_1,...., X_n$ be $(n+1)$ compact spaces, and suppose for each $i=1,..., n$, we have functionals ${\mathcal T}_i$ on ${\mathcal P}( X_{i-1})\times {\mathcal P}(X_i)$. For any probability measures $\mu$ on $X_0$ (resp., $\nu$ on $X_n$), define
\begin{equation}
{\mathcal T}(\mu, \nu)=\inf\{{\mathcal T}_{1}(\mu, \nu_1) + {\mathcal T}_{2}(\nu_1, \nu_2) ...+{\mathcal T}_{n}(\nu_{n-1}, \nu);\, \nu_i \in {\mathcal P}(X_i), i=1,..., n-1\}.
\end{equation}
If each ${\mathcal T}_i$ is a linear forward  (resp., backward) transfer with a corresponding   Kantorovich operator $T _i^+: C(X_{i})\to C(X_{i+1})$ (resp., $T _i^-:C(X_{i})\to C(X_{i-1})$), then  ${\mathcal T}$ is a linear forward (resp., backward) transfer with a Kantorovich operator given by 
 \[
 \hbox{$T^+=T ^+_{n}\circ T ^+_{{n-1}}\circ...\circ T ^+_{1}$ (resp., $T ^-=T ^-_1\circ T ^-_2\circ...\circ T ^-_{n}$)}
 \]
In other words, the following duality formula holds:
 \begin{equation}
{\mathcal T}_c(\mu, \nu)= 
\sup\big\{\int_{X_n}T ^+_{n}\circ T ^+_{{n-1}}\circ...T ^+_{1}f(y)\, d\nu(y)-\int_{X_0}f(x)\, d\mu(x);\,  f\in C(X_0) \big\}
\end{equation}
respectively,
\begin{equation}
{\mathcal T}_c(\mu, \nu)=\sup\big\{\int_{X_n}g(y)\, d\nu(y)-\int_{X_0}T ^-_1\circ T ^-_2\circ...\circ T ^-_{n}g(x);\,  g\in C(X_n) \big\}.
\end{equation}

\end{prop}

The convolution of two linear transfers associated to optimal mass transports with costs $c_1$ and $c_2$ respectively,  is also a mass transport corresponding to a cost functional given by the convolution $c_1\star c_2$.  However, the above calculus allows us to convolute a mass transport with a general linear transfer, and to define a broken geodesic problems for stochastic processes.  

 \begin{prop}{\rm (Lifting convolutions to Wasserstein space)} \label{lifting} Let $X_0, X_1,...., X_n$ be compact spaces, and suppose for each $i=1,..., n$, we have a cost function $c_i: X_{i-1}\times X_i$, its corresponding optimal mass transport
 \[
 {\mathcal T}_{c_i}(\mu, \nu)=\inf \{\int_{X_{i-1}\times X_i}c_i(x, y)\, d\pi; \pi \in {\mathcal K}(\mu. \nu)\}, 
 \]
 and its forward and backward transfers $T^+_{c_i}$ and $T^-_{c_i}$ defined in Example 3.7. Consider the following cost function on $X_0\times X_n$, defined by
\[
c(x, x')=\inf\left\{c_1(x, x_1) +c_2(x_1, x_2) ....+c_n(x_{n-1}, x'); \, x_1\in X_1, x_2\in X_2,..., x_{n-1}\in X_{n-1}\right\}. 
\]
 Let $\mu$ (resp., $\nu$ be probability measures on $X_0$ (resp., $X_n)$, then the following holds  
 \begin{equation}
{\mathcal T}_c(\mu, \nu)=\inf\{{\mathcal T}_{c_1}(\mu, \nu_1) + {\mathcal T}_{c_2}(\nu_1, \nu_2) ...+{\mathcal T}_{c_n}(\nu_{n-1}, \nu);\, \nu_i \in {\mathcal P}(X_i), i=1,..., n-1\},
\end{equation}
and the infimum is attained at $\bar \nu_1, \bar \nu_2,..., \bar \nu_{n-1}$.
 \begin{eqnarray}
{\mathcal T}_c(\mu, \nu)&=& 
\sup\big\{\int_{X_n}T^+_{c_n}\circ T^+_{c_{n-1}}\circ...T^+_{c_1}f(x)\, d\nu(x)-\int_{X_0}f(y)\, d\mu(y);\,  f\in C(X_0) \big\}\\
&=&\sup\big\{\int_{X_n}g(x)\, d\nu(x)-\int_{X_0}T ^-_{c_n}\circ T^-_{c_{n-1}}...\circ T^-_{c_1}g(x);\,  g\in C(X_n) \big\}.
\end{eqnarray}
\end{prop}
\noindent{\bf Proof:} It suffices to verify these formulas in the case of two cost functions. We do so using duality by noting that both ${\mathcal T}_{c_1\star c_2}$ and ${\mathcal T}_{c_1}\star {\mathcal T}_{c_2}$ have the same backward Kantorovich map equal to 
\begin{eqnarray*}
T_{c_1}^-\circ T_{c_2}^-f(x)&=&\sup_{x_1\in X_1}\{T^-_{c_2}f(x_1)-c_1(x, x_1)\}\\
&=&\sup_{x_1\in X_1, x_2\in X_2 }\{f(x_2)-c_2(x_1, x_2)-c_1(x, x_1)\}\\ &=&
\sup_{x_2\in X_2 }\{f(x_2)-c(x, x_2)\}=T_{c}^-f(x).
\end{eqnarray*}
This is illustrated  by the following example.\\

 \noindent {\bf Example 5.1:  The ballistic transfer} (Barton-Ghoussoub \cite{B-G})

 Let $L$ be a Tonelli Lagrangian, then  the deterministic ballistic mass transport is defined as 
 \begin{equation}
    \underline{\mathcal B}_d(\mu,\nu):=\inf\left\{\expect{\langle V,X(0)\rangle +\int_0^T L(t,X,\dot X(t))\,dt};\, V\sim\mu,\, X\in {\cal A},\, X(T)\sim \nu\right\},
\end{equation} 
where $\mathcal{A}$ is the space of random processes $X_t$ such that $\dot X \in L^2[0, T], M)$. % is the class of processes defined in Section 4.3.
This corresponds to the following cost functional defined on phase space $M^*\times M$ by 
 \begin{equation}\label{bal}
b (v, x):=\inf\{\langle v, \gamma (0)\rangle +\int_0^1L(t, \gamma (t), {\dot \gamma}(t))\, dt; \gamma \in C^1([0, T), M);   \gamma(1)=x\}. 
\end{equation}
It is then clear that 
\begin{equation}\label{basic}
b(t, v, x)=\inf\{\langle v, y\rangle+c(t, y, x);\, y\in M\},
\end{equation}
where the cost $c$ is given by the generating function associated to $L$ in Example 3.11, which means that $ \underline{\mathcal B}_d$ is the convolution of the Brenier cost with the cost induced by the Lagrangian $L$. 
The corresponding forward Kantorovich operator is then
\begin{equation}
T ^+_{b}f(x)=T^+ _{c}\circ T^+_2f(x)=V_{\tilde f}(1, x),
\end{equation}
where $V_{\tilde f}(T, x)$ is the final state of the solution of the Hamilton-Jacobi equation (\ref{HJ.0}) starting at $T^+_2f(x):=\tilde f(x)=-f^*(-x)$. So, if $\mu$ (resp., $\nu$) is a given probability measure on $M^*$ (resp., $M$), then we have 
\begin{eqnarray}\label{bal2}
{\mathcal T}_{b}(\mu, \nu)&=&\inf \{\int_{M^*\times M} b(v, x)\, d\pi;\, \pi\in \mK(\mu,\nu)\}\\
 &=&\sup\left\{\int_MV_{\tilde f}(T,x)\, d\nu(x)- \int_{M^*} f(v)\, d\mu(v); \,  \hbox{$f$ convex on $M^*$} \right\}.
\end{eqnarray}
A similar formula holds for the backward Kantorovich operator. 

However, we can now convolute a mass transport with a general linear transfer as in the following example. \\

 \noindent {\bf Example 5.2:  Stochastic ballistic transfer} (Barton-Ghoussoub \cite{B-G})

Consider the stochastic ballistic transportation problem defined as:
\begin{equation}
    \underline{\mathcal B}(\mu,\nu):=\inf\left\{\expect{\langle V,X(0)\rangle +\int_0^T L(t,X^\beta,\beta(t))\,dt}\middle\rvert V\sim\mu, X(\cdot)\in \mathcal{A},X(T)\sim \nu\right\},
\end{equation}
where we are using the notation of Example 4.4.  Note that this a convolution of the Brenier-Wasserstein transfer of Example 3.12 with the general stochastic transfer of Example 4.4. Under suitable conditions on $L$, one gets that 
\begin{equation}
    \underline{\mathcal B}(\mu,\nu)=\sup \left\{\int g\, d\nu-\int \widetilde{\psi_g}\,d\mu; g\in C_b\right\},
\end{equation}
where $\widetilde{h}$ is the concave legendre transform of $-h$ and $\psi_{g}$ is the solution to the Hamilton-Jacobi-Bellman equation
\begin{align}\tag{HJB}
    \pderiv{\psi}{t}+\frac{1}{2}\Delta\psi(t,x)+H(t,x,\nabla\psi)=0,\quad  \psi(1,x)=g(x).
\end{align}
In other words, $ \underline{\mathcal B}$ is a backward linear transform with Kantorovich operator $T^-g=\widetilde{\psi_g}$. \\

 \noindent {\bf Example 5.3: Broken geodesics on Wasserstein space}
 
  Let $L$ be a Lagrangian as above, then for any finite sequence of times $t_1<t_1<....<t_n$, we consider the cost functions $c_i, i=1,..., n$, 
 \[
c_i(x, y)=c_{t_i, t_{i+1}}=\inf\left\{\int_{t_i}^{t_{i+1}}L(t, \gamma (t), {\dot \gamma}(t))\, dt;  
\gamma(t_{i})=x,  \gamma(t_{i+1})=y\right\}. 
\]
The theory of broken geodesics consist of finding for any fixed $x, y$, the critical points of the function $(t_1, t_2,..., t_n)\to c_{t_1,..., t_n}(x, y)$  given by
\begin{equation}
c_{t_1,..., t_n}(x, y)=\inf\left\{c_1(x, x_1) +c_2(x_1, x_2) ....+c_n(x_{n-1}, y); \, x_1, x_2,..., x_{n-1}\in M\right\}. 
\end{equation}
Thanks to Proposition \ref{n-conv}, one can consider a broken geodesic problem for stochastic processes by considering for any finite sequence of times $t_1<t_1<....<t_n$ the backward transfer 
 \begin{equation}
   \T_{t_i, t_{i+1}}(\mu, \nu)=\inf\left\{ \expect{\int_{t_i}^{t_{i+1}}L(t,X(t),\beta_X(t,X(t)))\,dt};X(t_i)\sim \mu,  X_{t_{i+1}}\sim\nu,X(\cdot)\in \mathcal{A}\right\},  
\end{equation}
where   
again $\mathcal{A}$ is the class of processes defined in Section 4.3.

This stochastic transport does not fit in the standard optimal mass transport theory since it does not originate in optimizing a cost between two deterministic states.  However, by a result of Mikami-Thieulin \cite{M-T}, $\T_{t_i, t_{i+1}}$ is 
a backward linear transfer with Kantorovich potential given by 
$T_{i+1, i}f={V}_f(t_i, \cdot)$, where 
\begin{equation}
  {V}_f(t,x)=\sup_{X\in {\mathcal A}}\expcond{f(X(T))-\int_t^{t_{i+1}} L(s,X(s),\beta_X(s,X))\,ds}{X(t)=x},
\end{equation}
which is --at least formally-- a solution of the Hamilton-Jacobi equation
\begin{eqnarray}\label{HJ.i} 
\left\{ \begin{array}{lll}
\partial_tV+H(t, x, \nabla_xV)+\frac{1}{2} \Delta V&=&0 \,\, {\rm on}\,\,  (t_i, t_{i+1})\times M,\\
\hfill V(t_{i+1}, y)&=&f(y). %\nonumber
\end{array}  \right.
 \end{eqnarray}
One can then define the bacward linear transfer 
\begin{equation}
\T_{t_1,..., t_n}(\mu, \nu)=\inf\left\{\T_{t_1, t_2}(\mu, \sigma_1) +\T_{t_2,t_3}(\sigma_1, \sigma_2) ...+\T_{t_{n-1}, t_n}(\sigma_{n-1}, \nu); \, \sigma_1,... \sigma_{n-1}\in {\cal P}(M)\right\},  
\end{equation}
in such a way that 
\begin{equation}
\T_{t_1,..., t_n}(\mu, \nu)= 
\sup\big\{\int_{M}f(x)\, d\nu(x)-\int_{M}T _{t_{2}, t_1}\circ ...\circ T _{t_{n}, t_{n-1}}f(y)\, d\mu(y);\,  f \in C(M)\big\}. 
%\end{eqnarray}
\end{equation}
The broken stochastic geodesics consist of finding for any pair $(\mu, \nu)$, the critical points of the function $(t_1, t_2,..., t_n)\to \T_{t_1,..., t_n}(\mu, \nu)$ on Wasserstein space.  \qed

  \noindent {\bf Example 5.4: Projection on the set of balay\'ees of a given measure} 
   
Let $\T$ be a linear transfer on $X\times Y$ and $K$ a closed convex set of probability measures on $Y$. We consider the following minimization problem 
\begin{equation}
 \inf\{ {\cal T}(\mu, \sigma); \sigma \in K\},
\end{equation}
which amounts to finding ``the projection" of $\mu$ on $K$, when the ``distance" is given by the transfer $\T$. In some cases, the set $K:={\cal C}(\nu)$ is a convex compact subset of ${\mathcal P} (Y)$ that depends on a probability measure $\nu$ in such a way that the following map
\begin{equation*}
{\cal S}(\sigma, \nu)=\left\{ \begin{array}{llll}
0 \quad &\hbox{if $\sigma \in {\cal C}(\nu)$}\\
+\infty \quad &\hbox{\rm otherwise.}
\end{array} \right.
\end{equation*}
is a backward transfer on $Y\times Y$. It then follows that 
 \[
 \inf\{ {\cal T}(\mu, \sigma); \sigma \in {\cal C}(\nu)\}= \inf\{ {\cal T}(\mu, \sigma)+ {\cal S}(\sigma, \nu); \sigma \in \cal P (X)\}={\cal T}\star {\cal S} (\mu, \nu).
 \]
If now $T^-$ (resp., $S^-)$ are the backward Kantorovich operators for ${\cal T}$ (resp., ${\cal S}$), then  by Proposition \ref{con.con}, the Kantorovich operator for ${\cal T}\star {\cal S}$ is $T^-\circ S^-$, that is
 \begin{equation}
 \inf\{ {\cal T}(\mu, \sigma); \sigma \in {\cal C}(\nu)\}=\sup\{\int_Yg\, d\nu-\int_{X}T^-\circ S^-g\, d\mu;  g\in C(Y)\}.
\end{equation}
Here is an example motivated by a recent result in  
\cite{GoJu}. \\
 
 Consider now the problem 
\begin{equation}
{\cal P}(\mu, \nu)=\inf\{ {\cal T}_c(\mu, \sigma); \sigma\prec_C \nu\}, 
\end{equation}
 where  ${\cal T}_c$ is the optimal mass transport associated to a cost $c(x, y)$ on $X\times Y$, and $\prec_C$ is the convex order on a convex compact set $Y$. Then, 
 \[
  {\cal P}(\mu, \nu)= {\cal T}_c\star {\mathcal B} (\mu, \nu)  
\]
where ${\cal B}$ is the Balayage transfer. It follows that ${\cal P}$ is a  linear transfer  
with backward Kantorovich operator given by the composition of those for ${\cal T}_c$ and ${\cal B}$, that is
 \[
 T ^-f(x)=\sup\{ {\hat f} (y) - c(x,y); y\in Y\}, 
 \]
where $\hat f$ is the concave envelope of $f$ on $Y$.  
We note that this is the same Kantorovich operator as for 
 the (weak) barycentric transport (See Proposition \ref{bar}). In other words, we can then deduce the following result of Gozlan-Juillet \cite{GoJu}. 
Write $${\cal T}^c_B(\mu, \nu):=\inf\left\{\int_X c(x, \int_Y y d\pi_x(y))\, d\mu (x); \pi\in {\cal K}(\mu, \nu)\right\}.$$
 \begin{cor} Let $c$ be a lower semi-continuous cost functional on $X\times Y$, where $Y$ is convex compact. Then the following holds:
 \begin{enumerate}
  \item 
 $ {\cal T}_c\star {\cal B}= 
 {\cal T}^c_B.$  
 \item 
$
{\cal T}_c \oplus {\cal B}={\cal T}^c_M
$, where the latter is the martingale transport of Example 4.4.
 \end{enumerate}
 \end{cor}
Similar manipulations can be done when the balayage is given by the cones of subharmonic or plurisubharmonic functions. \\ 

\section{Distance-like transfers}

Suppose now that $\T$ is a functional on ${\mathcal P}(X)\times {\mathcal P}(Y)$  satisfying the triangular inequality, that is 
\begin{equation}\label{Tr}
\T (\mu, \nu)\leq \T (\mu, \sigma)+\T (\sigma, \nu) \quad \hbox{for all $\mu$, $\nu$ and $\sigma$ in ${\mathcal P}(X)$},  
\end{equation}
which translates into $\T \leq \T\star \T$ and if $\T$ is a backward transfer to $T^-\circ T^- \leq T^-$. 

Note that if in addition $\T(\mu, \mu) =0$ for every $\mu \in {\mathcal P}(X)$, then $\T = \T\star \T$.  We shall call such a transfer {\it idempotent}. 
It is easy to see that $\T$ is idempotent if and only if $(T^-)^2=T^-$ on $USC (X)$.

\subsection{Characterization of $\T$-Lipschitz functions on Wasserstein space}

\begin{prop} \label{Lips} Let $\T$ be a backward linear transfer on a compact space $X$ and $T^-$ be its associated backward Kantorovich map. 
  If $\T$ satisfies (\ref{Tr}), then 
  \begin{equation}\label{ineq}
\T(\mu, \nu)\geq \sup\{\int_XT^-f \, d (\nu-\mu); f\in C(X)\} \quad \hbox{for any $\mu\in {\cal P}(X)$ and $\nu \in {\cal A}$}.
\end{equation}
Moreover, if $\T$ is also a forward transfer, then for any $\mu, \nu \in {\cal A}$.
  \begin{equation}\label{ineq}
\T(\mu, \nu)= \sup\{\int_XT^-f \, d (\nu-\mu); f\in C(X)\}=\sup\{\int_XT^-\circ T^+f \, d (\nu-\mu); f\in C(X)\}.
 \end{equation}
%\end{enumerate}
  \end{prop} 
\noindent {\bf Proof:} The proof is straightforward since for every $\nu \in {\cal A}$, we have 
\begin{equation}\label{oneside}
\int_XT^-g\, d\nu \geq \int_Xg \, d\nu \hbox{\, \, for every $g\in C(X).$}
\end{equation}
 while if $\T$ satisfies (\ref{Tr}), then $\int_X(T^-)^2g\, d\mu \leq \int_XT^-g \, d\mu$ for every $\mu \in {\cal P}(X)$. 
 
 If now $\T$ is also a forward transfer, then  $
\int_XT^+ f \, d \nu =\inf\{\int_X f d\sigma +\T(\sigma, \nu);\sigma \in {\mathcal P}(X)\}$, hence for every $\mu\in {\cal A}$, 
\[
\hbox{$ \int_XT^+g\, d\mu \leq \int_Xg \, d\mu\leq \int_XT^-g\, d\mu$ for every $g\in C(X)$.}
 \]
Since by (\ref{compare}) of Proposition \ref{few}, we  have  
\begin{equation}%\label{compare}
\hbox{$T ^+\circ T ^-g(y) \geq g(y)$ for $y\in Y,$ \quad  \quad $T ^-\circ T ^+f(x) \leq f(x)$ for $x\in X$,}
\end{equation}
it follows that for every $\mu \in {\cal A}$  
 \begin{equation*}%\label{diag}
\hbox{$ \int_XT^+\circ T^-g\, d\mu =\int_XT^-g \, d\mu$ for every $g\in C(X)$.}
 \end{equation*}
 and 
  \begin{equation*}
\hbox{$ \int_XT^-\circ T^+f\, d\mu =\int_XT^+f \, d\mu$  for every $g\in C(X)$.}
 \end{equation*}
Assertion (\ref{rubin}) follows by recalling from Proposition \ref{few} that 
  \begin{eqnarray*}
{\mathcal T}(\mu, \nu)&=&\sup\big\{\int_{Y}T ^+\circ T ^-g(y)\, d\nu(y)-\int_{X}T^-g\, d\mu(x);\,  g \in C(Y)\big\}\label{tau-convex}\\
&=&\sup\big\{\int_{Y}T^+f(y)\, d\nu(y)-\int_{X}T^-\circ T^+f\, d\mu(x);\,  f \in C(X)\big\}. 
  \end{eqnarray*}
 The above proposition states that the maps $\mu \to \int_XT^-f \, d \mu$ and $\mu \to \int_XT^+\circ T^-g \, d \mu$ are $1$-Lipschitz for the metric-like $\T$ on the subset $\cal A$ of Wasserstein space. We now show the converse, that is all Lipschitz maps on $\cal A$ are of this form.

 \begin{thm}\label{nice} Suppose  ${\mathcal T}: {\mathcal P}(X)\times {\mathcal P}(X) \to \R\cup \{+\infty\}$ is bounded below, weak$^*$-lower semi-continuous and convex  metric-like functional such that ${\cal A}:=\{\mu \in {\cal P}(X); \T(\mu, \mu)=0\}$ is non-empty. Assume in addition that for any $\mu, \nu \in {\cal P}(X)$, we have
\begin{equation}\label{C3}
\T(\mu, \nu)=\inf\{\T (\mu, \sigma)+\T(\sigma, \nu); \sigma \in {\cal A}\}.
\end{equation}
Then the following hold: 
\begin{enumerate}
\item For any functional $\Phi: {\cal A}\to \R$  that is $\T$-Lipschitz, there exists $f\in C(X)$ such that 
\begin{equation}
\Phi(\mu)=\int_X f d\mu \hbox{\,  for every  $\mu\in \cal A$. }
\end{equation} 
\item If $\T$ is also a backward linear transfer, then 
\begin{equation}
\Phi(\mu)=\int_X f d\mu=\int_X T^-f d\mu \hbox{\,  for every  $\mu\in \cal A$. }
\end{equation} 
\item If in addition $\T$ is also a forward linear transfer, then 
\begin{equation}
\Phi(\mu)=\int_X f d\mu=\int_X T^-f d\mu=\int_X T^+\circ T^-f d\mu \hbox{\,  for every  $\mu\in \cal A$. }
\end{equation} 
Note that the functions  $\psi_0:=T^-f$ and $\psi_1:=T^+\circ T^-f $ are conjugate in the sense that $\psi_0=T^-\psi_1$ and $\psi_1=T^+\psi_0$.
 
\item Moreover, if $g$ is a function in $C(X)$ such that $\int_X g d\mu=\Phi(\mu)$  for all $\mu \in \cal A$, then 
\begin{equation}\label{unique}
\psi_0 \leq T^-g \quad \hbox{and}\quad \psi_1 \geq T^+g.
\end{equation}
 \end{enumerate}
\end{thm} 

 \noindent {\bf Proof:} Let $\Phi$ be such that $\mu \to \Phi (\mu)$ is $\T$-Lipschitz on $\cal A$ and define 
$$\Phi_0(\mu)=\sup\limits_{\sigma \in {\cal A}} \{\Phi (\sigma) -\T (\mu, \sigma)\} \quad \hbox{and \quad $\Phi_1(\nu)=\inf\limits_{\sigma \in {\cal A}} \{\Phi (\sigma) +\T (\sigma, \nu)$\}}.
$$ 
 It is clear that 
\begin{equation}\label{on.A}
\Phi_1(\mu) \leq \Phi (\mu) \leq \Phi_0 (\mu) \hbox{\quad for all $\mu\in {\cal A}$}. 
\end{equation}
We now show that 
 \begin{equation}\label{ineq.on.X}
 \Phi_0(\mu) \leq  \Phi_1(\mu) \hbox{\quad for all $\mu\in {\cal P}(X)$}.
 \end{equation}
For that  note that (\ref{Tr}) and the fact that $\mu \to \Phi (\mu)$ is $\T$-Lipschitz on $\cal A$ yield
\begin{align*}
\Phi_0(\mu)-\Phi_1(\mu)&=\sup\limits _{\sigma, \tau \in {\cal A}}\{\Phi (\sigma) -\T (\mu, \sigma) -\Phi (\tau) -\T (\tau, \mu)\}\\
&\leq \sup\limits _{\sigma, \tau \in {\cal A}}\{\Phi (\sigma) - \Phi (\tau) -\T(\tau, \sigma)\} \leq 0. 
\end{align*}
This combined with (\ref{on.A}) shows that 
\[
\Phi_1(\mu) = \Phi (\mu) = \Phi_0(\mu) \hbox{\quad for all $\mu\in {\cal A}$}.  
\]
We now show that for every $\mu \in {\cal P}(X)$, 
\begin{equation}\label{BB22}
\Phi_0(\mu)=\sup \{\Phi_1(\sigma)-\T(\mu, \sigma); \sigma \in {\cal P}(X)\}.
\end{equation}
For every $\mu \in {\cal P}(X)$, we have 
\begin{align*}
\Phi_0(\mu)=\sup\limits_{\sigma \in {\cal A}} \{ \Phi (\sigma) -\T (\mu, \sigma)\}
=\sup\limits_{\sigma \in {\cal A}} \{ \Phi_1(\sigma) -\T (\mu, \sigma)\} 
\leq \sup\limits_{\sigma \in {\cal P}(X)} \{  \Phi_1(\sigma) -\T (\mu, \sigma)\}.   
\end{align*}
On the other hand, for any $\nu, \mu \in {\cal P}(X)$, we have 
\begin{align*}
\Phi_1(\nu)-\Phi_0(\mu)&=\inf\limits _{\sigma, \tau \in {\cal A}}\{\Phi (\sigma) +\T(\sigma, \nu)- \Phi (\tau) +\T(\mu, \tau)\} \\
&\leq \inf\limits _{\sigma, \tau \in {\cal A}}\{\T(\sigma, \nu) +\T (\tau, \sigma) +\T(\mu, \tau)\} \\
&\leq \inf\limits _{\sigma \in {\cal A}}\{\T(\sigma, \nu) +\T (\mu, \sigma)\} \\
&=\T (\mu, \nu).
\end{align*}
This shows (\ref{BB22}).  The other conjugate formula 
\begin{equation}
\Phi_1(\nu)=\inf \{\Phi_0(\sigma)+\T(\sigma, \nu); \sigma \in {\cal P}(X)\}
\end{equation}
can be proved in a similar fashion. 

Note now that $\Phi_0$ is a  concave weak$^*$-upper semi-continuous  function on ${\cal P}(X)$, while $\Phi_1$ is a convex weak$^*$-lower semi-continuous. Since $\Phi_0 \leq \Phi_1$ on ${\cal P}(X)$, there exists $f \in C(X)$ such that
\begin{equation}
\Phi_0 (\mu) \leq \int_Xf d\mu \leq \Phi_1(\mu) \quad \hbox{for all  $\mu\in {\cal P}(X)$,}
\end{equation} 
hence
\begin{equation}
\Phi_0 (\mu) =\int_Xf d\mu =\Phi_1(\mu)=\Phi (\mu)  \quad \hbox{for all  $\mu\in {\cal A}$.}
\end{equation} 
2) Suppose now $\T$ is also a backward linear transfer with $T^-$ as a Kantorovich operator, then 
\[
\int_XT^-f d\mu =\sup\limits_{\sigma \in {\cal P}(X)} \{\int_Xf d\sigma -\T(\mu, \sigma)\}\leq \sup\limits_{\sigma \in {\cal P}(X)} \{\Phi_1(\sigma) -\T(\mu, \sigma)\}=\Phi_0(\mu).
\]
On the other hand,  if $\mu\in {\cal A}$, 
\[
\int_XT^-f d\mu \geq \sup\limits_{\sigma \in {\cal A}} \{\int_Xf d\sigma -\T(\mu, \sigma)\}\geq \int_Xf d\mu -\T(\mu, \mu)= \int_Xf d\mu.
 \]
3) Suppose in addition that $\T$ is a forward linear transfer with $T^+$ as a Kantorovich operator, then 
\[
\int_XT^+\circ T^-f d\mu=\inf\limits _{\sigma \in {\cal P}(X)} \{\int_XT^-f d\sigma +\T(\sigma, \mu)\}\leq \inf\limits _{\sigma \in {\cal P}(X)} \{\Phi_0(\sigma )+\T(\sigma, \mu)\}=\Phi_1(\mu).
\]
On the other hand, $T^+\circ T^- f \geq f$ in such a way that 
\[
\int_XT^+\circ T^-f d\mu \geq \int_X f d\mu \geq \Phi_0(\mu).
\]
In other words,
$T^-f$ and $T^+\circ T^-f$ are two conjugate functions verifying
\[
\int_XT^+\circ T^-f d\mu=\int_XT^-f d\mu=\Phi (\mu) \quad \hbox{for all $\mu\in {\cal A}$}.  
\]
4) To prove (\ref{unique}), first note that
\begin{align*}
\int_XT^-f d\mu &\leq \Phi_0(\mu)=\sup \{\Phi (\sigma) -\T(\mu, \sigma); \sigma \in {\cal A}\}\\
&\leq \sup \{\int_X gd\sigma -\T(\mu, \sigma); \sigma \in {\cal P}(X)\}\\
&=\int_X T^-g d \mu.
\end{align*}
On the other hand, 
\begin{align*}
\int_XT^+\circ T^-f d\mu &=\inf\{\int_XT^-f d\sigma +\T(\sigma, \mu); \sigma \in {\cal P}(X)\}\\
&=\inf\{\int_XT^-f d\sigma +\T(\sigma, \lambda)+\T(\lambda, \mu); \lambda \in {\cal A}, \sigma \in {\cal P}(X)\}\\
&=\inf\{\int_XT^+\circ T^-f d\lambda  +\T(\lambda, \mu); \lambda \in {\cal A}\}\\
&=\inf\{\int_Xg d\lambda  +\T(\lambda, \mu); \lambda \in {\cal A}\}\\
&=\inf\{\int_Xg d\lambda  +\T(\lambda, \mu); \lambda \in {\cal A}\}\\
&\geq \inf\{\int_Xg d\lambda  +\T(\lambda, \mu); \lambda \in  {\cal P}(X)\}\\
&=\int_XT^+g d\mu, 
\end{align*}
which completes the proof of the theorem. \qed

\subsection{Examples of  idempotent transfers} 

In the next sections, we shall associate to any backward or forward linear transfer an idempotent linear transfer. For now, we give a few examples of some transfers that are readily idempotent. 

\begin{enumerate}

\item If $I$ is any bounded below convex lower semi-continuous functional on Wasserstein space ${\mathcal P} (Y)$, and $m=\inf\{I(\sigma); \sigma \in {\mathcal P}(Y)$, then $\T(\mu, \nu)= I(\nu)-m$ is an idempotent backward linear transfer with an idempotent Kantorovich map $T^-f= I^*(f) +m$. 

\item  Any transfer induced by a bounded positive linear operator $T$ with $T^2=T$ and $T1=1$, and in particular, any point transformation $\sigma$ such that $\sigma^2=\sigma$ as per Example 3. 2.

\item The balayage transfer ${\mathcal B}$ since its Kantorovich map is $Tf=\hat f$, where for example in the case of  balayage with convex functions, $\hat f$ is the concave envelope of $f$. 

\item  If $\T_c$ is an optimal mass transport associated to a cost function $c$, then $\T_c$ is idempotent if $c(x,x)=0$ for every $x\in X$ and $c$ satisfies the triangular inequality
\begin{equation}
c(x,z) \leq c(x, y) +c (y, z) \quad \hbox{for all $x, y, z$ in $X$,}
\end{equation}
in which case 
\begin{equation}
\T_c(\mu, \nu)=\sup\{\int_XT_cf \, d (\nu-\mu); f\in C(X)\}. 
\end{equation}
A typical example is the Rubinstein-Kantorovich optimal mass transport associated to any metric -such as in the original Monge problem- since the latter satisfies the triangular inequality and is zero on the diagonal. If $c_p(x,y) = |x-y|^p$ and $0 < p \leq 1$,  then the corresponding optimal mass transport is idempotent since  $c_p$ again satisfies the triangular inequality. 
 $c_p\star c_p(x,y) = c_p(x,y)$, so $T^2 f(x) = T f(x)$, i.e. $T$ is idempotent.
\end{enumerate}

 \noindent {\bf Example 6.7: An idempotent optimal Skorohod embedding}
 
The following transfer  was considered in Ghoussoub-Kim-Palmer \cite{G-K-P3}. 
	\begin{align} \label{eqn:Skorokhod_cost}
		{\mathcal T}(\mu,\nu) :=  \inf
		\Big\{\mathbb{E}\Big[ \int_0^\tau L(t,B_t)dt\Big];\ \tau \in {\mathcal S}(\mu, \nu)  
		\Big\},  
	\end{align}
where ${\mathcal S}(\mu, \nu)$ denotes the set of --possibly randomized-- stopping times with finite expectation such that $\nu$ is realized by the distribution of $B_\tau$ (i.e, $B_\tau \sim\nu$ in our notation), where $B_t$ is Brownian motion starting with $\mu$ as a source distribution, i.e., $B_0\sim \mu$. Note that  ${\mathcal T}(\mu,\nu)=+\infty$ if ${\mathcal S}(\mu, \nu)=\emptyset$, which is the case if and only if $\mu$ and $\nu$ are not in subharmonic order. In this case, It has been proved in \cite{G-K-P3} that under suitable conditions, the backward linear transfer is given by $T^-\psi=J_\psi(0, \cdot)$,  
 where $J_\psi:\R^+\times \R^d\rightarrow \R$ is defined via the 
	dynamic programming principle
	\begin{align}\label{eqn:dynamic_programming}
		J_\psi(t,x) := \sup_{\tau \in \mathcal{R}^{t,x}}\Big\{\mathbb{E}^{t,x}\Big[\psi(B_\tau)-\int_t^\tau L(s,B_s)ds\Big]\Big\},
	\end{align}
	where the expectation superscripted with $t,x$ is with respect to the Brownian motions satisfying $B_t=x$, and the minimization is over all finite-expectation stopping times $\mathcal{R}^{t,x}$ on this restricted probability space such that $\tau \ge t$.
	$J_\psi(t, x)$ is actually a ``variational solution" for 
the quasi-variational Hamilton-Jacobi-Bellman equation:
	\begin{align} %\label{eqn:variational_operator}
		 \min\left\{\begin{array}{r} J(t,x) -\psi(x)\\ -\frac{\partial}{\partial t}J(t,x)-\frac{1}{2}\Delta J(t,x)+L(t,x)\end{array}\right\}=0.
	\end{align}
	Note that $J_\psi (t, x) \geq \psi (x)$, that is $T^-\psi \geq \psi$ for every $\psi$.
	
Assume now $t\to L(t, x)$ is decreasing, which yields that $t\to J(t, x)$ is increasing (see \cite{G-K-P3}). if $\psi (x)=T^-\phi =J_\phi (0, x)$ for some $\phi$, then for each $\epsilon>0$, there is $\tau$ such that 
\begin{align*}
J_\psi (0, x)&\leq \mathbb{E}^{t,x}\Big[\psi(B_\tau)-\int_t^\tau L(s,B_s)ds\Big]+\epsilon\\
&\leq \mathbb{E}^{t,x}\Big[J_\phi(t, B_\tau)-\int_t^\tau L(s,B_s)ds\Big]+\epsilon\\
&\leq J_\phi (0,x)+\epsilon.  
\end{align*}
where the last inequality uses the supermartingale property of the process $t\to J_\phi(t, B_\tau)-\int_t^\tau L(s,B_s)ds$. It follows that 
$$T^-\phi (x)\leq (T^-)^2\phi (x)=J_\psi(0, x)\leq  J_\phi (0, x)=T^-\phi(x), $$
and $T^-$ is therefore idempotent.

\section{Ergodic properties of equicontinuous semigroups of transfers}

Let $X$ be a compact space. Our main purpose is to associate to any backward linear transfer $\T$ on $\mcal{P}(X) \times \mcal{P}(X)$, an idempotent backward linear transfer $\T_\infty$ with the properties listed in Theorem \ref{fixedpointcontinuous} below. For that, we shall associate to $\T$, the semi-group of transfers $(\T_n)_n$ defined  
for each $n\in \N$, as ${\cal T}_n= {\cal T}\star {\cal T}\star ....\star {\cal T}$ $n$-times and study its limit as $n\to \infty$. This section deals with the case where $\T$ is continuous, hence the sequence of transfers $(\T_n)_n$ is equicontinuous for the Wasserstein metric. We shall prove the following.

 \begin{thm}[Fixed point of weak$^*$ continuous backward linear transfers]\label{fixedpointcontinuous} Suppose $\T$ is a backward linear transfer on $\mcal{P}(X) \times \mcal{P}(X)$ that is weak$^*$-continuous on ${\cal M}(X)$, and let $T^-$ be the corresponding backward Kantorovich operator that maps $C(X)$ into $C(X)$. Then, there exists a constant $ c = c(\T)\in \R$, an idempotent backward linear transfer $\T^-_\infty$ on $\mcal{P}(X) \times \mcal{P}(X)$ with Kantorovich operator $T^-_\infty: C(X) \to C(X)$ such that, 
 \begin{enumerate}
 \item  The constant $c(\T)=\inf\{{\cal T}(\mu, \mu); \mu\in {\cal P}(X)\}$;
\item For every $f\in C(X)$ and $x\in X$, $\lim\limits_{n\to +\infty} \frac{(T^-)^nf(x)}{n}=-c$:
\item  $\T_\infty%(\mu,\nu) 
 = (\T - c) \star \T_\infty$ and $T^-\circ T^-_\infty f + c = T^-_\infty f$ for all $f \in C(X)$;
\item The set ${\cal A}:=\{\mu \in {\cal P}(X); {\cal T}_\infty(\mu, \mu)=0\}$ is non-empty and for every 
$\mu, \nu \in {\cal P}(X)$, we have 
\begin{equation}
{\cal T}_\infty(\mu, \nu)=\inf\{ {\cal T}_\infty(\mu, \sigma)+{\cal T}_\infty(\sigma, \nu), \sigma \in {\cal A}\}.
\end{equation}
\end{enumerate}
  \end{thm}
This will follow from the following more general result. But first, we mention that there is an analogous result for the case when $\T$ is a forward linear transfer with operator $T^+$. The same statements hold as above, the only difference being that  
\begin{equation}
\hbox{$\lim\limits_{n\to +\infty} \frac{(T^+)^nf(x)}{n}=c$\quad for every $f\in C(X)$ and $x\in X$, }
\end{equation}
and 
\begin{equation}
\hbox{$T^+\circ T^+_\infty f - c = T^+_\infty f$ \quad for all $f \in C(X)$.}
\end{equation}
If now $\T$ is simultaneously  a backward and forward transfer, then we have the following,

\begin{cor} Suppose $\T$ is a backward  and forward linear transfer on $\mcal{P}(X) \times \mcal{P}(X)$ that is  continuous for the Wasserstein metric, then the  associated effective transfer $\T_\infty$ is also a backward  and forward linear transfer on $\mcal{P}(X) \times \mcal{P}(X)$, with $T_\infty^-$ (resp., $T_\infty^+$) as corresponding backward (resp., forward) effective Kantorovich operator. Moreover, 
 The associated effective transfer $\T_\infty$ can be expressed as 
\begin{eqnarray}\label{KAM.duals}
{\mathcal T}_\infty(\mu, \nu)=\sup\big\{\int_{X}f^+\, d\nu-\int_{X}f^-\, d\mu;\, (f^-, f^+)\in {\cal I}\big\},
\end{eqnarray}
where 
$${\cal I}=\big\{(f^-, f^+); \hbox{ $f^-$ (resp., $f^+$) is a backward (resp., forward) solution and $f^-=f^+$ on ${\cal A}$\big\}.}
$$

\end{cor}
 
\noindent{\bf Proof:} Since $T_\infty^+$ and $T_\infty^-$ are the Kantorovich opeartors for $\T_\infty$, we can use  (\ref{tau-convex}) of Proposition \ref{few} to write 
\begin{eqnarray}
{\mathcal T}_\infty(\mu, \nu)&=&\sup\big\{\int_{X}T_\infty ^+\circ T_\infty^-g\, d\nu -\int_{X}T_\infty^-g\, d\mu;\,  g \in C(X)\big\}\\
&=&\sup\big\{\int_{X}T_\infty^+f\, d\nu-\int_{X}T_\infty^-\circ T_\infty^+f\, d\mu;\,  f \in C(X)\big\}.
\end{eqnarray}
Note now that $f^-=T^-_\infty g$ (resp., $f^+=T^+_\infty f^-$) is a backward (resp., forward) weak KAM solution for $\T$ and in view of Proposition \ref{Lips},  $\int_Xf^-\, d\mu=\int_Xf^+\, d\mu$ for every $\mu \in  {\cal A}$. It follows that 
\begin{eqnarray}
{\mathcal T}_\infty(\mu, \nu)\leq \sup\big\{\int_{X}f^+\, d\nu-\int_{X}f^-\, d\mu;\, (f^-, f^+)\in {\cal I}\big\},
\end{eqnarray}
 For the reverse inequality, note first that  if $f^-, f^+ \in {\cal I}$, then since $f^-=T^-f^-$ and $f^+=T^+f^+$, the functions $\mu \to \int_Xf^- d\mu$ and $\mu \to \int_Xf^+ d\mu$ are $\T_\infty$-Lipschitz on the set $\cal A$. Hence Theorem \ref{nice} applies and we get a function $\chi$ such that $T^-\chi \leq T^-f^-=f^-$ and $T^+\circ T^-\chi \geq T^+f^+=f^+$. This readily implies the reverse inequality, hence that (\ref{KAM.duals}) hold.

\subsection{Effective Kantorovich operator associated to a semi-group of linear transfers}

Let $\{\T_{t}\}_{t \geq 0}$ be a family of backward linear transfers on $\mcal{P}(X)\times \mcal{P}(X)$ with associated Kantorovich operators $\{T_t\}_{t \geq 0}$, where $\T_0$ is the identity transfer, 
\as{
\T_0(\mu,\nu) = 
\begin{cases}
0 & \text{if } \mu = \nu \in \mcal{P}(X)\\
+\infty & \text{otherwise.}
\end{cases}
}

We make the following assumptions:
\enum{
\item[(H0)] The family $\{\T_{t}\}_{t \geq 0}$  is a semi-group under inf-convolution: $\T_{t+s} = \T_t\star \T_s$ for all $s,t\geq 0$.
\item[(H1)] For every $t > 0$, the transfer $\T_t$ is weak$^*$-continuous, and the Dirac measures are contained in $D_1(\T_t)$.
\item[(H2)] For any $\epsilon > 0$, $\{\T_t\}_{t \geq \epsilon}$ has common modulus of continuity $\delta$ (possibly depending on $\epsilon$).
}

The hypotheses (H1) and (H2) amount to an equi-continuity assumption for the family $\{T_tf\}_{t \geq 0}$ for each $f$, and is an artifact to ensure that we remain within the class of continuous functions in the limit $t \to +\infty$ (thanks to Arzela-Ascoli). It is likely these hypotheses can be weakened. Note in relation to (H2) that the semi-group property (H0) implies that a modulus of continuity for $\T_{t}$ is also one for $\T_{Nt}$, $N \in \N$. In the following, where we will be concerned with taking limits as $t \to +\infty$, it suffices to take $\epsilon = 1$.

\begin{prop}\label{estimation}
Under condition $(H0)$, there exists a finite constant $c$ and a positive constant $C > 0$ such that
\eqs{
\lf|\T_{t}(\mu,\nu) - tc\rt| \leq C,\quad \text{for every $t \geq 1$ and all $\mu,\nu \in \mcal{P}(X)$}.
}
In particular,
\eqs{
c = \lim_{t \to + \infty} \frac{\inf\{\T_t(\mu,\nu)\,;\, \mu,\nu \in \mathcal{P}(X)\}}{t}.
}
\end{prop}
We shall call the constant $c(\T)$ in Proposition \ref{estimation} the \textit{Ma\~n\'e critical value}, while the solutions $u \in C(X)$ of the functional equation $T_tu +ct = u$ for all $t \geq 0$, will be called \textit{backward weak KAM solutions}.\\

\noindent{\bf Proof:} 
Define $M_t := \max_{\mu,\nu}\T_t(\mu,\nu)$ and $M := \inf_{t\geq 1}\{\frac{M_t}{t}\} > -\infty$. The sequence $\{M_t\}_{t \geq 1}$ is subadditive, that is $M_{t+s} \leq M_t +M_s$, hence it is well known (see e.g. \cite{Bowen}) that $\{\frac{M_t}{t}\}_{t \geq 1}$ decreases to its infimum $M$ as $t \to \infty$. Indeed, 
fix $t > 0$ and write for any $s$, the decomposition $s = nt + r$, where $0 \leq r < t$. The subadditivity of $M_t$ implies
\eqs{
\frac{M_s}{s}  = \frac{M_{nt+r}}{nt+r} \leq \frac{M_{nt}}{nt} + \frac{M_r}{nt} \leq \frac{M_t}{t} + \frac{M_r}{nt}.
}
It follows that $\limsup_{s \to \infty}\frac{M_s}{s} \leq \frac{M_t}{t}$. On the other hand, $\inf_{t \geq 1}\frac{M_t}{t} \leq \liminf_{t \to \infty}\frac{M_t}{t}$. Therefore, $\frac{M_t}{t}$ converges to $M$ as $t \to \infty$.\\
On the other hand, if $m_t := \min_{\mu,\nu}\T_t(\mu,\nu)$, then the above applied to $-m_t$ yields that $\lim_{t \to \infty}\frac{m_t}{t} = m$.\\
  We now show that $m = M$. The uniform modulus of continuity $\delta$ implies the existence of a constant $C > 0$, such that $M_t - m_t \leq C$ for every $t > 0$. Then, we obtain the string of inequalities
\eqs{
tM - C \leq M_t - C \leq m_t \leq \T_t(\mu,\nu) \leq M_t \leq m_t + C \leq tm + C.
}
The left-most and right-most inequalities imply $M \leq m$ upon sending $t \to \infty$, hence $m  = M$. \qed

From Property 1) of Kantorovich operators and Proposition \ref{estimation}, we can deduce the following.

\begin{lem}\label{Kant} Under conditions $(H0)$, $(H1)$, and $(H2)$, and with the notation of Proposition \ref{estimation}, the following properties hold.
\begin{enumerate}
\item For any $f \in C(X)$, we have $|T_t f(x) + c t - \sup_{X}f| \leq C$ for all $t \geq 1$ and all $x \in X$.
\item The semi-group of operators $\{T_t\}_{t \geq 1}$ has the same modulus of continuity $\delta$ as $\{\T_t\}_{t \geq 1}$.
\item If $k < c$, then $T_tf + kt \to -\infty$, while if $k > c$, $T_tf + kt \to +\infty$, as $t \to \infty$, for any $f \in C(X)$.
\end{enumerate}
 
\end{lem}

\noindent{\bf Proof:}  1) By Proposition \ref{estimation} and since $T_t f(x) + ct = \sup_{\sigma}\{\int f\d\sigma - \lf(\T_t(\delta_x, \sigma)-ct\rt)\}$, we have $\sup_{X}f - C \leq T_t f(x) + ct \leq \sup_{X}f + C$.

For 2) we note that 
\begin{align*}
T_t f(x) &= \sup_{\sigma}\{\int f\d\sigma - \T_t(\delta_x, \sigma)\} \\
&\leq \sup_{\sigma}\{\int f\d\sigma - \T_t(\delta_y, \sigma)\} + \sup_{\sigma}\{\T_t(\delta_y, \sigma) - \T_t(\delta_x, \sigma)\}\\
&= T_t f(y) + \delta(d(x,y)).
\end{align*}
We now interchange $x$ and $y$ to obtain the reverse inequality.

3) follows from 1) since $\sup_{X}f - C + (k-c)t \leq T_t f(x) + kt \leq \sup_{X}f + C + (k-c)t$.

\begin{thm} \label{weakKAMthm} Given a semi-group of backward linear transfers $(\T_t)_{t\geq 0}$ satisfying conditions $(H0)$, $(H1)$, and $(H2)$, there exist a backward linear transfer $\T_\infty$, an associated Kantorovich operator $T_\infty: C(X) \to C(X)$ and a constant $c \in \R$ such that: 
\begin{enumerate}
\item For every $f \in C(X)$, $T_\infty f$ is a backward weak KAM solution, and $T_\infty$ is idempotent. In particular, backward weak KAM solutions are fixed points of $T_\infty$.
\item The backward linear transfer $\T_\infty$  
satisfies,  
\begin{equation}
\hbox{$\T_\infty  = (\T_t - ct) * \T_\infty $ for every $t \geq 0$,\quad and \quad $\T_\infty = \T_\infty * \T_\infty$.}
\end{equation}

\item For every $\mu,\nu \in \mcal{P}(X)$, we have 
\begin{equation}
\sup\lf\{\int T_\infty f\d (\nu-\mu)\,;\, f \in C(X)\rt\} \leq\T_\infty(\mu,\nu) \leq \liminf_{t\to\infty}(\T_t(\mu,\nu)- ct).
\end{equation}
\item The set ${\cal A}:=\{\sigma \in {\cal P}(X); {\cal T}_\infty(\sigma, \sigma)=0\}$ is non-empty, and for every 
$\mu, \nu \in {\cal P}(X)$, we have 
\begin{equation}
{\cal T}_\infty(\mu, \nu)=\inf\{ {\cal T}_\infty(\mu, \sigma)+{\cal T}_\infty(\sigma, \nu), \sigma \in {\cal A}\},
\end{equation}
and the infimum on ${\cal A}$ is attained. 
\item We also have 
\begin{equation}\lbl{infimum}
c = \inf\{{\cal T}_1(\mu, \mu); \mu\in {\cal P}(X)\},
\end{equation}
and the infimum is attained by a measure $\bar{\mu} \in \mathcal{A}$ such that 
\begin{equation}
(\bar{\mu}, \bar{\mu}) \in \mathcal{D}:= \{ (\mu,\nu) \in \mathcal{P}(X)\times\mathcal{P}(X)\,:\, \T_1(\mu,\nu) + \T_\infty(\nu,\mu) = c\}.
\end{equation}
 Moreover, every measure which attains the infimum in \refn{infimum} belongs to $\mcal{A}$.
 \end{enumerate}
\end{thm}
The backward linear transfer $\T_\infty$ is an analog of the \textit{Peierls barrier}, and the set $\mcal{A}$ is an analog of the \textit{projected Aubry set}. \\

\noindent{\bf Proof:} 
%\begin{enumerate}
1)  Given $f \in C(X)$, define $\bar{T}f(x) := \limsup_{t \to \infty}(T_tf(x) +c t)$. By (H2), $\bar{T}f$ has modulus of continuity $\delta$, and $\|\bar{T}f\|_\infty \leq \sup_{X} f + C$.

Noting that $\sup_{s \geq t} \{T_s f(x) + cs\}$ is a sequence of continuous functions that decrease monotonically to $\bar{T}f(x)$ as $t \to \infty$, we may apply Lemma \ref{monotone} to deduce for any $t' \geq 0$,
\as{
T_{t'} \bar{T}f(x) &= \lim_{t \to \infty} T_{t'} \lf[\sup_{s \geq t}\{T_s f(x) + c s\}\rt]\\
&\geq \lim_{t \to \infty} \sup_{s \geq t}\{T_{t'+s}f(x) + c s\}\\
&=\lim_{t \to \infty} \sup_{s \geq t}\{T_{t'+s}f(x) + c(t'+s)\} -c t'\\
&= \bar{T}f(x) -c t'.
}
Therefore, $T_{t'} \bar{T}f(x) + c t' \geq \bar{T}f(x)$. By monotonicity of the operators $T_t$, this inequality implies 
\eqs{
T_{t}\bar{T}f(x) + ct \geq T_{s}\bar{T}f(x) + c s
}
whenever $t \geq s \geq 0$, i.e. $\{T_{t}\bar{T}f + ct\}_{t \geq 1}$ is a monotone increasing sequence of continuous functions. In addition, we have from Corollary \ref{Kant} the uniform in time bound
\begin{equation*}
\|T_{t}\bar{T}f(x) + ct\|_\infty \leq \|\bar{T}f\|_\infty + C \leq \|f\|_\infty + 2C.
\end{equation*}
We may therefore define $T_\infty : C(X) \to C(X)$ via the formula,
\eqs{
T_\infty f(x) := \lim_{t \to \infty}T_t\bar{T}f(x) + c t,
}
and from Lemma \ref{monotone} deduce
\as{
T_t T_\infty f(x) + ct &= \lim_{s \to \infty}T_t \lf[T_{s}\bar{T}f(x) + c s\rt] + ct\\
 &= \lim_{s \to \infty}\lf\{T_{t+s}\bar{T}f(x) + c (t+s)\rt\}\\
  &= T_\infty f(x).
}
This further implies that $T_\infty T_\infty f(x) = T_\infty f(x)$ so $T_\infty$ is idempotent. It is straightforward to see that in the construction of $T_\infty$, properties 1)-4) of Proposition \ref{basicK} are preserved, and hence $T_\infty$ is a Kantorovich operator.

Finally we note that if $u$ satisfies $T_t u + ct = u$, then $T_\infty u = u$ from the defintion of $T_\infty$.

\noindent 2) $T_\infty $ is a Kantorovich operator, thus we may define
\eqs{
\T_\infty(\mu,\nu) := \sup\lf\{\int f\d\nu - \int T_\infty f\d\mu\,;\, f \in C(X)\rt\}
}
and it is a backward linear transfer; from $T_t T_\infty f + c t = T_\infty f$, it satisfies
\eqs{
\T_\infty(\mu,\nu) = (\T_t - c t)\star\T_\infty(\mu,\nu),\quad \text{for all $t \geq  0$,}
}
and from $T_\infty T_\infty u(x) = T_\infty u(x)$, it satisfies
\eqs{
\T_\infty (\mu,\nu) = \T_\infty\star\T_\infty(\mu,\nu), \text{for all $\mu,\nu$.}
}
3)  Note from $1$ that $T_\infty f(x) \geq \limsup_{t \to \infty}(T_tf(x) +c t)$, so 
\begin{align*}
\int_{X} T_\infty f \d\mu &\geq \int_{X}\limsup_{t \to \infty}(T_tf(x) +c t)\d\mu\\
&\geq \limsup_{t \to \infty}\int_{X}(T_tf(x) +c t)\d\mu.
\end{align*}
Hence
\begin{align*}
\T_\infty(\mu,\nu) 
&\leq \sup\liminf_{t \to \infty}\left\{\int_{X}f\d\nu - \int_{X}T_t f\d\mu - ct\,;\, f \in C(X)\right\}\\
&\leq \liminf_{t \to \infty}\sup\left\{\int_{X}f\d\nu - \int_{X}T_t f\d\mu - c t\,;\, f \in C(X)\right\}\\
&= \liminf_{t \to \infty}(\T_t(\mu,\nu)- c t).
\end{align*}
On the other hand, from $T_\infty \circ T_\infty f = T_\infty f$,
\begin{align*}
\T_\infty(\mu,\nu) &= \sup\lf\{\int_{X}f\d\nu - \int_{X}T_\infty f\d\mu\,;\, f \in C(X)\rt\}\\
&\geq \sup\lf\{\int_{X}T_\infty f\d(\nu-\mu)\,;\, f \in C(X)\rt\}.
\end{align*}
4) The proof of this result relies solely on the property $\T_\infty = \T_\infty \star \T_\infty$, and the argument is a minor adaption of the one given in \cite{B-B1}.  

Fix $\mu, \nu \in \mathcal{P}(X)$. From $\T_\infty = \T_\infty \star \T_\infty$, there exists $\sigma_1 \in \mathcal{P}(X)$ such that 
\begin{equation*}
\T_\infty(\mu,\nu) = \T_\infty(\mu,\sigma_1) + \T_\infty(\sigma_1, \nu).
\end{equation*}
Similarly, there exists a $\sigma_2$ such that 
\begin{equation*}
\T_\infty(\sigma_1, \nu) = \T_\infty(\sigma_1, \sigma_2)+ \T_\infty(\sigma_2, \nu).
\end{equation*}
Combining the above two equalities, we obtain
\begin{equation*}
\T_\infty(\mu,\nu) = \T_\infty(\mu,\sigma_1) + \T_\infty(\sigma_1, \sigma_2)+ \T_\infty(\sigma_2, \nu).
\end{equation*}
Note also that
\begin{equation}\lbl{split}
\T_\infty(\mu,\sigma_1) + \T_\infty(\sigma_1, \sigma_2) = \T_\infty(\mu,\sigma_2).
\end{equation}
This follows from 
\begin{align*}
\T_\infty(\mu,\nu) &= \T_\infty(\mu,\sigma_1) + \T_\infty(\sigma_1, \sigma_2)+ \T_\infty(\sigma_2, \nu)\\
&\geq \T_\infty\star\T_\infty(\mu, \sigma_2) + \T_\infty(\sigma_2, \nu)\\
&= \T_\infty(\mu, \sigma_2) + \T_\infty(\sigma_2, \nu)\\
&\geq \T_\infty\star\T_\infty(\mu,\nu)\\
&= \T_\infty(\mu,\nu).
\end{align*}
Hence all the inequalities are equalities; in particular (\ref{split}).

After $k$ times we have
\begin{equation*}
\T_\infty(\mu,\nu) = \sum_{i = 0}^{k}\T_\infty(\sigma_{i},\sigma_{i+1}) 
\end{equation*}
where $\sigma_0 := \mu$ and $\sigma_{k+1} := \nu$. This inductively generates a sequence $\{\sigma_{k}\}$ with the property 
\begin{equation*}
\sum_{i = \ell}^{m} \T_\infty(\sigma_{i}, \sigma_{i+1}) =  \T_\infty(\sigma_{\ell}, \sigma_{m+1})
\end{equation*}
whenever $0 \leq \ell < m \leq k$. In particular, for any subsequence $\sigma_{k_j}$, we have
\begin{equation}\label{sebseq}
\T_\infty(\mu,\sigma_{k_1}) + \sum_{j = 1}^{m} \T_\infty(\sigma_{k_j}, \sigma_{k_{j+1}}) + \T_\infty(\sigma_{k_{m+1}}, \nu) =  \T_\infty(\mu, \nu).
\end{equation}

Extract a weak$^*$ convergent subsequence $\{\sigma_{k_j}\}$ to some $\bar{\sigma} \in \mathcal{P}(X)$. By weak-$*$ l.s.c. of $\T_\infty$, we have
\begin{equation*}
\liminf_{j}\T_\infty(\sigma_{k_j}, \sigma_{k_{j+1}}) \geq \T_\infty(\bar{\sigma},\bar{\sigma}).
\end{equation*}
In particular, given $\epsilon > 0$, for all but finitely many $j$,
\begin{equation}\label{biggerthan}
\T_\infty(\sigma_{k_j}, \sigma_{k_{j+1}}) \geq \T_\infty(\bar{\sigma},\bar{\sigma}) - \epsilon.
\end{equation}
Therefore, by refining to a further (non-relabeled) subsequence if necessary, we obtain a subsequence $\{\sigma_{k_j}\}$ satisfying (\ref{biggerthan}) for all $j$. By further refinement, we may also assume,
\begin{equation}\label{boundarybigger}
\T_\infty(\mu, \sigma_{k_1}) \geq \T_\infty(\mu, \bar{\sigma}) - \epsilon.
\end{equation}
Therefore, by refining to a further (non-relabeled) subsequence if necessary, we obtain a subsequence $\{\sigma_{k_j}\}$ with properties (\ref{sebseq}), (\ref{biggerthan}), and (\ref{boundarybigger}).

Moreover, for all $m$ large enough (depending on $\epsilon$), we have
\begin{equation}\label{tailend}
\T_\infty(\sigma_{k_{m+1}}, \nu) \geq \T_\infty(\bar{\sigma}, \nu) - \epsilon
\end{equation}

Applying the inequalities of (\ref{biggerthan}), (\ref{boundarybigger}), and (\ref{tailend}), to (\ref{sebseq}), we obtain
\begin{equation*}
\T_\infty(\mu,\nu) \geq \T_\infty(\mu,\bar{\sigma}) + m\T_\infty(\bar{\sigma},\bar{\sigma}) + \T_\infty(\bar{\sigma},\nu)- (m+2)\epsilon
\end{equation*}
for large enough $m$. From the fact that $\T_\infty = \T_\infty * \T_\infty$, the above inequality is only possible if 
\begin{equation*}
\T_\infty(\bar{\sigma},\bar{\sigma}) \leq \frac{m+2}{m}\epsilon \leq 2\epsilon.
\end{equation*}
As $\epsilon$ is arbitrary, we obtain $\T_\infty(\bar{\sigma},\bar{\sigma}) \leq 0$, and consequently $\T_\infty(\bar{\sigma},\bar{\sigma}) = 0$ (the reverse inequality following from $\T_\infty = \T_\infty \star \T_\infty$). 

Finally, we note that $\T_\infty (\mu,\nu) = \T_\infty(\mu,\sigma_{k_j}) + \T_\infty(\sigma_{k_j},\nu)$ for all $j$, so at the $\liminf$, we find
\begin{equation*}
\T_\infty(\mu,\nu) \geq \T_\infty(\mu,\bar{\sigma}) + \T_\infty(\bar{\sigma},\nu).
\end{equation*}
The reverse inequality is immediate from $\T_\infty = \T_\infty \star \T_\infty$.\\

\noindent 5) First, we observe that $\T_1(\mu,\mu) \geq c$ for all $\mu$. This follows from
\eqs{
c = \lim_{t \to \infty}\min_{\mu,\nu}\frac{\T_t(\mu,\nu)}{t} = \lim_{n \to \infty}\min_{\mu,\nu}\frac{\T_n(\mu,\nu)}{n} \leq \T_1(\mu,\mu).
}
To achieve the reverse inequality, we construct inductively a sequence $\{\mu_k\} \subset \mathcal{A}$ such that $(\mu_k,\mu_{k+1}) \in \mathcal{D}$. The set $\mathcal{D}$ is convex by convexity of both $\T_1$ and $\T_\infty$. Therefore, the Cesaro averages belong to $\mathcal{D}$,
\begin{equation*}
(\frac{1}{n}\sum_{k = 1}^{n}\mu_k,\frac{1}{n}\sum_{k = 1}^{n}\mu_{k+1}) \in \mathcal{D}.
\end{equation*}
Denoting $\nu_n := \frac{1}{n}\sum_{k = 1}^{n}\mu_k$, we have
\begin{equation}\lbl{projectedMatherequality}
\T_1(\nu_n, \nu_n + \frac{1}{n}(\mu_{n+1} - \mu_1)) + \T_\infty(\nu_n, \nu_n + \frac{1}{n}(\mu_{n+1} - \mu_1))= c.
\end{equation}

Extract a weak$^*$-convergent subsequence $\nu_{n_j}$ with limit $\bar{\mu} \in \mathcal{A}$. Then by weak-$*$ lower semi-continuity of $\T_1$ (resp. $\T_\infty$), \refn{projectedMatherequality} yields at the limit,
\eqs{
\T_1(\bar{\mu},\bar{\mu}) = \T_1(\bar{\mu},\bar{\mu}) + \T_\infty(\bar{\mu},\bar{\mu}) \leq  c.
}
Hence, $c = \T_1(\bar{\mu},\bar{\mu})$, and $(\bar{\mu}, \bar{\mu}) \in \mathcal{D}$. 

Conversely, if $\mu$ is a measure which realises $c = \T_1(\mu,\mu)$, then by Property $3$ and $4$, we have
\eqs{
0 \leq \T_\infty(\mu,\mu) \leq \liminf_{t \to \infty}(\T_t(\mu,\mu)-ct) \leq \liminf_{n \to \infty}(\T_n(\mu,\mu)-cn) \leq 0,
}
so $\mu \in \mcal{A}$.
%\end{enumerate}

Similar results hold with appropriate changes for forward linear transfers. %\qed

\subsection{Optimal transports corresponding to a semi-group of cost functionals}

We now identify the effective transfer and Kantorovich map associated to a semi-group of linear transfers given by mass transports.

\begin{prop} Suppose $c_t(x, y)$ is a semi-group of equicontinuous cost functions on a compact space $X\times X$, that is 
\begin{equation}
c_{t+s}(x, y)=c_t\star c_s (x, y):=\inf\{c_t(x, z)+c_s(z, y); z\in X\}, 
\end{equation}
and consider the associated optimal mass transports 
\begin{equation}
\T_t(\mu,\nu) = \inf\{\int_{X\times X} c_t(x,y)\d\pi(x,y)\,;\, \pi \in \mcal{K}(\mu,\nu)\}.  
\end{equation}
\begin{enumerate} 
\item The family $(\T_t)_t$ then forms a semi-group of linear transfers for the convolution operation i.e., $\T_{t+s}=\T_t\star \T_s$ for any $s, t \geq 0$ that  is equicontinuous on ${\mathcal P}(X)\times {\mathcal P}(X)$, hence one can associate its effective Transfer $\T_\infty$ and the corresponding Kantorovich operator $T_\infty$. 
\item Letting $c_\infty(x,y) := \liminf_{t \to \infty}c_t(x,y)$, then : 
 \begin{equation}
\T_\infty (\mu, \nu)=\T_{c_\infty} (\mu, \nu):=\inf\{\int_{X\times X}c_\infty(x,y)\d\pi(x,y)\,;\, \pi \in \mcal{K}(\mu,\nu)\},  
\end{equation}
\begin{equation*}
T^-_\infty f(x) = \sup\{f(y) - c_\infty (x,y)\,;\, y \in X\}\, \text{and }\, T^+_\infty f(y) = \inf\{f(x) + c_\infty (x,y)\,;\, x \in X\}
\end{equation*}
\item The following holds
\begin{equation}\label{mather1000}
c=\inf\{{\cal T}_1(\mu, \mu); \mu\in {\cal P}(X)\}= \min \{ \int_{X \times X}c_1(x,y)\d \pi ; \pi \in {\mathcal P} (X\times X), \pi_1=\pi_2\}
\end{equation}
\item The set ${\mathcal A}:=\{\sigma \in {\mathcal P}(X); {\T}_\infty (\sigma, \sigma)=0\}$ consists of those $\sigma \in {\mathcal P}(X)$ supported on the set $A=\{x\in X; c_\infty (x, x)=0\}$. 
\item The minimizing measures in (\ref{mather1000}) are all supported on the set 
$$D:= \{ (x,y) \in X \times X\,;\, c_1(x,y) + c_\infty(y,x) = c\}.$$

\end{enumerate} 
\end{prop} 
\noindent{\bf Proof:} The  Kantorovich operator for $\T_t$ is given by  $T_t f(x) = \sup\{ f(y) - c_t(x,y)\,;\, y \in X\}$ and as shown in Proposition \ref{lifting}, we have $\T_{s+t}=\T_{c_t\star c_s}=\T_{c_t}\star \T_{c_s}=\T_{t}\star \T_s$, and $T_{t+s}=T_t\circ T_s$ for every $s, t$. It remains to show that the effective Kantorovich map $T_\infty$ associated to $(T_t)_t$ is equal to $T_{c_\infty}f:=\sup\{ T f(y) - c_\infty(x,y)\,;\, y \in X\}$. For that, we first note that 

\begin{equation}
T_\infty f\geq \limsup_{t}T_t f(x) \geq \sup_{y}\{f(y) - c_\infty(x,y)\}= T_{c_\infty}f(x).
\end{equation}
On the other hand, let $y_n$ achieve the supremum for $T_n f(x) = \sup\{f(y) - c_n(x,y)\,;\, y \in X\}$, and let $(n_j)_j$ be a subsequence   such that $\lim_{j \to \infty}T_{n_j}f(x) = \limsup_{n}T_n f(x)$. By refining to a further subsequence, we may assume by compactness of $X$, that $y_{n_j} \to \bar{y}$ as $j \to \infty$. Then by equi-continuity of the $c_n$'s, we deduce that
\begin{equation}
\limsup_{n}T_n f(x) = \lim_{j \to \infty}T_{n_j}f(x) = f(\bar{y}) - \liminf_{j}c_{n_j}(x, \bar{y}).
\end{equation}
As $\liminf_{j}c_{n_j}(x, \bar{y}) \geq \liminf_{n}c_n (x,\bar{y}) = c_\infty(x, \bar{y})$, we obtain
\begin{equation}
\limsup_{n}T_n f(x) \leq f(\bar{y}) - c_\infty(x, \bar{y}) \leq \sup_{y}\{f(y) - c_\infty(x,y)=T_{c_\infty}f(x).
\end{equation}
It follows that $\limsup_{n}T_n f(x) = T_{c_\infty}f(x)$, and since the same happens for every sequence $(n_k)_k$ going to $\infty$, we deduce that 
$\limsup_{t}T_t f =T_{c_\infty}f$ for every $f\in C(X)$. 

Finally, we note that $T (\limsup_{t}T_t f)(x) = \limsup_{t}T_t f(x)$ thanks to the fact that $c_t\star c_\infty = c_\infty$. This implies that $T_\infty f(x) = \limsup_{t}T_t f(x) = \sup_{y}\{f(y) - c_\infty(x,y)\}$, so  that $T_\infty=T_{c_\infty}$ and $\T_\infty=\T_{c_\infty}.$

Properties (1), (2) and (3) follow then immediately. Properties (4) and (5) now follow from an adaptation of the results of Bernard-Buffoni \cite{B-B1}.

\subsection{Fathi-Mather weak KAM theory}

Let $L$ be a time-independent \textit{Tonelli Lagrangian} on a compact Riemanian manifold $M$,  
and consider  $\T_t$ to be the cost minimizing transport 
$$\T_t(\mu,\nu) = \inf\{\int_{M\times M} c_t(x,y)\d\pi(x,y)\,;\, \pi \in \mcal{K}(\mu,\nu)\},$$ 
where 
$$c_t(x,y) := \inf\{\int_{0}^{t}L(\gamma(s), \dot{\gamma}(s))\d s\,;\, \gamma \in C^1([0,t];M); \gamma (0)=x, \gamma (t)=y\}.$$
As mentioned in the introduction, 
the Lax-Oleinik semi-group $S_t^-$, $t > 0$ is defined by the formula
\eqs{
S_t^- u(x) := \inf\{ u(\gamma(0)) + \int_{0}^{t}L(\gamma(s), \dot{\gamma}(s))\d s\,;\, \gamma \in C^1([0,t]; M), \gamma(t) = x\},
}
and a function $u \in C(M)$ is said to be a \textit{negative weak KAM solution} if $S_t^-u - ct = u$ for all $t \geq 0$. \\
Another semigroup $S^+_t$ is defined in terms of $S_t^-$ via the formula $S_t^+ u = - \hat{S}_t^-(-u)$, where $\hat{S}_t^-$ is the Lax-Oleinik semi-group  of the Lagrangian $\hat{L}(x,v) := L(x,-v)$. It turns out that
\eqs{
S_t^+ u(x) = \sup\{ u(\gamma(t)) - \int_{0}^{t}L(\gamma(s), \dot{\gamma}(s))\d s\,;\, \gamma \in C^1([0,t]; M), \gamma(0) = x\}.
}
Analogous to the negative weak KAM solutions, \textit{positive weak KAM solutions} are those $u$ satisfying $S_t^+u +ct = u$ for all $t \geq 0$. The semi-groups $S_t^-$ and $S_t^+$ are intimately connected with Hamilton-Jacobi equations, and Aubry-Mather theory. 

\begin{thm} Under the above conditions on $L$, there exists a unique constant  $c \in \R$ such that the following hold:
\begin{enumerate}
\item {\rm (Fathi \cite{Fa})} There exists a function $u_-: M \to \R$ (resp. $u_+$) such that $S_t^-u_- - c t = u_-$ (resp. $S_t^+u_- + c t = u_-$) for each $t \geq 0$.
\item {\rm (Bernard-Buffoni \cite{B-B1})} Let $c_\infty(x,y) := \liminf_{t \to \infty}c_t(x,y)$ denote the \textit{Peierls barrier function}. The following duality then holds:
\eqs{
\inf\{\int_{M\times M}c_\infty(x,y)\d\pi(x,y)\,;\, \pi \in \mcal{K}(\mu,\nu)\} = \sup_{u_+, u_-}\{\int_{M}u_+\d\nu - \int_{M}u_-\d\mu\},
}
where the supremum ranges over all $u_+, u_- \in C(M)$ such that $u_+$ (resp. $u_-$) is a positive (resp. negative) weak KAM solution, and such that $u_+ = u_-$ on the set $\mcal{A} := \{x \in M\,;\, c_\infty(x,x) = 0\}$. Moreover, $c_\infty(x,y) = \min_{z \in \mcal{A}}\{c_\infty(x,z) + c_\infty(z,y)\}$.
\item {\rm (Bernard-Buffoni \cite{B-B2})} The constant $c$ satisfies
\eqs{
c = \min_{\pi} \int_{M \times M}c_1(x,y)\d \pi(x,y), 
}
where the minimum is taken over all $\pi \in \mcal{P}(M\times M)$ with equal first and second marginals. The minimizing measures are all supported on $\mcal{D} := \{ (x,y) \in M \times M\,;\, c_1(x,y) + c_\infty(y,x) = c\}$.
\item {\rm  (Mather \cite{Mat})} The constant $c = \inf_{m}\int_{TM}L(x,v)\d m(x,v)$ where the infimum is taken over all measures $m \in \mcal{P}(TM)$ which are invariant under the Euler-Lagrange flow (generated by $L$).
\item {\rm (Fathi \cite{Fa})} A continuous function $u: M \to \R$ is a viscosity solution of $H(x, \nabla u(x)) = c[0]$ if and only if it is Lipschitz and $u$ is a negative weak KAM solution (i.e. $T_t^-u + c[0]t = u$). In particular, the statement is false if $c[0]$ is replaced with any other constant $c$. 
\end{enumerate}
\end{thm}
In the language of transfers, the cost-minimizing transport is both a forward and backward linear transfer, with forward (resp. backward) Kantorovich operators given by $T^+_t f(x) = V_f(t,x)$  and $T^-_t g(y) =  W_g^t(0,y)$, where
\eqs{
V_f(t',x) = \inf\{ f(\gamma(0)) + \int_{0}^{t'} L(\gamma(s), \dot{\gamma}(s))\d s\,;\, \gamma \in C^1([0,t'), M), \gamma(t') = x\}
}
and $W_g^t(t',y)$ the value functional
\eqs{
W_g^t(t',y) = \sup\{ g(\gamma(t')) - \int_{t'}^{t} L(\gamma(s), \dot{\gamma}(s))\d s\,;\, \gamma \in C^1([0,t'), M), \gamma(0) = x\}.
}
Observe that $V_f(t,x) = S^- f(x)$, while $W_g^t(0,y) = S^+ g(y)$. Hence (with unfortunate signs), $T^+_t f = S^-_t f(x)$, while $T_t^- f(x) = S^+_t f(x)$. Note also the translation of terminology in this setting: Our backward weak KAM solutions are Fathi's positive weak KAM solutions, while the analogous forward weak KAM solutions are Fathi's negative weak KAM solutions.

One can proceed with the construction outlined above to construct the negative (resp. positive) weak KAM solutions as the image of the Kantorovich operators $T^+_\infty$ (resp. $T^-_\infty$), and they will be given by
\eqs{
T^-_\infty f(x) = \sup\{f(y) - c_\infty (x,y)\,;\, y \in M\}\quad \text{and}\quad T^+_\infty f(y) = \inf\{f(x) + c_\infty (x,y)\,;\, x \in M\}
}
where $c_\infty(x,y) := \liminf_{t \to \infty}c_t(x,y)$. 

The backward (resp. forward) generalised Peierls barrier associated to $T^-_\infty$ (resp. $T^+_\infty$) are the same and is the cost-minimizing transport with cost $c_\infty$,  which by duality we can write as
\eqs{
\inf\{\int_{M\times M}c_\infty(x,y)\d\pi(x,y)\,;\, \pi \in \mcal{K}(\mu,\nu)\}=\sup\{\int_{M}T^+_\infty f\d\nu - \int_{M}T^-_\infty\circ T^+_\infty f \d\mu\,;\, f\in C(M)\}.
}
It can be checked this is exactly the statement 2 in the above theorem. 

 \subsection{The Schr\"odinger semigroup}
Recall the Schr\"odinger bridge of Example 4.5. Let $M$ be a compact Riemannian manifold and fix some reference non-negative measure $R$ on path space $\Omega=C([0,\infty], M)$. Let $(X_t)_t$ be a random process on $M$ whose law is $R$, and denote by $R_{0t}$ the joint law of the initial position $X_0$ and the position $X_t$ at time $t$, that is $R_{0t}=(X_0, X_t)_\#R$. Assume $R$ is the reversible Kolmogorov continuous Markov process associated with the generator $\frac{1}{2}(\Delta -\nabla V\cdot \nabla)$ and the initial probability measure $m=e^{-V(x)}dx$ for some function $V$.

For probability measures $\mu$ and $\nu$ on $M$, define
\begin{equation}\label{schrotrans}
\T_{t}(\mu,\nu) := \inf\{ \int_{M} \mathcal{H}(r_t^x, \pi_x)d \mu(x)\,;\, \pi \in \mathcal{K}(\mu,\nu),\, d \pi(x,y) = d \mu( x) d\pi_x(y)\}
\end{equation}
where $d R_{0t}(x, y) = d m(x)d r_t^x(y)$ is the disintegration of $R_{0t}$ with respect to its initial measure $m$.

\begin{prop}
The collection $\{\T_t\}_{t \geq 0}$ is a semigroup of backward linear transfers with Kantorovich operators $T_t f(x) := \log S_t e^f (x)$ where $(S_t)_t$ is the semi-group associated to $R$; in particular, 
\begin{equation}\label{dualitySchro}
\T_t(\mu,\nu) = \sup\left\{\int_{M} f d\nu - \int_{M}\log S_t e^f d\mu\,;\, f \in C(M)\right\}.
\end{equation} 
The corresponding idempotent backward linear transfer is $\T_\infty(\mu,\nu) = \mathcal{H}(m, \nu)$, and its effective Kantorovich map is $T_\infty f(x) := \log S_\infty e^f$, where $S_\infty g  := \int g \d m$.
\end{prop}
\noindent{\bf Proof:} It is easy to see that for each $t$, $T_t$ is monotone, 1-Lipschitz and convex, and also satisfies $T_t(f + c) = T_t f + c$ for any constant $c$. It follows that $\T_{t, \mu}^*(f)=\int_MT_tf\, d\mu$ for each $t$ by Proposition \ref{basicK}. The semigroup property then follows from the semigroup $(S_t)_t$ and the property that $\T_t \star \T_s$ is a backward linear transfer with Kantorovich operator $T_t \circ T_s f(x) = \log S_tS_s e^f (x) = \log S_{s+t}e^f (x) = T_{t+s} f(x)$ by Proposition \ref{inf.tens}.

Now we remark that it is a standard property of the semigroup $(S_t)_t$ on a compact Riemannian manifold, that under suitable conditions on $V$, $S_t e^f \to S_\infty e^f$, uniformly on $M$, as $t \to \infty$, for any $f \in C(M)$. This immediately implies by definition of $T_t$, that $T_t f \to T_\infty f$ uniformly as $t \to \infty$ for any $f \in C(M)$. We then deduce from the $1$-Lipschitz property, that $T_t \circ T_\infty f(x) = T_\infty f(x)$. We conclude that $T_\infty$ is a Kantorovich operator from Theorem \ref{weakKAMthm}. Finally we see that $\T_\infty(\mu,\nu)$ is
\begin{align*}
\T_\infty(\mu,\nu) &:= \sup\{\int f\d\nu - \int T_\infty f\d\mu\,;\, f \in C(M) \}\\
&= \sup\{\int f\d\nu - \log \int e^f \d m\,;\, f \in C(M)\}\\
&= \mathcal{H}(m, \nu),
\end{align*}
(see  Section 9, for the last equality).

 \section{Weak KAM solutions for non-continuous transfers}
 
 We now deal with cases where  $\T$ is not necessarily weak$^*$-continuous on ${\cal M}(X)$. 
 \subsection{The case of non-continuous transfers with bounded oscillation} 
 
  We now consider situations where $\T$ is not equicontinuous, but there is some control on the oscillation of the transfers $\T^n$. 
 
\begin{lem} \label{limits}  Let $X$ be a compact space and let $\T$ be a backward linear transfer such that $\mathcal{D}_1(\T)$ contains the Dirac measures. Assume that   
\begin{equation}\label{inv}
\T(\mu_0, \mu_0) <+\infty \hbox{ for some $\mu_0 \in {\cal P} (X)$. }
\end{equation}
 Then, the following properties hold:
\begin{enumerate}
\item $c(\T):=  \sup_{n} \frac{\inf\{\T_n(\mu,\nu)\,;\, \mu,\nu \in \mcal{P}(X)\}}{n}\leq \inf\{\T(\mu,\mu)\,;\, \mu \in \mcal{P}(X)\} < +\infty$.
\item For each $f\in C(X)$ and $x\in X$, we have 
\begin{equation}
\limsup_n\frac{T^nf(x)}{n} \leq -c(\T).
\end{equation}
\item For each $f\in C(X)$ and $\mu \in {\cal P}(X)$, we have 
\begin{equation}
\liminf_n\frac{1}{n}\int_XT^nf \, d\mu \geq - \T(\mu, \mu).
\end{equation}

\item For each $n \in N$, we have
\begin{equation}
 \sup_{\mu \in {\cal P}(X)} \int (T^nf(x) +nc)\, d\mu  \geq \inf_{y\in X}f(y).
\end{equation}
\item If for some $K>0$, we have \begin{equation}\label{stronger}
\liminf_n \{\inf_{\mu \in {\cal P}(X)}\T^n (\mu, \mu)-\inf_{\mu, \nu \in {\cal P}(X)}\T^n (\mu, \nu)\} \leq K, 
\end{equation}
then, 
\begin{equation}
\sup_{x\in X}\liminf_n T^nf(x) +nc \leq \sup_{y\in X}f(y) +K.
\end{equation}
\end{enumerate}
\end{lem}
In the next section, we shall prove that actually, 
\[
c(\T)=\inf\{\T(\mu,\mu)\,;\, \mu \in \mcal{P}(X)\}.
\]
  \noindent{\bf Proof:} 1) Let $\T$ be a backward linear transfer and consider for each $n\in \N$, ${\cal T}_n= {\cal T}\star {\cal T}\star ....\star {\cal T}$ the backward linear transfer obtained by iterating its convolution $n$ times. The sequence $m_n:=\inf \{\T_n(\mu,\nu)\,;\, \mu,\nu \in \mathcal{P}(X)\}$ is superadditive, that is $m_{n+k}\geq m_n + m_k$ for all positive integers, $n, k$. Since 
$$\frac{m_n}{n} \leq \sup_n \frac{1}{n}\T_n(\mu_0, \mu_0) \leq \T(\mu_0, \mu_0)<+\infty, $$
it follows that there exists a number $c(\T) \in \R$ such that 
  \begin{equation}
 \lim_n\frac{m_n}{n} := \lim_n\frac{1}{n}\inf\left\{\T_n(\mu,\nu)\,;\, \mu,\nu \in \mathcal{P}(X)\right\}=\sup_n\inf\limits_{\mu, \nu}\frac{\T_n(\mu,\nu)}{n}=c(\T) <+\infty.
\end{equation}
2) follows from 1) since 
\begin{align*}
T^nf(x)&=\sup\{\int_Xf\ d\sigma -\T_n(x, \sigma)\,;\, \sigma \in \mcal{P}(X)\}\\
&\leq \sup f -\inf\{\T_n(x, \sigma)\,;\, \sigma \in \mcal{P}(X)\}\\
&\leq \sup f -\inf \{\T_n(\mu, \sigma)\,;\, \mu, \sigma \in \mcal{P}(X)\}.
\end{align*}
For 3) note that 
\begin{align*}
\int_XT^nf \d\mu &=\sup\{\int_Xf\ d\sigma -\T_n(\mu, \sigma)\,;\, \sigma \in \mcal{P}(X)\}\\
&\geq \int_X f \, d\mu -\T_n(\mu, \mu)\\
&\geq \int_X f \, d\mu -n\T (\mu, \mu).
\end{align*}
  4) Write 
   \begin{align*}
  \sup_{\mu \in {\cal P}(X)} \int (T^nf(x) +nc)\, d\mu  
 &=  \sup_{\mu\in \mcal{P}(X)}\sup_{\sigma \in \mcal{P}(X)} \{\int fd\sigma -\T_n(\mu, \sigma) +nc\}\\ 
      & = \inf_{X} f  - \inf_{\sigma, \mu} \T_n(\mu, \sigma) +nc\\ 
                          &\geq \inf_X f. 
  \end{align*}
The latter inequality follows from Lemma \ref{regularise} since $\inf_{ \mu,\sigma} \T_n(\mu, \sigma) \leq nc$. \\
For 5),
write 
 \begin{align*}
  \sup_{x\in X}\liminf_n T^nf(x) +nc &\leq \liminf_n  \sup_{x\in X} T^nf(x) +kc\\
   &=  \liminf_n \sup_{x\in X} \sup_{\sigma \in {\cal P}(X)}\{\int_X fd \sigma -\T^n(x, \sigma)  +nc \}\\
      &\leq   \liminf_n \sup_{x\in X} \sup_{\sigma \in {\cal P}(X)}\{\sup f -\T^k(x, \sigma)  +kc \}\\
      &=  \sup f-\limsup_n \inf_{x\in X} \inf_{\sigma \in {\cal P}(X)}\{\T^n(x, \sigma)  -nc \}\\
       &\leq  \sup f-\limsup_n \inf_{\mu,  \sigma \in {\cal P}(X)}\{\T^n(\mu, \sigma)  -nc \}\\
         &\leq  \sup f-\limsup_n \inf_{\mu,  \sigma \in {\cal P}(X)}\{\T^n(\mu, \sigma)  -nc \}\\
         &\leq  \sup f +K.
   \end{align*}
  
Now we can prove the following.
 \begin{thm}\label{gen} Suppose $\T$ is a backward linear transfer on $\mcal{P}(X) \times \mcal{P}(X)$ such that $\mathcal{D}_1(\T)$ contains all Dirac measures. Assume (\ref{inv}), (\ref{stronger}) and 
     \begin{equation}\label{T1b}
   \sup_{x\in X}\inf_{\sigma \in {\mathcal P}(X)} \T(x, \sigma) <+\infty. 
   \end{equation}
  If $T: C(X) \to USC(X)$, where $T$ is the backward Kantorovich operator associated to $\T$, then there exists $h\in USC_\sigma (X)$ such that $Th+c=h$ on $X$. 
  \end{thm}
 \noindent{\bf Proof:} Note that condition (\ref{T1b}) means that $T^if$ is bounded below for any $f\in C(X)$ and any $i \in N$. We distinguish two cases:\\
 
 \noindent {\bf Case 1:} Assume the following: 
 \begin{equation}
\hbox{There is $f\in C(X)$ so that $\forall x\in X$, there exists $n\in \N$ with $T^nf(x)+nc<f(x)$. }
\end{equation}
 Since $T^nf$ is in $USC(X)$, then for each $x\in X$, there exists $n\in \N$ such that $T^nf +nc<f$ on a neighborhood of $x$, and since $X$ is compact, there is a finite number $r$ of iterates of $T$ such that $\inf_{0\leq i \leq r}(T^if +ic) <f$.\\
  Set $g_r=\inf_{1\leq i \leq r}(T^if +ic)$ and note that 
 $g_r \in USC (X)$, $\inf g_r>-\infty$ because of (\ref{T1b}), and $g_r <f$. Note now that 
 \[
 Tg_r+c\leq \inf_{2\leq i \leq r+1}\{T^if +ic\}.
 \]
On the other hand, $Tg_r +c \leq Tf +c$, hence 
\[
 Tg_r+c\leq \inf_{1\leq i \leq r}\{T^if +ic\} =g_r.
 \]
It follows that the sequence $\{T^ng_r+nc\}_n$ is decreasing to some function $h \in USC(X)$. Note that $h\leq g_r$ hence is bounded above. 

Now we show that $h$ is proper, that is not identically $-\infty$. Indeed, if it was, then for every $x$, the sequence $g_n(x)=T^ng_r(x)+nc$ will be decreasing to $-\infty$. It follows that for each $x\in X$, there exists $i$, such that $g_i (x)< \inf g_r -1$, hence on a neighborhood of $x$ since $g_i$ is in $USC(X)$. By compactness and since the $(g_n)_n$ is decreasing, we get a function $g_N$ such that $g_N < \inf g_r -1$ on $X$. On the other hand, the preceding lemma yields that  $\sup_{\mu \in {\cal P}(X)} \int g_N\, d\mu  \geq \inf g_r$.  It follows that 
\[
\inf g_r  \leq \sup_{\mu \in {\cal P}(X)} \int g_N\, d\mu \leq \inf g_r -1,
\]
which is a contradiction, hence $h$ is proper.

Finally, note that by Lemma \ref{monotone}, we have 
$$Th+c=T(\lim_n T^ng_r+nc)+c=\lim_n T^{n+1}g_r+(n+1)c=h.$$

 \noindent {\bf Case 2:} We now assume that for any $f\in C(X)$, there exists $x\in X$ such that 
 \begin{equation}\label{case2}
 T^n f(x)+nc\geq f(x) \quad \hbox{for all $n\in \N.$}
 \end{equation}
 We now consider for each $f\in C(X)$, the function $\tilde f:=\liminf_nT^nf+nc$. It is clear that $\tilde f \in USC_\sigma (X)$, and by our assumption, there exists $x\in X$ such that  $\tilde f (x) \geq f(x)>-\infty$, and hence it is proper. On the other, we have by Lemma \ref{limits}, that $\sup_{x\in X} \tilde f \leq \sup_{x\in X} f(x)+K$.  Moreover, by Lemma \ref{monotone},  
 \[
 T\tilde f +c=T(\liminf_nT^nf+nc) +c \geq \liminf_n T^{n+1}f +(n+1)c=\tilde f.
 \]
It follows that the sequence $\{T^n\tilde f +nc\}_n$ is increasing to a function $h\in USC_\sigma (X)$. 
Note that $h\geq \tilde f$, hence it is proper. On the other hand, by Lemma \ref{limits}, we have $h\leq \sup \tilde f \leq \sup f +K<+\infty$ and we are done.

 \begin{cor} Let $X$ is a compact space and let $\T$ be a backward linear transfer such that $\mathcal{D}_1(\T)$ contains the Dirac measures. If $\T$ is bounded above on ${\cal P}(X)\times {\cal P}(X)$, then 
 \begin{equation}
\frac{\T_n(\mu, \nu)}{n} \to c \quad \hbox{uniformly on ${\cal P}(X)\times {\cal P}(X)$.}
\end{equation}
Moreover, there exists an idempotent operator $T_\infty: C(X)\to USC_\sigma(X)$ such that for each $f\in C(X)$,  $T_\infty f$ is a backward weak KAM solution. 
\end{cor}

 \noindent{\bf Proof:} Note that conditions (\ref{inv}) and (\ref{T1b}) are readily satisfied. To prove  (\ref{stronger}), one can easily see that for any $\mu, \nu \in {\cal P}(X)$,  
 \[
 \inf\limits_{{\cal P}\times {\cal P}} \T_n + 2\inf\limits_{{\cal P}\times {\cal P}}  \T \leq \T_{n+2}(\mu, \nu) \leq 2\sup\limits_{{\cal P}\times {\cal P}}  \T +\inf\limits_{{\cal P}\times {\cal P}}  \T_n, 
 \]
from which follows that 
\begin{align*}
\sup\limits_{{\cal P}\times {\cal P}}  \T_{n+2} -\inf\limits_{{\cal P}\times {\cal P}}  \T_{n+2} &\leq 2\sup\limits_{{\cal P}\times {\cal P}}  \T +\inf\limits_{{\cal P}\times {\cal P}}  \T_n- \inf\limits_{{\cal P}\times {\cal P}} \T_n - 2\inf\limits_{{\cal P}\times {\cal P}}  \T\\
&=2\sup\limits_{{\cal P}\times {\cal P}}  \T - 2\inf\limits_{{\cal P}\times {\cal P}}  \T\\
&=:K < \infty.
\end{align*}
Theorem \ref{gen} then applies to get a weak KAM solution $h$ in $USC_\sigma (X)$. 

Note now that since $h$ is bounded above and is proper, i.e., $h(x_0)>-\infty$ for some $x_0\in X$, we have
\[
h(x_0) -\sup\limits_{{\cal P}\times {\cal P}}\T_n  \leq T_nh(y)=\sup\limits_{\sigma \in {\cal P}}\{\int_Xh\ d\sigma -\T_n(y, \sigma\} \leq \sup\limits_X  h -\inf\limits_{{\cal P}\times {\cal P}} \T_n, 
\]
hence
\[
\inf\limits_{{\cal P}\times {\cal P}} \T_n -\sup\limits_X h    \leq nc-h(x_0) \leq  \sup\limits_{{\cal P}\times {\cal P}}\T_n -h(x_0) \leq \inf\limits_{{\cal P}\times {\cal P}} \T_n +K-h(x_0), 
\]
and 
\begin{equation}\label{up}
-K\leq \T_n(\mu, \nu)-nc \leq K+\sup\limits_X h -h(x_0),
\end{equation}
from which follows that 
\begin{equation}
\frac{\T_n(\mu, \nu)}{n} \to c \quad \hbox{uniformly on ${\cal P}(X)\times {\cal P}(X)$. }
\end{equation}
Note now that (\ref{up}) yields that for every $f\in C(X)$, there is $C>0$ such that %we have 
\begin{equation}
\|T^nf+nc\|_\infty \leq \|f\|_\infty +C,
\end{equation}
from which follows that $\hat T f:=\liminf_n T^nf+nc$ is bounded, belongs to $USC(X)$ and satisfies $T (\hat T f) + c \geq {\hat T}f$. The sequence 
$(T^n (\hat T f) + nc)_n$ is therefore increasing to a function $T_\infty f$ in $USC_\sigma (X)$ such that $T\circ T_\infty f +c=T_\infty f$. 
\qed

Here is another situation where we can obtain weak KAM solutions. It will be relevant for the stochastic Mather theory. 

\begin{prop}\label{hor1} Suppose $\T$ is a backward linear transfer on $\mcal{P}(X) \times \mcal{P}(X)$ such that $\mathcal{D}_1(\T)$ contains all Dirac measures and that (\ref{inv}), (\ref{T1b}) hold. If there exists $u, v \in USC(X)$ that are bounded below such that  
\begin{equation}\label{hor}
T^nu +nv=u \quad \hbox{ for all $n\in \N$},
\end{equation}
  then there exists $h\in USC(X)$ such that $Th+c=h$ on $X$, where $T$ is the backward Kantorovich operator associated to $\T$.
  %\end{enumerate}
\end{prop} 
\noindent{\bf Proof:} Note that (\ref{hor}) and  Lemma  \ref{limits} yield that necessarily $-v(x) \leq -c$, from which follows that 
\[
T^nu +nc \leq u \quad \hbox{ for all $n\in \N$}. 
\]
Applying $T^m$ and using the linearity of $T^m$ with respect to constants, we find $T^{m+n}u + cn \leq T^mu$, and hence
\eqs{
T^{m+n}u + c(m+n) \leq T^mu + cm
}
So $n \mapsto T^nu + cn$ is decreasing. The same reasoning as in Case (1) of the proof of Theorem \ref{gen} yields that $(T^nu + cn)_n$ decreases to a proper function $h\in USC (X)$ such that $Th+c=h$ on $X$.

\subsection{Weak KAM solutions associated to non-continuous optimal mass transports}

The following extends a result established by Bernard-Buffoni \cite{B-B2} in the case where the cost function $c(x, y)$ is continuous.
\begin{cor} Let $c$ be a bounded lower semi-continuous cost functional on a compact space $X$ and consider the associated optimal mass transport 
\begin{equation}
\T(\mu,\nu) = \inf\{\int_{X\times X} c(x,y)\d\pi(x,y)\,;\, \pi \in \mcal{K}(\mu,\nu)\}. 
\end{equation}
 Let $c_\infty(x,y) := \liminf_{n \to \infty}c_n(x,y)$, where for each $n\in \N$, 
\[
c_n(y, x)=\inf\left\{c(y, x_1) +c(x_1, x_2) ....+c(x_{n-1}, x); \, x_1, x_2, x_{n-1}\in X\right\}. 
\]
Then, the following hold: 
\begin{enumerate}
\item The corresponding effective transfer is given by
\begin{equation}
\T_\infty (\mu, \nu)=\T_{c_\infty} (\mu, \nu):=\inf\{\int_{X\times X}c_\infty(x,y)\d\pi(x,y)\,;\, \pi \in \mcal{K}(\mu,\nu)\},  
\end{equation}
and the associated effective Kantorovich maps are given by
\begin{equation}
T^-_\infty f(x) = \sup\{f(y) - c_\infty (x,y)\,;\, y \in X\}\, \text{and}\, \, T^+_\infty f(y) = \inf\{f(x) + c_\infty (x,y)\,;\, x \in X\}.
\end{equation}
\item The following also holds
\begin{equation}\label{mather100}
c(\T)=\inf\{{\cal T}(\mu, \mu); \mu\in {\cal P}(X)\}= \min_{\mu\in {\cal P}(X)} \int_{X \times X}c(x,y)\d \pi(x,y); \pi \in \mcal{K}(\mu,\nu)\}.. 
\end{equation}
\item The set ${\mathcal A}:=\{\sigma \in {\mathcal P}(X); {\T}_\infty (\sigma, \sigma)=0\}$ consists of those $\sigma \in {\mathcal P}(X)$ supported on the set $A=\{x\in X; c_\infty (x, x)=0\}$. 
\item The minimizing measures in (\ref{mather100}) are all supported on the set 
$$D:= \{ (x,y) \in X \times X\,;\, c (x,y) + c_\infty(y,x) = c(\T)\}.$$

\end{enumerate} 
\end{cor} 

\noindent {\bf Example 7.1: Iterates of power costs:}  Let $c_p(x,y) = |x-y|^p$ for $p > 0$, then, 
$c_p\star c_p(x,y) = \inf\{ |x-z|^p + |z-y|^p\,;\, z \in X\}$ is minimised at some point $z = (1-\lambda)x + \lambda y$ on the line between $x$ and $y$, so that $c_p\star c_p(x,y) = \left(\lambda^p + (1-\lambda)^p\right)|x-y|^p$. For $p > 1$, the optimal $\lambda$ is $\frac{1}{2}$. Hence, by considering $\T_p$ to be the optimal mass transport associated to $c_p$ with its corresponding Kantorovich operator, $T_p f(x) = \sup\{f(y) - c_p(x,y)\,;\, y \in X\}$, we then  have 
$$(T_p)^n f(x) =\sup\{ f(y) - \frac{1}{n^{p-1}}|x-y|^p\}.$$ 
Hence when $n \to \infty$, $(T_p)^n f(x) \to \sup_{x}f(x)=T_\infty f (x)$   
and $c(\T_p)=0$. 

\section{Regularizations of linear transfers and applications}

We continue to deal with cases where  $\T$ is not necessarily weak$^*$-continuous on ${\cal M}(X)$ and may even have infinite values. 
The strategy now is to reduce the situation to the  bounded and continuous case via a regularization procedure.  
\subsection{Regularization and weak KAM solutions for unbounded transfers}

  \begin{lem}[Regularisation of a backward linear transfer]\label{regularise} Let $(X, d)$ be a complete metric space and let ${\cal W}_{d}(\mu,\nu)$ be the cost minimising optimal transport associated to the cost $d(x,y)$. 
For a given  backward linear transfer $\T: \mcal{P}(X)\times \mcal{P}(X) \to \R \cup \{+\infty\}$, we associate for each $\epsilon > 0$ the functional \eqs{
\T_\epsilon(\mu,\nu) := \inf\{ \frac{1}{\epsilon}W_1(\mu,\sigma_1) + \T(\sigma_1, \sigma_2) + \frac{1}{\epsilon}W_1(\sigma_2, \nu)\,;\, \sigma_1, \sigma_2 \in \mcal{P}(X)\}.
}
Then, $\T_\epsilon$ has the following properties:
\begin{enumerate}

\item $\T_\epsilon$ is a weak$^*$ continuous backward linear transfer.
\item $\inf\{\T_\epsilon(\mu,\nu)\,;\, \mu, \nu \in \mcal{P}(X)\} = \inf\{\T(\mu,\nu)\,;\, \mu, \nu \in \mcal{P}(X)\}$.
\item $\T_\epsilon(\mu,\nu) \leq \T(\mu,\nu)$ and $\T_\epsilon(\mu,\nu) \uparrow \T(\mu,\nu)$ as $\epsilon \to 0$.
\item $\T_\epsilon$ $\Gamma$-converges to $\T$ as $\epsilon \to 0$. 
\item If $T_\epsilon$, $T$, denote the backward Kantorovich operators associated to $\T_\epsilon$, $\T$, respectively, then  for any $f \in USC(X)$, $T_{\epsilon} f(x) \searrow T f(x)$ as $\epsilon \to 0$ .

\end{enumerate}
\end{lem}
\prf{ First note that since $d$ is continuous, the linear transfer ${\cal W}_{d}$ is weak-$*$ continuous on ${\mathcal P}(X)$ (See e.g., \cite{Sant}, Theorem 1.51, p.40).

1. We know that for each fixed $\epsilon > 0$, $\T_\epsilon$ is a weak$^*$ lower semi-continuous linear backward transfer. To prove that it is continuous, 
assume $\mu_n \to \mu$ and $\nu_n \to \nu$. By the lower semi-continuity, we have 
$\liminf_{n}\T_\epsilon(\mu_n,\nu_n) \geq \T_\epsilon(\mu,\nu)$. On the other hand, from the fact that $\limsup_{n} \inf_{\sigma_1,\sigma_2} \leq \inf_{\sigma_1,\sigma_2} \limsup_n$, we have 
\as{
\limsup_{n}\T_\epsilon(\mu_n,\nu_n) &\leq \inf\{ \limsup_{n}\frac{1}{\epsilon}W_1(\mu_n,\sigma_1) + \T(\sigma_1, \sigma_2) + \limsup_{n}\frac{1}{\epsilon}W_1(\sigma_2, \nu_n)\,;\, \sigma_1, \sigma_2 \in \mcal{P}(X)\}\\
&= \inf\lf\{ \frac{1}{\epsilon}W_1(\mu,\sigma_1) + \T(\sigma_1, \sigma_2) + \frac{1}{\epsilon}W_1(\sigma_2, \nu)\,;\, \sigma_1, \sigma_2 \in \mcal{P}(X)\rt\}\\
&= \T_\epsilon(\mu,\nu),
}
which shows that $\T_\epsilon(\mu_n,\nu_n) \to \T_\epsilon(\mu,\nu)$ as $n \to \infty$.

2. Observe from the definition of $\T_\epsilon$, that
\begin{equation}
\inf_{\mu,\nu}\{\T_\epsilon(\mu,\nu)\} = \inf_{\sigma, \sigma', \mu, \nu}\{\frac{1}{\epsilon}W_1(\mu, \sigma) + \T(\sigma, \sigma') + \frac{1}{\epsilon}W_1(\sigma', \nu)\}
\end{equation}
and it is clear that for fixed $\sigma, \sigma'$, the minimal $\mu$ is $\sigma$, and $\nu$ is $\sigma'$, and the transport cost $W_1(\mu,\sigma) = 0 = W_1(\sigma', \nu)$.

3. The inequality $\T_\epsilon(\mu,\nu) \leq \T(\mu,\nu)$ holds by selecting $\sigma_1 = \mu$ and $\sigma_2 = \nu$ and noting that ${\cal W}_{d}(\sigma, \sigma)=0$ for every $\sigma \in {\cal P}(X)$. The monotone property of $\epsilon \mapsto \T_\epsilon(\mu,\nu)$ is immediate by definition.
Let now $\sigma^\epsilon_1$, $\sigma^\epsilon_2$ realise the infimum
\begin{equation}\label{infachieve}
\T_{\epsilon}(\mu,\nu) = \frac{1}{\epsilon}W_1(\mu,\sigma^1_\epsilon) + \T(\sigma^1_\epsilon, \sigma^2_\epsilon) + \frac{1}{\epsilon}W_1(\sigma^2_\epsilon, \nu).
\end{equation}
By refining if necessary, we may assume that $\sigma^1_\epsilon \to \overline{\sigma}_1$ and $\sigma^2_\epsilon \to \overline{\sigma}_2$ as $\epsilon \to 0$. If $\sup_{\epsilon > 0}\T_\epsilon(\mu,\nu) < \infty)$, then $W_1(\mu,\sigma^1_\epsilon) \to 0$ and $W_1(\sigma^2_\epsilon, \nu) \to 0$ as $\epsilon \to 0$, hence $\overline{\sigma}_1 = \mu$ and $\overline{\sigma}_2 = \nu$. Then (\ref{infachieve}) and weak-$*$ lower semi-continuity of $\T$ implies
\eqs{
\liminf_{\epsilon \to 0}\T_\epsilon(\mu,\nu) \geq \liminf_{\epsilon \to 0}\T(\sigma^1_\epsilon, \sigma^2_\epsilon) \geq \T(\mu,\nu).
}

4. First recall that for $\Gamma$-convergence, one needs to prove the $\Gamma$-$\liminf$ inequality: For every sequence $(\mu^\epsilon, \nu^\epsilon) \to (\mu,\nu)$, it holds that $\liminf_{\epsilon \to 0}\T_\epsilon(\mu^\epsilon, \nu^\epsilon) \geq \T(\mu,\nu)$, and the $\Gamma$-$\limsup$ inequality: There exists a sequence $(\mu^\epsilon, \nu^\epsilon) \to (\mu,\nu)$ such that $\limsup_{\epsilon \to 0}\T_\epsilon(\mu^\epsilon, \nu^\epsilon) \leq \T(\mu,\nu)$. \\
The $\Gamma$-$\limsup$ inequality is immediate: Take $(\mu^\epsilon, \nu^\epsilon) = (\mu,\nu)$, and the inequality follows from $\T_\epsilon \leq \T$.\\
 For the $\Gamma$-$\liminf$ inequality, we can assume without loss that $\liminf_{\epsilon \to 0}\T_{\epsilon}(\mu^\epsilon,\nu^\epsilon) < +\infty$, since otherwise there is nothing to prove. Now by monotonicity, we have $\T_{\epsilon}(\mu^\epsilon, \nu^\epsilon) \geq \T_{\epsilon'}(\mu^\epsilon, \nu^\epsilon)$ for $\epsilon \leq \epsilon'$. The weak-$*$ lower semi-continuity of $\T_{\epsilon'}$ therefore implies
\as{
\liminf_{\epsilon \to 0}\T_{\epsilon}(\mu^\epsilon, \nu^\epsilon) \geq \liminf_{\epsilon \to 0}\T_{\epsilon'}(\mu^\epsilon, \nu^\epsilon)\geq \T_{\epsilon'}(\mu,\nu).
}
By 3) and letting $\epsilon' \to 0$, we obtain $\liminf_{\epsilon \to 0}\T_{\epsilon}(\mu^\epsilon, \nu^\epsilon) \geq \T(\mu,\nu)$.

5. First note that the monotonicity of $T_\epsilon f(x)$ is immediate from the expression 
$$T_\epsilon f(x) = \sup\lf\{\int f\d\sigma - \T_\epsilon(\delta_x,\sigma)\rt\},$$ 
and the monotonicity of $\T_\epsilon$. We immediately have $\liminf_{\epsilon \to 0}T_\epsilon f(x) \geq T f(x)$. On the other hand, let $\epsilon_j$ be a sequence such that $T_{\epsilon_j}f(x) \to \limsup_{\epsilon \to 0}T_\epsilon f(x)$. Then
\as{
T_{\epsilon_j} f(x) = \sup_{\sigma}\lf\{\int f\d\sigma - \T_\epsilon(\delta_x, \sigma)\rt\}
= \int f\d\sigma^{\epsilon_j} - \T_{\epsilon_j}(\delta_x, \sigma^{\epsilon_j}).
}
By refining to a further subsequence if necessary, we may assume $\sigma^{\epsilon_j} \to \sigma^*$. Then we obtain with $j \to \infty$,
\as{
\limsup_{\epsilon \to 0}T_\epsilon f(x) &\leq \int f \d\sigma^* - \liminf_{j \to \infty}\T_{\epsilon_j}(\delta_x, \sigma^{\epsilon_j})\\
 &\leq \int f \d\sigma^* - \T(\delta_x, \sigma^*)\\
&\leq \sup_{\sigma}\lf\{\int f\d\sigma - \T(\delta_x, \sigma)\rt\} = T f(x),
}
where the second inequality was obtained from the $\Gamma$-convergence.
}
\begin{lem}\label{c_epsilon}
Let $X$ be a compact metric space and let $\T$ be a backward linear transfer such that $\mathcal{D}_1(\T)$ contains the Dirac measures. Assume hypothesis (\ref{inv}) and 
 let $\T_\epsilon$ be the regularisation of $\T$ according to Lemma \ref{regularise}. Then, the following properties hold:
\begin{enumerate}
\item $c(\T_\epsilon)$  
 is the unique constant such that $|(\T_\epsilon)_n(\mu,\nu) - nc_\epsilon| \leq C_\epsilon$, for all $n$ and all $\mu,\nu$.
\item $c(\T_\epsilon) \uparrow c(\T)$ as $\epsilon \to 0$.
\item $c(\T)=  
\inf\{\T(\mu,\mu)\,;\, \mu \in \mcal{P}(X)\}$.
\end{enumerate}
\end{lem}
\noindent{\bf Proof:} 
 Use Lemma \ref{regularise} to regularise $\T$ to $\T_\epsilon$ , and let $\delta_\epsilon$ be the modulus of continuity for $\T_\epsilon$, which is also    the modulus of continuity for $(\T_\epsilon)_n$, the $n$-fold inf-convolution of $\T_\epsilon$. 
 Use now Corollary \ref{fixedpointcontinuous} for each $\epsilon$ to find $c_\epsilon = c (\T_\epsilon)$ with the properties stated there.  
Note,  in particular that $c (\T_\epsilon) \leq c(\T)$.
It follows that  $c (\T_\epsilon)$ converges  as $\epsilon \to 0$.  We let $K (\T)$ be this limit. Note that  $K(\T) \leq c(\T)$. We shall prove that 
\begin{equation}\label{constants}
c(\T)=\inf\{{\cal T}(\mu, \mu); \mu\in {\cal P}(X)\}=K(\T).
\end{equation}
This follows from the $\Gamma$-convergence, since 
\[
c (\T_\epsilon)=\inf \{\T_\epsilon(\mu, \mu); \mu \in {\cal P}(X)\}=\T_\epsilon (\mu_\epsilon, \mu_\epsilon)  \]
for some $\mu_\epsilon$, then if $\bar \mu$ is a cluster point for $(\mu_\epsilon)$ as $\epsilon \to 0$, the $\Gamma$-convergence of $\T_\epsilon$ implies that $c (\T_\epsilon)=\T_\epsilon (\mu_\epsilon, \mu_\epsilon) \to \T(\bar \mu, \bar \mu)$. If now $\nu$ is any other probability measure, then 
$\T (\nu, \nu) \geq \T_\epsilon (\nu, \nu) \geq  \T_\epsilon (\mu_\epsilon, \mu_\epsilon)$, hence  $\T (\nu, \nu) \geq \T ( \bar \mu, \bar \mu) =K(\T)$ and $$K(\T)=\inf\{{\cal T}(\mu, \mu); \mu\in {\cal P}(X)\}.$$  
On the other hand, for every $\mu$, 
\[
c(\T)=\sup_n\inf\limits_{\sigma, \nu}\frac{\T_n(\sigma,\nu)}{n}\leq \sup_n\frac{\T_n(\mu,\mu)}{n}\leq \T(\mu, \mu),
\]
since $\T_n(\mu,\mu)$ is subadditive on the diagonal, hence $c(\T) \leq K(\T)$ and (\ref{constants}) follows.  \qed

The above theorem has the following useful corollary, which implies the uniqueness of the level $c$, where weak KAM solutions occur. 

\begin{cor}\label{useful} Suppose $\T$ is a backward linear transfer on $\mcal{P}(X) \times \mcal{P}(X)$ such that $\mathcal{D}_1(\T)$ contains all Dirac measures and that (\ref{inv}), (\ref{T1b}) hold. Then,
\begin{enumerate}
\item If $Tu+d  \leq u$ for some $d\in \R$ and some $u\in USC(X)$, then $d\leq c$. 
\item If $Tv+d  \geq v$ for some $d\in \R$ and some $v\in USC(X)$, then $d\geq c$. 
\end{enumerate}
\end{cor}
\noindent{\bf Proof:} If $Tu+d  \leq u$ for some $d\in \R$ and $u\in USC(X)$, then $T^nu +nd \leq u$ for all $n\in \N$. 
Applying $T^m$ and using the linearity of $T^m$ with respect to constants, we find $T^{m+n}u + dn \leq T^mu$, and hence
\eqs{
T^{m+n}u + d(m+n) \leq T^mu + dm
}
So $n \mapsto T^nu + dn$ is decreasing to a function $\tilde u$. But if $d>c$, then $T^nu + dn \geq T^nu + cn$ and $\tilde u$ is proper by the first part of the proof of Theorem \ref{gen} and $T{\tilde u}+d={\tilde u}$ on $X$. It follows that for any $\mu \in {\cal P}(X)$, 
\[
\int_X T^n{\tilde u}\, d\mu = -nd +\int_X {\tilde u}\, d\mu.
\]
On the other hand, let $\bar \mu$ be such that $\T(\bar \mu, \bar \mu)=\inf_{\mu \in {\cal P}(X)}\T(\mu, \mu)=c$.
Then, from Lemma \ref{limits} we have 
\[
 \liminf_n\frac{1}{n}\int_XT^nf \, d {\bar \mu} \geq - \T({\bar \mu}, {\bar \mu})=-c.
\]
It follows that $-c \leq -d$, which is a contradiction. \qed

 Here is a case where we can associate to $\T$  an effective Kantorovich operator without the equicontinuity assumption.

\begin{thm}\label{fixedpointgeneral}
Let $X$ is a compact metric space and let $\T$ be a backward linear transfer such that $\mathcal{D}_1(\T)$ contains the Dirac measures. Assume (\ref{inv}) and 
that for some $\epsilon>0$, we have 
\begin{equation}\label{cepsilon}
c(\T_\epsilon) = c(\T),
\end{equation}
where $\T_\epsilon$ be the regularisation of $\T$ according to Lemma \ref{regularise}.
 Then, there exists an idempotent backward linear transfer $\T_\infty$ on $\mcal{P}(X) \times \mcal{P}(X)$, with a Kantorovich operator $T^\infty: C(X) \to USC (X)$  such that 
$T\circ T_\infty f +c= T_\infty f $ for all $f \in C(X)$.

\end{thm}
  \noindent{\bf Proof:}  
Consider the regularisation $\T_\epsilon$ of $\T$. By Corollary \ref{fixedpointcontinuous}, there exists a Kantorovich operator $T^\infty_\epsilon: C(X) \to C(X)$, such that $T_\epsilon \circ T^\infty_\epsilon f + c_\epsilon = T^\infty_\epsilon f$ for all $f \in C(X)$, and an idempotent transfer $\T_{\epsilon, \infty}$.
  We have the following properties: Under the assumption that $c_\epsilon = c$,

\begin{enumerate}
\item  $T_\epsilon^\infty f \leq T_{\epsilon'}^\infty f$ for all $f$, whenever $\epsilon < \epsilon'$.

\item There exists a $\bar{\mu} \in \mcal{P}(X)$ such that $\T(\bar{\mu}, \bar{\mu}) = c$,  and $\int_X T^\infty f \d\bar{\mu} \geq \int_X f \d\bar{\mu}$. 
\end{enumerate} 

To see that property 1 holds, observe that from monotonicity in $\epsilon$ for $T_\epsilon$ (see Lemma \ref{regularise}), we obtain monotonicity in $\epsilon$ for $\bar{T}_\epsilon f(x) := \limsup_{n} (T^n_\epsilon f(x) + n c_\epsilon)$ under the assumption that $c_\epsilon = c$. Hence by definition of $T^\infty_\epsilon f(x) = \lim_{n \to \infty}(T^n_\epsilon \circ \bar{T}_\epsilon f(x) + nc_\epsilon)$, we deduce monotonicity for $T^\infty_\epsilon$.

For Property 2, let $\bar{\mu}_\epsilon$ achieve $c_\epsilon = \T_\epsilon(\bar{\mu}_\epsilon,\bar{\mu}_\epsilon)$. By Theorem \ref{weakKAMthm}, we have $\T_{\epsilon, \infty}(\bar{\mu}_\epsilon,\bar{\mu}_\epsilon) = 0$, which implies
$\int_X f \d\bar{\mu}_\epsilon \leq \int_X T^\infty_\epsilon f \d\bar{\mu}_\epsilon.$
 
 On the other hand, extract a subsequence $\epsilon_j$ of the $\bar{\mu}_\epsilon$ so that $\bar{\mu}_{\epsilon_j} \to \bar{\mu}$. Then for any $\epsilon > 0$, eventually, $\epsilon_j < \epsilon$. It then follows the monotonicity of Property 1 that $\int_X T^\infty_{\epsilon_j} f \d\bar{\mu}_{\epsilon_j} \leq \int_X T^\infty_{\epsilon} f \d\bar{\mu}_{\epsilon_j}$. Let $j \to \infty$ to obtain
\begin{equation*}
\int_X f \d\bar{\mu}\leq \int_X T^\infty_{\epsilon} f \d\bar{\mu}.
\end{equation*}

The monotonicity of $T^\infty_\epsilon f$ and the above lower bound ensures that for $\bar{\mu}$-a.e. $x$, the limit $\lim_{\epsilon \to 0}T_\epsilon^\infty f(x)$ exists as a real number and is not $-\infty$. In particular, we deduce that $T^\infty f(x) := \lim_{\epsilon \to 0} T_{\epsilon}^{\infty}  f(x)$, satisfies $\int_X T^\infty f\d\bar{\mu}\geq \int_X f\d\bar{\mu}$; in particular, it is not identically $-\infty$, and belongs to $USC(X)$. 
 
By Lemma \ref{monotone}, we have
\[
T \circ T^\infty f(x) + c = \lim_{\epsilon \to 0}T \circ T_{\epsilon}^{\infty}  f(x) +c_\epsilon \leq  \lim_{\epsilon \to 0}T_\epsilon \circ T_{\epsilon}^{\infty} f(x) +c_\epsilon=\lim_{\epsilon \to 0}T_{\epsilon}^{\infty}  f(x) =T^\infty f .
\]  
On the other hand, the monotonicity in $\epsilon$ gives 
\eqs{
T_{\epsilon}^{\infty}  f = T_\epsilon \circ T_{\epsilon}^{\infty} f+c_\epsilon  \leq T_{\epsilon' }\circ T_{\epsilon}^{\infty}  f+c
}
 for any $\epsilon' > \epsilon$. By Lemma \ref{monotone} applied to $T_{\epsilon'}$ and the sequence $T_{\epsilon}^{\infty}  f$, we can pass the limit in $\epsilon$ through $T_\epsilon'$ to obtain
\as{
T^{\infty} f(x) = \lim_{\epsilon \to 0}T_{\epsilon}^{\infty}  f(x) \leq \lim_{\epsilon \to 0}T_{\epsilon'} \circ T_{\epsilon}^{\infty}  f(x)+c= T_{\epsilon'} \circ T^\infty f(x)+c.
}
Now we let $\epsilon' \to 0$ and use Property 4 of Lemma \ref{regularise} to obtain
%\eqs{
$T_\infty f(x)\leq T \circ T^{\infty} f(x)+c,$
%}
and thus obtaining equality.

\subsection{Weak KAM solutions for unbounded transfers}  

The following lemma shows that the above hypothesis $c (\T_\epsilon)=c(\T)$ is not vacuous as it occurs in many examples.

\begin{prop}
Let $\T$ be a backward linear transfer with Kantorovich operator $T$. In any of the following cases,
\begin{enumerate}
\item $\inf\{\T(\mu,\nu); \mu,\nu \in \mcal{P}(X)\} = \inf\{\T(\mu,\mu)\,;\, \mu \in \mcal{P}(X)\}$,
\item $\T$ is symmetric and for some $\epsilon_0 > 0$, $T^-$ maps every continuous function to a $1/\epsilon_0$-Lipschitz function,
\end{enumerate}
we have $c(\T_\epsilon) = c(\T)$ for all small enough $\epsilon > 0$.
\end{prop}
\noindent{\bf Proof:}  To see 1) note that $\inf\{\T(\mu,\nu); \mu, \nu \in \mcal{P}(X)\}=\inf\{\T_\epsilon (\mu,\nu); \mu, \nu \in \mcal{P}(X)\}$ for every $\epsilon>0$. By property $2$ of Lemma \ref{regularise}, and property $3$ of Lemma \ref{c_epsilon}, we get 
%\begin{align*}
\[
c(\T_\epsilon)\leq c (\T)= \inf\limits_{\mu \in \mcal{P}(X)}\T(\mu,\mu)= \inf\{\T(\mu,\nu); \mu, \nu \in \mcal{P}(X)\}=\inf\{\T_\epsilon (\mu,\nu); \mu, \nu \in \mcal{P}(X)\}\leq c(\T_\epsilon).  
%\end{align*} 
\]
For 2) write
\begin{align*}
\T_\epsilon(\mu,\mu) &= \inf_{\sigma_1, \sigma_2}\{ \frac{1}{\epsilon}W(\mu, \sigma_1) + \T(\sigma_1, \sigma_2) + \frac{1}{\epsilon}W(\sigma_2, \mu)\}\\
&\geq \inf_{\sigma_1, \sigma_2}\{ \frac{1}{\epsilon}W(\mu, \sigma_1) - \T(\sigma_1, \sigma_1) + \T(\sigma_1, \sigma_2)  + \frac{1}{\epsilon}W(\sigma_2, \mu)\} + \inf\{\T(\sigma_1, \sigma_1)\,;\, \sigma_1\}\\
&\geq \inf_{\sigma_1, \sigma_2, \sigma_3}\{ \frac{1}{\epsilon}W(\mu, \sigma_1) - \T(\sigma_1, \sigma_3)+ \T(\sigma_3, \sigma_2)  + \frac{1}{\epsilon}W(\sigma_2, \mu)\} + \inf\{\T(\sigma_1, \sigma_1)\,;\, \sigma_1\}\\
&= (\frac{1}{\epsilon}W)\star(-\T)\star\T\star(\frac{1}{\epsilon}W)(\mu,\mu) + c.
\end{align*}
It suffices to show that $(\frac{1}{\epsilon}W)\star(-\T)\star\T\star(\frac{1}{\epsilon}W)(\mu,\mu) \geq 0$. Note that we can write
\begin{align*}
(-\T)\star\T(\mu,\nu) &= \inf_{\sigma}\{-\T(\mu, \sigma) + \T(\sigma, \nu)\}\\
&= \inf_{f}\inf_{\sigma}\{\int_X T f\d\mu - \int_X f\d\sigma + \T(\sigma, \nu)\}\\
&= \inf_{f}\{\int_X T^-f \d\mu + \int_X T^+(-f)\d\nu\}.
\end{align*}
Then with the notation that $S_\epsilon^-$ is the backward Kantorovich operator for $\frac{1}{\epsilon}W$, we arrive at
\begin{align*}
(\frac{1}{\epsilon}W)\star(-\T)\star\T\star(\frac{1}{\epsilon}W)(\mu,\mu) &= \inf_{\sigma_1, \sigma_2}\{\frac{1}{\epsilon}W(\mu, \sigma_1)+ (-\T)\star \T(\sigma_1, \sigma_2) + \frac{1}{\epsilon}W(\sigma_2,\mu)\}\\
&= \inf_{f}\inf_{\sigma_1, \sigma_2}\{\frac{1}{\epsilon}W(\mu, \sigma_1) + \int_X T^-f \d\sigma_1 + \int_X T^+(-f)\d\sigma_2 + \frac{1}{\epsilon}W(\sigma_2,\mu)\}\\
&= \inf_{f}\{ -\int_X S^-_\epsilon(-T^- f) \d\mu - \int_X S_\epsilon^-(-T^+(-f)) \d\mu\}\\
&= \inf_{f}\{ -\int_X S^-_\epsilon(-T^- f) \d\mu - \int_X S_\epsilon^-(T^- f) \d\mu\}\\
&= \inf_{f}\{ -\int_X (-T^- f)\d\mu - \int_X T^- f\d\mu\}\\
&= 0,
\end{align*}
where the second-last equality follows from the fact that whenever $g$ is $\frac{1}{\epsilon_0}$-Lipschitz, then $S_\epsilon^- g = g$.

\begin{prop}\label{perturbedtransfer}
Let $S: C(X) \to C(Y)$ be a Markov operator (i.e., a bounded linear positive operator such that $T1=1$) and let $S^*: \mathcal{M}(X) \to \mathcal{M}(Y)$ be its adjoint. Given a backward linear transfer $\T: \mathcal{P}(X)\times \mathcal{P}(Y)$ and $\lambda \in (0,1)$, define
\begin{equation}
\tilde{\T}(\mu,\nu) := \T(\mu, \lambda S^*\mu + (1-\lambda)\nu).
\end{equation}
Then, $\tilde{\T}$ is a backward linear transfer with Kantorovich operator 
\begin{equation}
\tilde{T}^-f(x) := T^- \lf(\frac{1}{1-\lambda}f\rt)(x) - \frac{\lambda}{1-\lambda}Sf(x)
\end{equation}
\end{prop}
\noindent {\bf Proof:}
Write
\begin{align*}
(\tilde{\T}_\mu)^*(f) &= \sup_{\sigma}\{ \int_{X}f\d\sigma - \tilde{\T}(\mu,\sigma)\}\\
&= \sup_{\sigma}\{ \int_{X}f\d\sigma - \T(\mu, \lambda S^*\mu + (1-\lambda)\sigma)\}
\end{align*}
with $\tilde{\sigma} := \lambda S^*\mu + (1-\lambda)\sigma$, we obtain $\sigma = \frac{1}{1-\lambda}\tilde{\sigma} - \frac{\lambda}{1-\lambda}S^*\mu$. Hence after substitution we obtain
\begin{align*}
(\tilde{\T}_\mu)^*(f) &= \sup_{\tilde{\sigma}}\{\int_{X}\frac{1}{1-\lambda}f\d \tilde{\sigma} - \T(\mu, \tilde{\sigma})\} -\frac{\lambda}{1-\lambda} \int_{Y}Sf \d\mu\\
&= \int_{X}\lf[T^- \lf(\frac{1}{1-\lambda}f\rt) - \frac{\lambda}{1-\lambda}Sf\rt]\d\mu.
\end{align*}

\begin{thm}
Let $\T$ be a backward linear transfer on $\mcal{P}(X)\times \mcal{P}(X)$. Then, for every $\lambda \in (0, 1)$, there exists a convex function $\phi_\lambda$ such that the linear transfer given by
\begin{equation}
\tilde{\T}_\lambda (\mu,\nu) := \T(\mu, \lambda (\nabla \phi_\lambda)_\#\mu + (1-\lambda)\nu)
\end{equation}
is such that 
\[
\inf\limits_{\mu. \nu \in {\cal P}(X)} {\tilde \T}_\lambda (\mu, \nu)=\inf\limits_{\mu \in {\cal P}(X)}{\tilde \T}_\lambda(\mu, \mu).
\] 
In particular, ${\tilde \T}_\lambda$ admits weak KAM solutions, that is there exists $g\in USC(X)$ and $c\in \R$ such that 
\begin{equation}
T^-g   +c= \lambda \, g (\nabla \phi) +(1-\lambda)\, g.
\end{equation}
\end{thm}
\noindent {\bf Proof:} Let $\T(\mu_0, \nu_0) = \inf_{\mu,\nu}\T(\mu,\nu) < +\infty$ for some $\mu_0$ and $\nu_0$, and use Brenier's theorem to find a convex function $\phi$ such that $\nabla \phi_\#\mu_0 = (1 - \frac{1}{\lambda})\mu_0 + \frac{1}{\lambda}\nu_0$.
%with the property that there is a continuous map $\Phi: X \to X$ for which $\Phi \#\mu_0 = (1 - \frac{1}{\lambda})\mu_0 + \frac{1}{\lambda}\nu_0$ for some $\lambda \in (0,1)$. Then there exists a 

Consider now the backward linear transfer $\tilde{\T}(\mu,\nu) := \tilde{\T}(\mu, \lambda \nabla \phi\#\mu + (1-\lambda)\nu)$ 
and note that ${\cal \T}(\mu_0, \mu_0)=\T(\mu_0, \nu_0) < +\infty$. Moreover,  
\begin{align*}
\inf_{\mu,\nu}\tilde{\T}(\mu,\nu) \geq \inf_{\mu,\nu}\T(\mu,\nu) = \T(\mu_0, \nu_0) = \tilde{\T}(\mu_0, \mu_0) \geq \inf_{\mu}\tilde{\T}(\mu,\mu),
\end{align*} 
hence 
$
\inf_{\mu,\nu}\tilde{\T}(\mu,\nu) = \inf_{\mu}\tilde{\T}(\mu,\mu).
$, and in particular, $\tilde{\T}$ satisfies the hypotheses of Theorem \ref{fixedpointgeneral}, and admits weak KAM solutions for its Kantorovich operator, which is given by $\tilde{T}^- f = T^- (\frac{1}{1-\lambda}f) - \frac{\lambda}{1-\lambda}f \circ \nabla \phi $. In other words, by setting $g:=\frac{1}{1-\lambda}f$, we have 
\[
T^-g   +c= \lambda \, g (\nabla \phi) +(1-\lambda)\, g.
\]

%According to Proposition \ref{perturbedtransfer} with $Sf := f\circ \Phi$, the transfer $\tilde{\T}(\mu,\nu) := \T(\mu, \lambda \Phi\#\mu + (1-\lambda)\nu)$ is a backward linear transfer with the specified expression for $\tilde{T}^-f$.
%Note that the assumption on $\Phi\#\mu_0$ implies $\tilde{\T}(\mu_0, \mu_0) = \T(\mu_0, \lambda \Phi \# \mu_0 + (1-\lambda)\mu_0) = \T(\mu_0, \nu_0) < +\infty$, and moreover,
%\begin{align*}
%\inf_{\mu,\nu}\tilde{\T}(\mu,\nu) \geq \inf_{\mu,\nu}\T(\mu,\nu) = \T(\mu_0, \nu_0) = \tilde{\T}(\mu_0, \mu_0) \geq \inf_{\mu}\tilde{\T}(\mu,\mu)
%\end{align*} 
%hence the equality $\inf_{\mu}\tilde{\T}(\mu,\mu) = \inf_{\mu,\nu}\tilde{\T}(\mu,\nu)$.
%$\tilde{\T}$ satisfies the hypotheses of Theorem \ref{fixedpointgeneral}, and 
\subsection{The heat semi-group and other examples}

Assumption \refn{cepsilon} is actually satisfied by a large number of our transfer examples.\\

1) Let $\T$ be the backward transfer associated to a convex lower semi-continuous functional $I$ on Wasserstein space, that is $\T (\mu, \nu):=I(\nu)$. Assumption (\ref{cepsilon}) then holds trivially as $c(\T)=\inf I$ in this case, and the associated idempotent transfer is $\T_\infty (\mu, \nu):=I(\nu) -c$, while the corresponding idempotent operator is $T_\infty f=I^*(f) +c$, where $I^*$ is the Legendre transform of $I$. \\

 2) Assumption (\ref{cepsilon}) clearly holds for any transfer that is $\{0, +\infty\}$-valued provided (\ref{inv}) is satisfied. Note that if $T$ is a Markov operator, then assumption (\ref{inv}) means that $T$ has an invariant measure. In this case, $c(\T)=0$, and for every $f$, $T_\infty f$ is an invariant function under $f$.
 
3)  If $T$ is induced by a continuous point transformation, i.e., $Tf(x)=f(\sigma  (x)$ for a continuous map $\sigma: X \to X$, then by a Theorem of  Bogolyubov and Krylov, $T$ has an invariant measure and the above applies. The operator $T_\infty$ is then given by 
 \[
 T^\infty f (x)=f(\limsup_{m \to \infty}\sigma^m(x)):=f(\sigma^\infty(x)).
 \]
  However, the regularity of the invariant functions $T_\infty f$ can vary widely. For example, 
 \begin{itemize}
\item If  one takes $X = [0,1]$ and $\sigma(x) = x^2$, then $\sigma^\infty(x) = 0$, if $x \in [0,1)$ and $1$ if $x = 1$. In this case, $T^\infty f$ only belongs to $USC_\sigma(X)$. 

\item On the other hand,  if $\sigma(x) = 1-x$, then $\sigma^\infty (x) = \max\{x, 1-x\}$ is continuous.
\end{itemize}

  4) {\it The heat semi-group:} Recall the Skorokhod transfer of Example 3. 4. Instead of considering a stopping time $\tau$, we let $\tau = t > 0$ to be deterministic, and define, for measures $\mu$, $\nu$, on a compact Riemannian manifold $M$, the linear transfer,
\begin{equation*}
\T_t(\mu,\nu) = \begin{cases}
0 & \text{if $B_0 \sim \mu$ and $B_t \sim \nu$}\\
+\infty & \text{otherwise.}
\end{cases}
\end{equation*}
Then $T_tf(x)= \E^x[f(B_t)] = P_t f(x)$, where $P_t$ is the heat semigroup. Note that since the volume measure $\lambda_M$, i.e., the uniform probability measure on $M$, is invariant, we have $\T_t(\lambda_M,\lambda_M)=0$, hence  Condition (\ref{inv}) is satisfied. We now have the following easy proposition, which puts our asymptotic result in the following classical context.  
\begin{prop}
The collection $\{\T_t\}_{t > 0}$ is a semigroup of backward linear transfers with Kantorovich operators $\{T_t\}_{t > 0}$. The corresponding idempotent backward linear transfer $\T_\infty(\mu,\nu) = \sup\{ \int_M fd\nu - \int_{M}f d\lambda_M \,;\, f \in C(M)\}$ with Kantorovich operator $T_\infty f(x) = \int_{M}f\d\lambda_M$.  
\end{prop}
\noindent{\bf Proof:} It is immediate to verify that 
\begin{align*}
(\T_{t,\mu})^*(f) = \sup\{\int_{M}f d\nu\,;\, \nu \text{ such that }B_t \sim \nu, B_0 \sim \mu\} = \int_M P_t f(x)d \mu(x).
\end{align*}
Moreover, it is a standard property of the heat semigroup, that $P_t f \to P_\infty f = \int_{M}f\d\lambda_M$, uniformly on $M$, as $t \to \infty$, for any $f \in C(M)$. By the $1$-Lipschitz property of $T_t$, we conclude $T_t \circ T_\infty f = T_\infty f$.

 \section{Stochastic weak KAM on the Torus}

In this section, we are interested in making the connection between our general notion of linear transfers, stochastic mass transports, and existing work on stochastic weak KAM theory, in particular, by Gomes \cite{Gom}. We shall therefore restrict our setting to $M = \mathbb{T}^d := \R^d/\Z^d$, the $d$-dimensional flat torus. Note that, unlike the deterministic Mather theory, this does not fall under the Monge-Kantorovich setting. 

First, we introduce the stochastic mass transport of a probability measure $\mu$ to a probability measure $\nu$ on $\mathcal{P}(M)$ in time $t > 0$ (see e.g. \cite{M-T} when the space is $\R^d$ and $t = 1$). Define $\T_t(\mu,\nu): \mcal{P}(M)\times \mcal{P}(M) \to \R \cup \{+\infty\}$ via the formula,
\begin{equation}\label{stoctrans}
\T_{t}(\mu,\nu) := \inf\lf\{\E \int_{0}^{t} L(X(s), \beta_X(s,X))\d s\,;\, X(0) \sim \mu, X(t) \sim \nu, X \in \mcal{A}_{[0,t]}\rt\},
\end{equation}
where $L: TM \to [0,\infty)$ is a given Lagrangian function which we detail below, and $X$ is a continuous semi-martingale with an associated drift $\beta_X$, belonging to a class of stochastic processes $\mcal{A}_{[0,t]}$ defined below.

Stochastic transport has a dual formulation (first proven in Mikami-Thieullin \cite{M-T} for the space $\R^d$) that permits it to be realised as a backward linear transfer. In fact, by introducing the operator $T_t: C(M) \to USC(M)$ via the formula
\begin{equation}\label{stochasticop}
T_{t} f(x) := \sup_{X \in \mcal{A}_{[0,t]}} \lf\{\E \lf[f(X(t))|X(0) = x\rt] - \E\lf[\int_{0}^{t} L(X(s),\beta_X(s,X))\d s|X(0) = x\rt]\rt\},
\end{equation}
the duality relation between $\T_t$ and $T_t$ can be readily detailed, see Proposition \ref{backwardlintrans}. The operator $T_t$ connects to the work of Gomes via a Hamilton-Jacobi-Bellman equation (\ref{timedep}),

Concerning the assumptions on $L$, we make the following hypotheses.
 
\enum{
\item[(A0)] $L$ is continuous, non-negative, $L(x,0) = 0$, and $D^2_v L(x,v)$ is positive definite for all $(x,v) \in TM$ (in particular $v \mapsto L(x,v)$ is convex).
 
\item[(A1)] There exists a function $\gamma = \gamma(|v|): \R^n \to [0,\infty)$ such that $\lim_{|v| \to \infty} \frac{L(x,v)}{\gamma(v)} = +\infty$ and $\lim_{|v| \to \infty}\frac{|v|}{\gamma(v)} = 0$.
}
To complete the definition for $\T_t$, we need to define the set of processes $\mcal{A}_{[0,t]}$.  As in \cite{M-T},  let $(\Omega, \mcal{F}, \P)$ be a complete probability space with normal filtration $\{\mcal{F}_t\}_{t \geq 0}$, and define $\mcal{A}_{[0,t]}$ to be the set of continuous semi-martingales $X: \Omega \times [0,t] \to M$ such that there exists a Borel measurable drift $\beta_X: [0,t] \times C([0,t]) \to \R^d$ for which
\enum{
\item $\omega \mapsto \beta_X(s,\omega)$ is $\mcal{B}(C([0,s]))_{+}$-measurable for all $s \in [0,t]$, where $\mcal{B}(C([0,s]))$ is the Borel $\sigma$-algbera of $C[0,s]$.
\item $W_X(s) := X(s) - X(0) - \int_{0}^{s}\beta_{X}(s', X)\d s'$ is a $\sigma(X(s)\,;\, 0 \leq s \leq t)$  
$M$-valued Brownian motion. 
}
 An adaptation of their proofs to the case of a compact torus yields the following.  

\begin{prop}\label{backwardlintrans} Under the above hypothesis on $L$, the following assertions hold:
\begin{enumerate}
\item For each $t > 0$, $\T_t$ is a backward linear transfer with Kantorovich operator $T_t$, and the family $\{\T_t\}_{t > 0}$  is a semi-group of transfers under convolutions.

\item For any $\mu,\nu \in \mcal{P}(M)$ for which $\T_t(\mu,\nu) < \infty$, there exists a minimiser $\bar{X} \in \mcal{A}_{[0,t]}$ for $\T_t(\mu,\nu)$. For every $f \in C(M)$ and $x \in M$, there exists a maximiser for  $T_tf(x)$.
\item Fix $t_1 > 0$, and $u\in C(M)$, the function $U(t,x) := T_{t_1-t} u(x)$ defined for $0 \leq t \leq t_1$  
 is the unique viscosity solution of 
\begin{equation}\label{timedep}
\frac{\partial U}{\partial t}(t,x) + \frac{1}{2}\Delta_x U(t,x) + H(x, \nabla_x U(t,x)) = 0,\quad (t,x) \in [0,t_1)\times M, 
\end{equation}
with $U(t_1,x) = u(x)$.
\item If $f \in C^\infty(M)$ and $t > 0$,  $U(t',x) := T_{t-t'} f \in C^{1,2}([0,t]\times M)$ and $U$ is a classical solution to the Hamilton-Jacobi-Bellman equation \refn{timedep}. 
The maximiser $\bar{X}$ satisfies
\eqs{
\beta_{\bar{X}}(s,\bar{X}) = D_p H(\bar{X}(s), D_x U(s,\bar{X}(s))).
}
\end{enumerate} 
\end{prop}
 In order to define the Man\'e constant $c(\T)$ and develop a corresponding Mather theory, we need to establish that there exists a probability measure $\mu \in {\cal P}(M)$ such that $\T_1(\mu, \mu)<+\infty$. Such a measure can be obtained as the first marginal of a probability measure $m$ on phase space $TM$ that is flow invariant, that is one that satisfies
\begin{equation}
\int_{TM}A^v \vphi(x)\d m(x,v) = 0 \hbox{ for all $\vphi \in C^2(M)$ where $A^v\vphi := \frac{1}{2}\Delta \vphi + v\cdot \nabla \vphi$. }
\end{equation} 
To this end, let $\mcal{P}_\gamma(M)$ denote the set of probability measures on $TM$ such that $$\int_{TM}\gamma(v)\d m(x,v) < +\infty,$$ and denote by ${\cal N}_0$ the class of such probability measures $m$, that is,
\begin{equation*}
\mcal{N}_0 := \{ m \in \mcal{P}_\gamma(TM)\,;\, \int_{TM}A^v \vphi(x)\d m(x,v) = 0 \text{ for all }\vphi \in C^2(M)\}.
\end{equation*}

\begin{prop} The set ${\cal N}_0$ of `flow-invariant' probability measures $m$ on $TM$ is non-empty and 
\begin{equation}
c := \inf\{\T_1(\mu,\mu)\,;\, \mu \in \mcal{P}(M)\}=\inf \{\int_{TM} L(x,v)\d m(x,v);\, m \in {\cal N}_0\}. 
\end{equation}
Moreover, the infimum over  ${\cal N}_0$  
is attained by a measure $\bar m$, that we call a {\it stochastic Mather measure.} Its projection  $\mu_{\bar m}$  on $\mcal{P}(M)$ is a minimiser for $\T_1$. 

Conversely, every minimizing measure $\bar \mu$ of $\T_1(\mu,\mu)$ induces a stochastic Mather measure $m_{\bar \mu}$.
\end{prop} 
\noindent {\bf Proof:} 
Given $\mu \in \mcal{P}(M)$, consider $X \in \mcal{A}_{[0,1]}$ that realises the infimum for $\T_1(\mu,\mu)$, that is
\eqs{
\T_1(\mu,\mu) = \E \int_{0}^{1}L(X(s), \beta_X(s,X))\d s.
}
Define a probability measure $m = m_\mu \in \mcal{P}_\gamma(TM)$ via its action on the subset of continuous functions $\psi :TM \to \R$ with $\sup_{(x,v) \in TM}\lf|\frac{\psi(x,v)}{\gamma(v)}\rt| < +\infty$ and $\lim_{|(x,v)| \to \infty}\frac{\psi(x,v)}{\gamma(v)} \to 0$ via the formula
\begin{equation}\label{definemathermeasure}
\int_{TM}\psi(x,v)\d m(x,v) := \E \int_{0}^{1}\psi(X(s), \beta_X(s,X))\d s.
\end{equation}
We claim that $\int_{TM}A^v \vphi(x)\d m(x,v) = 0$ for every $\vphi \in C^2(M)$. 
Indeed, by the definition of $m$,
\as{
\int_{TM}A^v\vphi(x)\d m(x,v) &= \E \int_{0}^{1}A^{\beta_X(s,X)}\vphi(X(s))\d s\\
 &= \E \int_{0}^{1}\frac{\d }{\d s}[\vphi(X(s))]\d s\quad \text{(It\^o's lemma)}\\
 &= \E \vphi(X(1)) - \E\vphi(X(0))\\
 &= 0,\quad\quad \text{($X(0) \sim \mu \sim X(1)$).}
}
This implies that $m \in \mcal{N}_0$, so that
\begin{align}
\T_1(\mu,\mu) = \E \int_{0}^{1}L(X(s), \beta_X(s,X))\d s
 = \int_{TM}L(x,v)\d m(x,v)\lbl{equalitymathermane}
 \geq \inf_{m \in \mcal{N}_0}\int_{TM}L(x,v)\d m(x,v), 
\end{align}
hence
\eqs{
\inf_{\mu \in \mcal{P}(M)}\T_1(\mu,\mu) \geq \inf_{m \in \mcal{N}_0}\int_{TM}L(x,v)\d m(x,v).
}
Conversely, suppose $m \in \mcal{N}_0$, and let $\vphi(x,t)$ be a smooth solution to the Hamilton-Jacobi-Bellman equation. Since  $\int_{TM}A^v \vphi(x,t)\d m(x,v)= 0$ for every $t$, it follows that $\mu_m := \pi_M\# m$ satisfies
\as{
\int_{M}[\vphi(x,1) - \vphi(x,0)]\d \mu_m(x) &= \int_{[0,1]}\frac{\d}{\d t}\lf[\int_{M}\vphi(x,t)\d \mu_m\rt] \d t\\
&= \int_{0}^{1}\int_{TM}\partial_t \vphi(x,t)\d m(x,v)\d t\\
&= \int_{0}^{1}\int_{TM}[v\cdot \nabla \vphi(x,t) - H(x, \nabla_x \vphi(x,t))]\d m(x,v)\d t.
}
Since $H(x,p) := \sup_{v}\lf\{\langle p, v\rangle - L(x,v)\rt\}$, 
\as{
v\cdot \nabla \vphi(x,t) - H(x, \nabla_x \vphi(x,t)) \leq L(x,v),
}
hence combining the above two displays implies
\eqs{
\int_{M}[\vphi(x,1) - \vphi(x,0)]\d \mu_m(x) \leq \int_{TM}L(x,v)\d m(x,v)
}
for every Hamilton-Jacobi-Bellman solution $\vphi$ on $[0,1)\times M$ with $\vphi(\cdot,1) \in C^\infty(M)$. Taking the supremum over all such solutions $\vphi$ yields
\eqs{
\sup\lf\{\int_{M}[\vphi(x,1) - \vphi(x,0)]\d \mu_m(x)\,;\, \vphi(\cdot, 1) \in C^\infty(M)\rt\} \leq \int_{TM}L(x,v)\d m(x,v).
}
By duality, $\T_1(\mu_m,\mu_m) = \sup\lf\{\int_{M}[\vphi(x,1) - \vphi(x,0)]\d \mu_m(x)\,;\, \vphi(\cdot, 1) \in C^\infty(M)\rt\}$, so that 
\begin{equation}\label{optimalprojection}
\T_1(\mu_m,\mu_m) \leq \int_{TM}L(x,v)\d m(x,v)
\end{equation}
and therefore
%\eqs{
$\inf_{\mu \in \mcal{P}(M)}\T_1(\mu,\mu) \leq \int_{TM}L(x,v)\d m(x,v),$ and we are done. \qed

\noindent The following summarizes the main asymptotic properties of  $\{\T_t\}_{t > 0}$.
 
\begin{prop} Let $\{\T_t\}_{t \geq 0}$ be the family of stochastic transfers defined via \refn{stoctrans} with associated backward Kantorovich operators $\{T_t\}_{t \geq 0}$ given by \refn{stochasticop}. Let $c$ be the critical value obtained in the last proposition. Then, 
 
\begin{enumerate}
\item The equation
\eqs{
T_t u + kt = u,\quad t \geq 0, \quad u \in C(M),
}
has solutions (the \textit{backward weak KAM solutions}) if and only if $k = c$.
 
\item The backward weak KAM solutions are exactly the viscosity solutions of the stationary Hamilton-Jacobi-Bellman equation
\begin{equation}\label{stat-HJB}
\frac{1}{2}\Delta u + H(x, D_x u) = -c.
\end{equation}
 \end{enumerate}
\end{prop}
\begin{proof} The fact that there are solutions for (\ref{stat-HJB}) was established by Gomes \cite{Gom}. We give a proof based on Proposition \ref{hor1} that clarifies the relationship between such solutions and the notion of backward weak KAM solutions.

Let $\alpha >0$ and consider 
$$u_\alpha(x):= \inf\lf\{\E \int_{0}^{+\infty} e^{-s}L(X(s), \beta_X(s,X))\d s\,; \, X \in \mcal{A}_{[0,t]}, \,  X(0)=x\rt\}.$$
It is well known that one then has
\[
u_\alpha(x)=\inf\lf\{\E \int_{0}^{t} e^{-s}L(X(s), \beta_X(s,X))\d s + e^{-\alpha t}u_\alpha (X(t))\,; \, X \in \mcal{A}_{[0,t]}, \,  X(0)=x\rt\},
\]
and
\[
\alpha u_\alpha -\frac{1}{2}\Delta u_\alpha + H(x, D_x u_\alpha)=0. 
\]
It is straightforward to check that this implies that 
\[
T_t u_\alpha +t \alpha u_\alpha = u_\alpha. 
\]
Proposition \ref{hor1} applies to get the result with $t=n$. Note that for constructing a viscosity solution for  \refn{stat-HJB}, it suffices to find a weak KAM solution for $T_1$. Indeed, 
 suppose there exists a function $f \in C(M)$ such that $T_1 f(x) + c = f(x)$ for all $x \in M$, we need to show that $T_t f(x) + ct = f(x)$ for all $t > 0$.
But note that from the semi-group property, the claim is true for $t = n \in \N$. For other $t > 0$, by writing uniquely $t = n + \alpha$ where $n \in \N$ and $0 \leq \alpha < 1$, it then suffices to prove that $T_\alpha f(x) + \alpha c = f(x)$.
Note that the function $U(t,x) := T_{1-t}f(x) + c(1-t)$ satisfies the Hamilton-Jacobi-Bellman equation
\begin{equation}\label{viscoinitial}
\begin{cases}
\frac{\partial U}{\partial t}(t,x) + \frac{1}{2}\Delta U(t,x) + H_c(x, \nabla U(t,x)) = 0,\quad t \in [0,1), x \in M\\
U(1,t) = f(x).
\end{cases}
\end{equation}
where $H_c(x,p) := H(x,p) + c$, with the additional property that $U(0,x) = U(1,x)$. We may then apply a comparison result for Hamilton-Jacobi-Bellman (see e.g. \cite{FS}, Section V.8, Theorem 8.1) to deduce that in fact the condition $U(0,x) = U(1,x)$ implies $U(t,x) = U(1,x)$ for every $t \in [0,1]$. In particular, at $t = 1-\alpha$, we deduce that $T_\alpha f(x) + c\alpha = f(x)$.

 As to the relationship between 1) and 2) observe that if $u$ is a backward weak KAM solution,  
then $U(t,x) := T_{t_1-t}u(x) + c(t_1 - t)= u(x)$ is a viscosity solution to (\ref{viscoinitial}) where the final time is $t_1$.
Hence $u$ is a viscosity solution of \refn{stat-HJB}.

Conversely, suppose $u$ is a viscosity solution to \refn{stat-HJB}.Then, $(x,t) \mapsto u(x)$ is a viscosity solution to \refn{viscoinitial}.
 On the other hand,  
$T_{t_1-t} u +c(t_1 - t)$ is also a viscosity solution of \refn{viscoinitial}. By the uniqueness of such solutions,  
it follows that $T_{t_1-t}u(x) + c(t_1-t) = u(x) $. As $t_1 > 0$ is arbitrary, this shows that $u$ is a backward weak KAM solution.

\end{proof}

We finish this section with the following characterization of the Man\'e value, motivated by the work of Fathi \cite{Fa} in the deterministic case. Let $u \in C(M)$ and $k \in \R$, and  say that {\it $u$ is dominated by $L-k$} and write $u \prec L - k$ if for every $t > 0$, it holds for every $X \in \mcal{A}_{[0,t]}$ and every $x\in M$, 
\begin{equation}
\E [u(X(t))|X(0) = x] - u(x) \leq \E \lf[\int_{0}^{t}L(X(s), \beta_X(s,X)\d s|X(0) = x\rt] - kt.
\end{equation}
\begin{prop}
The Ma\~n\'e critical value satisfies
\eqs{
c = \sup\lf\{k \in \R\,:\, \exists u \quad \hbox{such that}\quad  u \prec L - k\rt\}.
}
\end{prop}

\begin{proof}
By the above,  
there exists a $u$ such that $T_t u + c t = u$, so that by definition of $T_t$,
\eqs{
u(x) - ct \geq \E [u(X(t))|X(0) = x] - \E \lf[\int_{0}^{t} L(X(s),\beta_X(s,X)\d s|X(0) = x\rt]
} 
for every $X \in \mcal{A}_{[0,t]}$. This shows that $u \prec L - c$, so $c$ is itself admissible in the supremum.

On the other hand, if $k \in \R$ is such that $u \prec L - k$, then it is easy to see that $T_tu(x) \leq u(x) - kt$ for all $t$. In particular, applying $T_s$ and using the linearity of $T_s$ with respect to constants, we find $T_{s+t}u + kt \leq T_su$, and hence
\eqs{
T_{s+t}u + k(t+s) \leq T_su + ks
}
So $t \mapsto T_tu + kt$ is decreasing and the result follows from Corollary \ref{useful}.
 
\end{proof}

\section{Convex couplings and convex and Entropic Transfers}
First, recall that the  {\it increasing Legendre transform} (resp., {\it decreasing Legendre transform)} of a function $\alpha: \R^+\to \R$ (resp., $\beta: \R^+ \setminus \{0\} \to \R$) is defined as
    \begin{equation}
  \alpha^{\oplus}(t)=\sup\{ts-\alpha(s); s\geq 0\} \quad\hbox{resp.,  $\beta^{\ominus}(t)=\sup\{-ts-\beta(s); s > 0\}$ }
   \end{equation}
By extending $\alpha$ to the whole real line by setting $\alpha (t)=+\infty$ if $t<0$, and using the standard Legendre transform, one can easily show that $\alpha$ is convex increasing on $\R^+$ if and only if  $\alpha^{\oplus}$ is convex and increasing on $\R^+$. We then have the following reciprocal formula
  \begin{equation}
\alpha (t)=\sup\{ts-\alpha^{\oplus} (s); s\geq 0\}.
  \end{equation}
Similarly, if $\beta$ is convex decreasing on $\R^+\setminus \{0\}$, we have
 \begin{equation}
\beta (t)=\sup\{-ts-\beta^{\ominus} (s); s\geq 0\}.
  \end{equation}

 \subsection{Convex couplings}
 
  We now give a few examples of convex couplings, which are not necessarily convex transfers.   
  
  \begin{prop} Let $\alpha: \R^+\to \R$ (resp., $\beta: \R^+ \setminus \{0\} \to \R$) be a convex (resp., concave) increasing functions.
   % \begin{enumerate}
  If ${\cal T}$ is a linear backward (resp., forward) transfer with  Kantorovich operator $T^-$ (resp.,  Kantorovich operator $T^+$), then $\alpha ({\cal T})$ is a  backward convex  (resp.,  forward convex) coupling associated to a family of (Kantorovich) operators $(T_s^-)_{s\geq 0}$ (resp.,$(T_s^+)_{s\geq 0}$, where
\begin{equation}
T_s^-f=sT^-(\frac{f}{s})-\alpha^\oplus(s)\quad \hbox{(resp., $T_s^+f=sT^+(\frac{f}{s})-\alpha^\oplus(s)$.}
\end{equation}
In particular, for any $p\geq 1$, $\T^p$ is a forward (resp., backward) convex coupling. 
 
    \end{prop} 
\noindent{\bf Proof:}
It suffices to write 
\begin{eqnarray*}
\alpha ({\mathcal T}(\mu, \nu))
&=& 
\sup\big\{s\int_{Y}{T^+}f\, d\nu-s\int_{X}f\, d\mu -\alpha^\oplus(s);\, s\in \R^+,  f \in C(X) \big\}\\
&=& 
\sup\big\{\int_{Y}sT^+(\frac{h}{s})\, d\nu  -\alpha^\oplus(s)-\int_{X}h\, d\mu;\, s\in \R^+,  h \in C(X) \big\},
\end{eqnarray*}
which means that $\alpha ({\cal T})$ is a forward convex coupling corresponding to the family of (Kantorvich) operators $T_s^+f=sT^+(\frac{h}{s})-\alpha^\oplus(s)$.\\

\noindent{\bf Example 11.1): A mean-field planning problem} (Orrieri-Porretta-Savar\'e \cite{Sav})

Let $L:\R^d\times \R^d \to \R$ be a Tonelli Lagrangian and $F:\R^d \times L^\infty([0, T];{\cal P}(\R^d)) \to \R$ be a functional that is convex in the second variable,  and consider the following mean-field planning problem between two probability measures $\mu$ and $\nu$,    
   \begin{equation} 
{\cal T}(\mu, \nu)=:\min
\left\{\int_0^T\int_{\R^d} L(x,\vv)\,\rho(t, dx)\, dt+
\int_0^TF(x,\rho(t, dx))\,\d t; \,\, \vv\in
L^2(\rho (t, dx)\, dt)\right\},
\end{equation}
subject to $\rho$ and $v$ satisfying
\begin{eqnarray}
 \label{con} 
\partial_t \rho+\nabla\cdot (\rho\,\vv )=0, \quad
\rho(0,\cdot)=\mu\,, \rho(T,\cdot)=\nu.  
\end{eqnarray}
Then, ${\cal T}$ is  both a forward and backward convex coupling. 
 
Indeed, following Orrieri-Porretta-Savar\'e \cite{Sav}, we consider for each $\ell\in C([0, T], \R^d)$ the Kantorovich operator defined on $C(\R^d)$ via
\[
T_\ell (u)= u_\ell (T, x)-   \iint_Q F^*(x, \ell(t,x))\,\d x
\]
 where $u_\ell(t, x)$ is a solution of the Hamilton-Jacobi equation 
\begin{eqnarray}\label{HJN}
 -\partial_t u+H(x,Du)&=&\ell\quad\text{in } Q:=(0, T)\times \R^d, \\
 u(0, x)&=&u(x).
 \end{eqnarray}
and  $F^*(x, \ell)=\sup \left\{\langle \ell, \rho\rangle -F(x, \rho); m\in L^\infty([0, T];{\cal P}(\R^d))\right\}$. 

A standard min-max argument then yields that
\begin{eqnarray*} 
{\cal T}(\mu, \nu)&=&\sup \left\{\int_{\R^d}T_\ell u\, d\nu-\int_{\R^d}u \, d\mu; u\in C(\R^d), \ell \in C([0, T], \R^d)\right\}\\
&=&\sup \left\{\int_{\R^d}u_\ell (T, x) d\nu-
  \int_{\R^d}u_\ell (0, x)\, d\mu (x)-
                     \iint_Q F^*(x, \ell(t,x))\,\d x; \, u_\ell\, {\rm solves}\,\, (\ref{HJN})  
                 \right\}.
\end{eqnarray*}

\begin{rmk} Another convex -but only backward- coupling can be defined as 
 \begin{equation} 
{\cal T}(\mu, \nu)=:\min
\int_0^T\int_{\R^d} L(x,\vv)\,\rho(t, dx)\, dt+
\int_0^TF(x,\rho(t, dx))\,\d t;\,  \vv\in
L^2(\rho (t, dx)\, dt),
\end{equation}
subject to $\rho$ and $v$ satisfying
\begin{eqnarray} \label{con.bis} 
\partial_t \rho-\Delta \rho+\nabla\cdot (\rho\,\vv )=0, \, 
\rho(0,\cdot)=\mu\,, \rho(T,\cdot)=\nu.  
\end{eqnarray}
\end{rmk}
We do not know whether $\T$ is a convex transfer. This is equivalent to the question whether $\ell \to T_\ell$ is concave, or equivalently whether the map $\ell \to u_\ell (T, x)$ is concave. \\

  \noindent {\bf Example 11.2: A backward convex coupling which is not a convex transfer}
  
Let $\Omega \subset \R^d$ be a Borel measurable subset with $1 < |\Omega| < \infty$, $\lambda := \frac{1}{|\Omega|}$, and define for any two given probability measures $\mu$, $\nu$ on $\Omega$, the correlation,
\begin{equation}
\T_\lambda(\mu,\nu) = 
\begin{cases}
0 & \text{if }\nu \in \mcal{C}_\lambda(\mu)\\
+\infty & \text{otherwise,}
\end{cases}
\end{equation}
where $\mcal{C}_\lambda(\mu) := \{ \nu \in \mcal{P}(\Omega)\,;\, \lambda\lf|\frac{\d\nu}{\d\mu}\rt| \leq 1\, \mu\text{-a.e.}\}$. Note that when $\mu = \lambda \d x|_{\Omega}$ (the uniform measure on $\Omega$), 
\begin{equation}
\T_\lambda(\lambda \d x|_{\Omega},\nu) =
\begin{cases}
0 & \text{if $\lf|\frac{\d\nu}{\d x}\rt| \leq 1$ Lebesgue-a.e.}\\
+\infty & \text{otherwise.}
\end{cases}
\end{equation}
We claim that $\T_\lambda$ is a  backward convex coupling but not a convex transfer. Indeed, for the first claim, consider $\alpha_m(t) := (\lambda t)^m\log (\lambda t)$ for $m \geq 1$ and $t \geq 0$, and define
\begin{equation}
\T_m(\mu,\nu) := 
\begin{cases}
\int_{\Omega}\alpha_m\lf(\lf|\frac{\d\nu}{\d\mu}\rt|\rt)\d \mu, & \text{if } \nu << \mu,\\
+\infty & \text{otherwise.}
\end{cases}
\end{equation}
By Example 11.1, $\T_m$ is a backward convex transfer and 
\begin{equation}
(\T_{m,\mu})^*(f) = \inf\{\int_{\Omega}[\alpha_m^{\oplus}(f(x) + t) - t]\d\mu(x)\,;\, t \in \R\}.
\end{equation}
The function $\alpha_m^{\oplus}$ can be explicitly computed as 
\begin{equation}
\alpha^{\oplus}_m(t)= 
\begin{cases} e^{-1 + \frac{1}{m-1} W\lf(\beta_m t \rt) }\lf[ \beta_m t + \frac{1}{m} e^{W\lf(\beta_m t \rt)}\rt] & \text{if } t \geq -\frac{\lambda}{m-1} e^{-1},\\
0 & \text{if }  t < -\frac{\lambda}{m-1} e^{-1}.
\end{cases}
\end{equation}
where $\beta_m := \frac{m-1}{\lambda m}e^{\frac{m-1}{m}}$, and $W$ is the \textit{Lambert-W} function. It is easy to see that $\T_\lambda(\mu,\nu) = \sup_{m}\T_m(\mu,\nu)$; hence it is a  backward convex coupling (as a supremum of  backward convex transfers). 

However, $\T_\lambda$ is not a  backward convex  transfer, since 
$$(\T_{\lambda, \mu})^*(f) = (\sup_{m} \T_{m,\mu})^*(f) \leq \inf_{m} \T_{m,\mu}^*(f) = \int \frac{f}{\lambda}\d \mu,$$
with the inequality being in general strict.

Note that this also implies that the Wasserstein projection on the set ${\cal C}_\mu$, that is 
\begin{equation}
W_2^2(P_1[\nu], \nu ) =\inf\{W_2^2(\sigma, \nu); \lf|\frac{\d\sigma}{\d x}\rt| \leq 1\}= \inf\{ \T_\lambda(\lambda\d x|_{\Omega}, \sigma) + W_2^2(\sigma, \nu)\,;\, \sigma \in \mcal{P}(\Omega)\}
\end{equation}
is in fact an inf-convolution of a backward convex coupling $\T_\lambda$ with the linear transfer $W_2^2$, and no duality formula can then be extracted.

 \subsection{Convex and entropic  transfers}
 
 \begin{prop} Let $\alpha: \R^+\to \R$ (resp., $\beta: \R^+ \setminus \{0\} \to \R$) be a convex (resp., concave) increasing functions.
    \begin{enumerate}

   \item If ${\mathcal E}$ is a $\beta$-entropic backward transfer with Kantorovich operator $E ^-$, then it 
is a backward convex transfer with Kantorovich family $(T^-_s)_{s>0}$ given by
\begin{equation}
T_s^-f=sT^-f+(-\beta)^\ominus(s).
\end{equation}

\item Similarly, if ${\mathcal E}$ is an $\alpha$-entropic forward transfer with Kantorovich operator $E ^+$, then it 
is a forward convex transfer with Kantorovich family $(T^+_s)_{s>0}$ given by
\begin{equation}
T_s^+f=sT^+f-\alpha^\oplus(s).
\end{equation}
   \end{enumerate}
   \end{prop} 
\noindent{\bf Proof:}
Use the fact that $(-\beta)$ is convex decreasing to write that for any $g\in C(Y)$, 
\[
\beta \big(\int_XT^-g\, d\mu)=\inf\{ s\int_XT^-g\, d\mu +(-\beta)^\ominus(s); s>0\},
\]
hence $\E$ is a backward convex transfer with Kantorovich family   given by $
T_s^-f=sT^-f+(-\beta)^\ominus(s)$.\\

\noindent{\bf Example 11.4: General entropic functionals are convex transfers}

 Consider the following {\it generalized entropy,}
 \begin{equation}
 {\cal E}_\alpha (\mu, \nu)=\int_X \alpha (|\frac{d\nu}{d\mu}|)\,  d\mu, \quad \hbox{if $\nu<<\mu$ and $+\infty$ otherwise,}
 \end{equation}
 where  $\alpha$ is any strictly convex lower semi-continuous superlinear (i.e., $\lim\limits_{t\to +\infty}\frac{\alpha (t)}{t}=+\infty$) real-valued function on $\R^+$. 
 It is then easy to show \cite{GRS}  that 
 \begin{equation}
  ({\cal E}_\alpha)_\mu^*(f)=\inf\{\int_X[\alpha^\oplus(f(x)+t)-t]\, d\mu(x); t\in \R\}, 
 \end{equation}
In other words, ${\cal E}_\alpha$ is a backward convex transfer with Kantorovich family
\[
T_t^-f(x)=\alpha^\ominus(f(x)+t)-t.
\]
  
\noindent{\bf Example 11.5: The logarithmic entropy is a $\log$-entropic backward transfer}

 The relative logarithmic entropy ${\cal H}(\mu, \nu)$ 
  is defined as
\[
\hbox{${\cal H}(\mu, \nu):=\int_X\log (\frac {d\nu}{d\mu})\, d\nu$ if $\nu<<\mu$ and $+\infty$ otherwise.}
\]
It can also be written as 
\[
\hbox{${\cal H}(\mu, \nu):=\int_X h(\frac {d\nu}{d\mu})\, d\mu$ if $\nu<<\mu$ and $+\infty$ otherwise,}
\]
where $h(t)=t\log t -t+1$, which is strictly convex and positive. Since $h^*(t)=e^t-1$, it follows that 
 \[
  {\cal H}_\mu^*(f)=\inf\{\int_X(e^te^{f(x)}-1-t)\, d\mu(x); t\in \R\}=\log\int_Xe^f\, d\mu.
 \]
 In other words, 
${\cal H}(\mu, \nu)=\sup\{\int_X f\, d\nu-\log\int_Xe^f\, d\mu; f\in C(X)\}$, 
 and ${\cal H}$ is therefore  {\it a  $\beta$-entropic backward transfer} with $\beta (t)=\log t$, and $E ^-f=e^{f}$ is a Kantorovich operator. \\
 ${\cal H}$ is a backward convex transfer since for any $f\in C(X)$, 
 \[
 \log\int_Xe^f\, d\mu=\inf\{s\int_Xe^f\, d\mu +\beta^\ominus(s); s>0\}.
 \]
 In other words, it is a backward convex transfer with Kantorovich family $T^-_sf= se^f+\beta^\ominus(s)$ where $s>0$. \\

\noindent{\bf Example 11.6: The Fisher-Donsker-Varadhan information is a backward convex transfer} \cite{DV}

Consider an $\XX$-valued time-continuous Markov process $(\Omega, \FF,
(X_t)_{t\ge0}, (\pp_x)_{x\in \XX})$  with an invariant probability
measure $\mu.$ Assume the transition semigroup, denoted
$(P_t)_{t\ge0},$ to be completely continuous on
$L^2(\mu):=L^2(\XX,\BB,\mu)$. Let $\LL$ be its generator with
domain $\dd_2(\LL)$ on $L^2(\mu)$ and assume the corresponding Dirichlet form
$
\EE(g,g):=\<-\LL g, g\>_{\mu} $ for $g\in \dd_2(\LL)
$
is closable in $L^2(\mu),$ with closure $(\EE, \dd(\EE))$. The Fisher-Donsker-Varadhan information of $\nu$ with
respect to $\mu$ is defined by
\begin{equation}
{\cal I}(\mu|\nu):=\begin{cases}\EE(\sqrt{f}, \sqrt{f}), \ \ &\text{ if }\ \nu=f\mu, \sqrt{f}\in\dd(\EE)\\
+\infty, &\text{ otherwise.}
\end{cases}
\end{equation}
Note that when $(P_t)$ is $\mu$-symmetric, $\nu\mapsto I(\mu|\nu)$ is
exactly the Donsker-Varadhan entropy i.e.\! the rate function
governing the large deviation principle of the empirical measure
$
L_t:=\frac 1t\int_0^t \delta_{X_s} ds
$
for large time $t$.  The corresponding Feynman-Kac semigroup on $L^2(\mu)$
\begin{equation}\label{Feynman-Kac} P_t^u g(x):=\ee^x g(X_t)
\exp\left(\int_0^tu(X_s)\,ds\right).
\end{equation}
It has been proved in \cite{Wu} that ${\cal I}_\mu^*(f)=\log\|P_1^f\|_{L^2(\mu)}$, which yields that $\cal I$ is a backward convex  transfer.  
\begin{eqnarray*}
{\cal I}_\mu^*(f)=\log\|P_1^f\|_{L^2(\mu)}=\frac{1}{2}\log\|P_1^f\|^2_{L^2(\mu)}=\frac{1}{2}\log \sup\{\int | P_{_1}^fg|^2\, d\mu; \|g\|_{L^2(\mu)}\leq 1\}.
\end{eqnarray*}
In other words, with $\beta (t)=\log t$, we have 
\begin{eqnarray*}
{\cal I}(\mu, \nu)&=&\sup\{\int_Y f\, d\nu-\frac{1}{2}\log \sup\{\int_X | P_1^fg|^2\, d\mu; \|g\|_{L^2(\mu)}\leq 1\}; f\in C(X)\}\\
&=&\sup\{\int_Y f\, d\nu+\sup_{s>0}\sup\limits_{\|g\|_{L^2(\mu)}\leq 1}\frac{1}{2}\{\int_X (-s| P_1^fg|^2-\beta^\ominus(s))\, d\mu  \}; f\in C(X)\}\\
&=&\sup\{\int_Y f\, d\nu-\inf_{s>0}\inf\limits_{\|g\|_{L^2(\mu)}\leq 1}\frac{1}{2}\{\int_X (s| P_1^fg|^2+\beta^\ominus(s))\, d\mu  \}; f\in C(X)\}\\
&=&\sup\{\int_Y f\, d\nu-\int_X T^-_{s, g}f\, d\mu \, ;s \in \R^+, \|g\|_{L^2(\mu)}\leq 1, f\in C(X)\}.
 \end{eqnarray*}
Hence, it is a backward convex transfer, with Kantorovich family $(T^-_{s, g})_{s,g}$ defined by $T^-_{s, g}f=\frac{s}{2}| P_{_1}^fg|^2+\frac{1}{2}\beta^\ominus(s)$.\\

\subsection{Operations on convex and entropic transfers}

The class of backward (resp., forward) convex couplings and transfers
satisfy the following permanence properties. The most important being that the inf-convolution with linear transfers  generate many new  examples of convex and entropic transfers..

\begin{prop}\label{con.con} Let ${\mathcal F}$ be a backward convex coupling (resp., transfer) with  Kantorovich family $(F)_i ^-$, Then,
 \begin{enumerate}
 
 \item If $a\in \R^+\setminus \{0\}$, then  $a{\mathcal F}$ is a backward convex coupling (resp., transfer) with Kantorovich family given by 
$F _{a, i}^-(f)=aF_i ^-(\frac{f}{a}).$
 
\item If %
${\cal T}$ is a backward linear transport on $Y\times Z$ with Kantorovich operator $T^-$, and ${\mathcal F}$ is a backward convex  transfer,   then ${\mathcal F}\star{\mathcal T}$ is a backward convex transfer with Kantorovich family given by $F _{i}^-\circ T ^-$. 
\end{enumerate}
 \end{prop} 
\noindent{\bf Proof:} Immediate. For 2) we calculate the Legendre dual of $({\mathcal F}\star T)_\mu$ at $g\in C(Z)$ and obtain,
  \begin{eqnarray*}
( {\mathcal F}\star T)_\mu^*(g)&=&\sup\limits_{\nu \in {\mathcal P}(Z)}\sup\limits_{\sigma \in  {\mathcal P}(Y)} 
 \left\{\int_{Z} g\, d\nu -  {\mathcal F}(\mu, \sigma) -{\mathcal T}(\sigma, \nu) \right\}\\
&=&  \sup\limits_{\sigma \in  {\mathcal P}(Y)} 
 \left\{{\mathcal T}_\sigma^* (g)-{\mathcal F}(\mu, \sigma) \right\}\\
&=& \sup\limits_{\sigma \in  {\mathcal P}(Y)} 
 \left\{\int_{Y} T^-g\, d\sigma-{\mathcal F}(\mu, \sigma) \right\}\\
&=&({\mathcal F})_\mu^* (T^-(g))\\
 &=& \inf\limits_{i\in I}\int_{X}F_i^-\circ T^-g(x))\, d\mu(x). 
 \end{eqnarray*} 
 The same properties hold for entropic transfers.
That  we will denote by ${\cal E}$ as opposed to ${\cal T}$ to distinguish them from the linear transfers. We shall use $E^+$ and $E^-$ for their Kantorovich operators.

 \begin{prop}\rm  Let $\beta: \R\to \R$ be a concave increasing function and let ${\mathcal E}$ be a backward  $\beta$-entropic transfer with Kantorovich operator $E ^-$. Then,
 \begin{enumerate}

\item If $\lambda \in \R^+\setminus \{0\}$, then  $\lambda{\mathcal E}$  is a backward $(\lambda \beta)$-entropic transfer   with Kantorovich operator $E _\lambda^-(f)=  E ^-(\frac{f}{\lambda})$.

\item  $\tilde {\mathcal E}$ is a forward $((-\beta)^\ominus)^\oplus$-entropic transfer with Kantorovich operator ${\tilde E}^+h=-E^-(-h)$.

\item If ${\mathcal T}$  is a backward linear transfer  on  $Y\times Z$  with Kantorovich operator $T^-$, then ${\mathcal E}\star{\mathcal T}$ is a  a backward  $\beta$-entropic transfer on  $X\times Z$ with Kantorovich operator equal to $E^-\circ T^-$. In other words,  
 \begin{equation}
{\mathcal E}\star{\mathcal T}\, (\mu, \nu)= 
\sup\big\{\int_{Z}g(y)\, d\nu(y)-\beta (\int_{X}E^-\circ T^-g(x))\, d\mu(x));\, g\in C(Z) \big\}. 
\end{equation}

\end{enumerate}
\end{prop}
\noindent {\bf Proof:} 1) is trivial. For 2) note that since $\beta$ is concave and increasing, then 
\begin{eqnarray*}
\tilde {\mathcal T}(\nu, \mu))&=&{\mathcal T}(\mu, \nu))\\
&=&\sup\{\int_Yg\, d\nu-\beta \big(\int_XT^-g\, d\mu); g\in C(Y)\}\\
&=&
\sup\{\int_Y g\, d\nu+ \sup_{s>0}\{\int_X-sT^-g\, d\mu -(-\beta)^\ominus(s)\}; g\in C(X)\}\\
&=&
\sup\{\int_Y g\, d\nu- s\int_XT^-g\, d\mu -(-\beta)^\ominus(s); s>0, g\in C(X)\}\\
&=&
\sup\{s\int_X-T^-(-h)\, d\mu -(-\beta)^\ominus(s)-\int_Y h\, d\nu; s>0, g\in C(X)\\
&=&
\sup\{((-\beta)^\ominus)^\oplus(\int_X-T^-(-h)\, d\mu)-\int_Y h\, d\nu; s>0, h\in C(X)\}.
\end{eqnarray*}
In other words, $\tilde {\mathcal T}$ is a $(\beta^\ominus)^\oplus$-entropic forward  transfer. \\
For 3) we  calculate the Legendre dual of $({\mathcal E}\star T)_\mu$ at $g\in C(Z)$ and obtain,
 \begin{eqnarray*}
( {\mathcal E}\star T)_\mu^*(g)&=&\sup\limits_{\nu \in {\mathcal P}(Z)}\sup\limits_{\sigma \in  {\mathcal P}(Y)} 
 \left\{\int_{Z} g\, d\nu -  {\mathcal E}(\mu, \sigma) -{\mathcal T}(\sigma, \nu) \right\}\\
&=&  
\sup\limits_{\sigma \in  {\mathcal P}(Y)} 
 \left\{{\mathcal T}_\sigma^* (g)-{\mathcal E}(\mu, \sigma) \right\}\\
&=& \sup\limits_{\sigma \in  {\mathcal P}(Y)} 
 \left\{\int_{Y} T^-g\, d\sigma-{\mathcal E}(\mu, \sigma) \right\}\\
&=&({\mathcal E})_\mu^* (T^-(g))\\
 &=& \beta \Big(\int_{X}E^-\circ T^-g(x))\, d\mu(x)). 
 \end{eqnarray*} 
A similar statement holds for forward $\alpha$-entropic  transfers where $\alpha$ is now a convex increasing function on $\R^+$. But we then have to reverse the orders. For example, if ${\mathcal T}$ (resp., ${\mathcal E}$) is a forward linear transfer  on  $Z\times X$ (resp., a forward  $\alpha$-entropic transfer on $X\times Y$) with Kantorovich operator $T^+$ (resp., $E^+$), then ${\mathcal T}\star{\mathcal E}$ is a forward  $\alpha$-entropic transfer on  $Z\times Y$ with Kantorovich operator equal to $E^+\circ T^+$. In other words, 
 \begin{equation}\label{E1}
{\mathcal T}\star{\mathcal E}\, (\mu, \nu)= 
\sup\big\{\alpha\big(\int_{Y}E^+\circ T^+f(y))\, d\nu(y)\big)-\int_{X}f(x)\, d\mu(x);\,  f\in C(X) \big\}.
\end{equation}

\subsection{Subdifferentials of linear and convex transfers}

If  $ \mathcal{T}$ is a linear transfer, then both $ \mathcal{T}_\mu$ and $ \mathcal{T}_\nu$ are convex weak$^*$ lower semi-continuous and one can therefore consider their (weak$^*$) subdifferential $\partial \T_{\mu}$ (resp., $\partial \T_{\nu}$) in the sense of convex analysis. In other words,  
\[
\hbox{$g \in \partial \T_{\mu}(\nu)$ if and only if  $\T(\mu,\nu') \geq \T(\mu,\nu) + \int_Y g\d (\nu' - \nu)$ \,\, for any $\nu'\in {\mathcal P}(Y)$.}
\]
In other words, $g \in \partial \T_{\mu}(\nu)$ if and only if $\T_{\mu}(\nu) + \T^*_\mu (g) = \langle g, \nu\rangle.$ Since $\T_{\mu}(\nu) = \T(\mu,\nu)$ and $\T^*_\mu (g) = \int T^{-}g \d\mu$, we then obtain the following characterization of the subdifferentials. 

 \begin{prop}Let $\T$ be a backward (resp., forward) linear transfer. Then the subdifferential of $\T_\mu :\mathcal{P}(Y) \to \R\cup\{+\infty\}$ at $\nu \in \mathcal{P}(Y)$ (resp., $\T_\nu:\mathcal{P}(X) \to \R\cup\{+\infty\}$ at $\mu \in \mathcal{P}(X)$) is given by
\begin{equation}
\partial \T_{\mu}(\nu) = \left\{g \in C(Y)\,:\, \int_{Y}g(y)\d\nu(y) - \int_{X}T^- g(x)\d\mu(x) = \T(\mu,\nu)\right\}
\end{equation}
respectively,
\begin{equation}
\partial \T_{\nu}(\mu) = \lf\{f \in C(X)\,:\, \int_{Y}T^{+}f(y)\d\nu(y) - \int_{X} f(x)\d\mu(x) = \T(\mu,\nu)\rt\}
\end{equation}
In other words, the subdifferential of $\T_\mu$ at $\nu$ (resp., $\T_\nu$ at $\mu$) is exactly the set of maximisers for the dual formulation of $\T(\mu,\nu)$. 
\end{prop}
It is easy to see that the same expressions hold - with the necessary modifications - for  backward  convex (resp., forward)  transfers, as well as backward $\beta$-entropic (resp., forward $\alpha$-entropic) transfers.

In the following, we observe some elementary consequences for elements in the subdifferential.

 \begin{prop} Suppose $\T$ is a linear backward transfer such that the Dirac masses are contained in $D_1(\T)$. Fix $\mu \in \mcal{P}(X)$ and $\nu \in \mcal{P}(Y)$. Then, there exists $\bar{\pi} \in {\cal K}(\mu, \nu)$ such that  for each $\bar{f} \in \partial \T_\mu(\nu)$, we have
\begin{equation}
T^{-}\bar{f}(x) = \int_{Y}\bar{f}(y)\d\bar{\pi}_x(y) - \T(x,\bar{\pi}_x), \quad \text{for $\mu$-a.e. $x \in X$, 
}
\end{equation}
where $\bar{\pi}_x$ is a disintegration of $\bar{\pi}$ with w.r.t. $\mu$.

Conversely, if $\nu \mapsto \T(\mu,\nu)$ is strictly convex and $\bar{f} \in\partial \T_\mu(\nu)$ for some $\nu \in \mcal{P}(Y)$. If $x\to \sigma_x$ is any selection such that 
$$T^- \bar{f}(x) = \sup_{\sigma}\lf\{\int \bar{f}\d\sigma - \T(\delta_x,\sigma)\rt\}=\int_Y \bar{f}\d\sigma_x - \T(\delta_x,\sigma_x),$$
then $\T(\mu, \nu)$ is attained by the measure $\bar \pi=\int_X\sigma_xd\mu(x)$. 
  \end{prop}
\noindent {\bf Proof:} By a recent result \cite{BBP}, there exists $\bar{\pi} \in {\cal K}(\mu, \nu)$ such that 
\[
\T(\mu, \nu)= \int_{X}\T(x,\bar{\pi}_x)\d\mu(x). 
\]
If $\bar{f} \in \partial \T_\mu(\nu)$, then by definition
\eqs{\label{equality}
\int_{Y}\bar{f}(y)\d\nu(y) - \int_{X}T^- \bar{f}(x)\d\mu(x) = \T(\mu,\nu) = \int_{X}\T(x,\bar{\pi}_x)\d\mu(x), 
}
that is
$
\int_{X}\lf[T^- \bar{f}(x) - \int_{Y}\bar{f}(y)\d\bar{\pi}_x(y) + \T(x,\bar{\pi}_x)   \rt]\d\mu = 0.
$
Since $T^- \bar{f}(x) = \sup_{\sigma}\lf\{\int \bar{f}\d\sigma - \T(x,\sigma)\rt\}$, the quantity in the brackets  is non-negative and we get our claim.

Conversely, If $\bar{f}\in \partial \T_\mu(\nu)$ is non-empty for some $\nu \in \mcal{P}(Y)$, then  
 $
\int \bar{f}\d\nu - \int T^{-}\bar{f}\d\mu = \T(\mu,\nu). $
From the expression $T^- \bar{f}(x) = \sup_{\sigma}\lf\{\int_Y \bar{f}\d\sigma - \T(\delta_x,\sigma)\rt\}$, we know the supremum will be achieved by some $\sigma_x$. Defining $\tilde{\pi}$ by $\d\tilde{\pi}(x,y) = \d\mu(x)\d\sigma_x(y)$, and the right marginal of $\tilde{\pi}$ by $\tilde{\nu}$, we integrate against $\mu$ to achieve
\eqs{
\int T^- \bar{f}\d\mu = \int \bar{f}\d\tilde{\nu} - \int \T(\delta_x, \sigma_x)\d\mu.
}
This shows that $\T(\mu,\tilde{\nu}) = \inf_{\pi \in \Gamma(\mu,\tilde{\nu})}\int \T(\delta_x,\pi_x)\d\mu = \int \T(\delta_x, \sigma_x)\d\mu$, 
and consequently, $\bar{f} \in \partial \T_\mu(\tilde{\nu})$. But by strict convexity, this can only be true if $\tilde{\nu} = \nu$.\qed

While the attainment in the primal problem $\T(\mu, \nu)$ holds in full generality as shown in \cite{BBP}, the attainment in the dual problem depends heavily on the problem at hand \cite{G-K-L1}. However, since this is equivalent to the sub-differentiability of the partial functional $\T_\mu$, we can use general existence results such as the Brondsted-Rockafellar theorem \cite{Phelps}, to state that $\partial \T_\mu (\nu)$ exist for a weak$^*$-dense set of $\nu \in {\cal P}(Y)$, and therefore the dual problem is generically attained. 

 \begin{cor} Suppose $\T$ is a linear backward transfer on $\mcal{P}(X)\times \mcal{P}(Y)$ such that the Dirac masses are contained in $D_1(\T)$. Assume $Y$ is metrizable. Fix $\mu \in \mcal{P}(X)$, then for every $\nu \in \mcal{P}(Y)$ and every $\epsilon >0$, there exists $\nu_\epsilon \in \mcal{P}(Y)$ such that $W_2(\nu, \nu_\epsilon) <\epsilon$ and the dual problem for $\T(\mu, \nu_\epsilon)$ is attained. 

\end{cor}

The following can be seen as Euler-Lagrange equations for variational problems on spaces of measures, and follows closely \cite{FGJ}.

\begin{prop} Let $\T_\alpha(\mu,\nu) := \int_{X} \alpha\lf(\frac{\d\nu}{\d\mu}\rt)\d\mu$ be the generalised entropy transfer considered in Example 11.1, and let $\T$ be any linear backward transfer. For a fixed $\mu$, consider  the functional  
$
I_\mu(\nu) := \T_\alpha(\mu,\nu) - \T(\mu,\nu),
$
and assume $\bar{\nu}$ realises $\inf_{\nu \in \mcal{P}(X)}I_\mu(\nu)$. Then, there exists $\bar{f} \in \partial \T_\mu(\bar{\nu})$ such that the following Euler-Lagrange equation holds for $\bar{\nu}-$a.e. $x \in X$,
\eqs{
\alpha'\lf(\frac{\d\bar{\nu}}{\d\mu}\rt) = \bar{f} + C,  
}
where $C$ is a constant. 

If $\T_\alpha$ is replaced with the logarithmic entropic transfer $\mcal{H}(\mu,\nu) = \int \log(\frac{\d\nu}{\d\mu})\d\nu$, then 
\eqs{
\log \lf(\frac{\d\bar{\nu}}{\d\mu}\rt) = \bar{f} + C.  
}
\end{prop}
\noindent  {\bf Proof:} Recall that $\T_\alpha(\mu,\nu) := \int_{X}\alpha(|\frac{\d\nu}{\d\mu}|) \d\mu$ if $\nu << \mu$ (and $+\infty$ otherwise) is a backward convex transferwith
\eqs{
\T_{\mu}^*(f) = \inf\lf\{\int_{X}[\alpha^{\oplus}(f(x) + t) - t]\d\mu(x)\,;\, t \in \R\rt\},
}
where 
$
T_t^- f(x) := \alpha^{\ominus}(f(x) + t) - t.
$
are the corresponding Kantorovich transfers. Here $\alpha \in C^1$, is strictly convex and superlinear. It follows that  
\eqs{
\alpha'(|\frac{\d\nu}{\d\mu}|) \in \partial \T_{\mu}(\nu).
}
We can see this either directly from the subdifferential definition, or from observing 
$$\alpha^{\oplus}(\alpha'(|\frac{\d\nu}{\d\mu}|)) = \frac{\d\nu}{\d\mu}\alpha'(|\frac{\d\nu}{\d\mu}|) - \alpha(|\frac{\d\nu}{\d\mu}|).$$
In particular,
\as{
\T_{\mu}^*\lf(\alpha'(|\frac{\d\nu}{\d\mu}|)\rt) = \int_{X}\alpha^{\oplus}\lf(\alpha'(|\frac{\d\nu}{\d\mu}|)\rt) \d\mu .
}
The rest is an easy adaptation of Theorem 2.2 in \cite{FGJ}.
 
\section{Inequalities between transfers}
Let ${\mathcal T}$ be a linear or convex coupling, and let ${\mathcal E}_1$, ${\mathcal E}_2$ be entropic transfers on $X\times X$. Standard Transport-Entropy or Transport-Information inequalities are usually of the form 
\begin{equation}\label{one}	
\hbox{${\mathcal T}(\sigma, \mu) \leq \lambda_1 {\mathcal E}_1(\mu, \sigma)$ \quad for all $\sigma \in {\cal P}(X)$,} 
\end{equation}
\begin{equation}\label{two}	
\hbox{${\mathcal T}(\mu, \sigma) \leq \lambda_2 {\mathcal E}_2(\mu, \sigma)$ \quad for all $\sigma \in {\cal P}(X)$,} 
\end{equation}
\begin{equation}\label{three}
\hbox{${\mathcal T}(\sigma_1, \sigma_2) \leq \lambda_1 {\mathcal E}_1(\sigma_1, \mu)+\lambda_2 {\mathcal E}_2( \sigma_2, \mu)$ \quad for all $\sigma_1, \sigma_2 \in {\cal P}(X)$,} 
\end{equation}
where $\mu$ is a fixed measure, and $\lambda_1$, $\lambda_2$ are two positive reals.
In our terminology, Problem \ref{one} (resp., \ref{two}), (resp., \ref{three}) amount to find $\mu$,  $\lambda_1$, and $\lambda_2$ such that 
\begin{equation}
(\lambda_1{\mathcal E}_1)\star (-{\mathcal T}) \, (\mu, \mu)\geq 0,
\end{equation} 
\begin{equation}(\lambda_2{\mathcal E}_2)\star (-\tilde {\mathcal T})\,  (\mu, \mu)\geq 0,
\end{equation} 
 \begin{equation}(\lambda_1\tilde {\mathcal E}_1)\star (-{\mathcal T})\star (\lambda_2{\mathcal E}_2) \, (\mu, \mu)\geq 0,
\end{equation} 
where $\tilde {\mathcal T} (\mu, \nu)={\mathcal T} (\nu, \mu)$. Note for example that 
\[
\tilde {\mathcal E}_1\star  (-{\mathcal T})\star  {\mathcal E}_2\, (\mu, \nu)= \inf\{\tilde{\mathcal E}_1(\mu, \sigma_1)-{\mathcal T}_2(\sigma_1, \sigma_2)+{\mathcal E}_2(\sigma_2, \nu);\, \sigma_1, \sigma_2 \in {\mathcal P}(Z)\}.
\]
We shall therefore write duality formulas for the transfers ${\mathcal E}_1\star (-{\mathcal T})$, ${\mathcal E}_2\star (-\tilde {\mathcal T})$ and $\tilde {\mathcal E}_1\star (-{\mathcal T})\star {\mathcal E}_2$ between any two measures $\mu$ and $\nu$, where ${\mathcal T}$ is any convex transfer, while ${\mathcal E}_1$, ${\mathcal E}_2$ are entropic transfers.

\subsection{Backward convex coupling to backward convex  transfer inequalities}

 We would like to prove inequalities such as 
\begin{equation}\label{one.main}	
\hbox{${\mathcal F}_2(\sigma, \mu) \leq {\mathcal F}_1(\mu, \sigma)$ \quad for all $\sigma \in {\cal P}(X)$,} 
\end{equation}
where  ${\mathcal F}_1$ is a backward convex transfer and ${\mathcal F}_2$ is a  backward convex coupling. We then apply it to Transport-Entropy inequalities of the form 
\begin{equation}\label{one.revised}	
\hbox{${\mathcal F}(\sigma, \mu) \leq \lambda {\mathcal E}\star {\cal T}(\mu, \sigma)$ \quad for all $\sigma \in {\cal P}(X)$,} 
\end{equation}
where ${\mathcal F}$ is a backward  convex coupling, while ${\mathcal E}$ is a $\beta$-entropic transfer and ${\cal T}$ is a backward linear transfer. 

\begin{prop} \label{back-back} 
Let ${\mathcal F}_1$ be a  backward convex  transfer with Kantorovich operator $(F^-_{1,i})_{i\in I}$ on $X_1\times X_2$, and ${\mathcal F}_2$ is a  backward convex coupling on $X_2\times X_3$ with Kantorovich operator $(F^-_{2, j})_{j\in J}$. 
\begin{enumerate}

\item The following duality formula hold: 
\begin{equation}
{\mathcal F}_1\star -{\mathcal F}_2\, (\mu, \nu)=\inf\limits_{f\in C(X_3)}\inf\limits_{j\in J}\sup\limits_{i\in I}\left\{-\int_{X_1} F_{1,i}^-\circ F _{2,j}^-f\, d\mu-\int_{X_3} f\, d\nu\right\}.
\end{equation}
\item If ${\mathcal F}_1$ is a backward $\beta$-entropic transfer on $X_1\times X_2$ with Kantorovich operator $E_1^-$, 
 then
\begin{equation}
{\mathcal F}_1\star -{\mathcal F}_2\, (\mu, \nu)=\inf\limits_{f\in C(X_3)}\inf\limits_{j\in J}\left\{-\beta (\int_{X_1} E_1^-\circ F_{2,j}^-f\, d\mu)-\int_{X_3} f\, d\nu  
\right\}.
\end{equation}
 \end{enumerate}
\end{prop}
\noindent{\bf Proof:} Write
\begin{eqnarray*}
{\mathcal F}_1\star -{\mathcal F}_2\, (\mu, \nu)&=& \inf\{{\mathcal F}_1(\mu, \sigma)-{\mathcal F}_2(\sigma, \nu);\, \sigma \in {\mathcal P}({X_2}) \}\\
 &=&
\inf\limits_{\sigma \in {\mathcal P}({X_2})}\left\{{\mathcal F}_1(\mu, \sigma)- \sup\limits_{f\in C(X_3)}\sup\limits_{j\in J}\left\{\int_{X_3} f\, d\nu-\int_{X_2} F _{2,j}^-f\, d\sigma\right\}\right\}\\
&=&
\inf\limits_{\sigma \in {\mathcal P}({X_2})}\inf\limits_{f\in C(X_3)}\inf\limits_{j\in J}\left\{{\mathcal F}_1(\mu, \sigma)- \int_{X_3} f\, d\nu+\int_{X_2} F _{2,j}^-f\, d\sigma \right\}\\
&=&\inf\limits_{f\in C(X_3)}\inf\limits_{j\in J}\left\{-\sup\limits_{\sigma \in {\mathcal P}({X_2})}\{-\int_{X_2} F _{2,j}^-f\, d\sigma    -{\mathcal F}_1(\mu, \sigma)\}- \int_{X_3} f\, d\nu  \right\}\\
&=&\inf\limits_{f\in C(X_3)}\inf\limits_{j\in J}\left\{-({\mathcal F}_{1})_\mu^*(-F _{2,j}^-f)-\int_{X_3} f\, d\nu  
\right\}\\
&=&\inf\limits_{f\in C(X_3)}\inf\limits_{j\in J}\left\{-\inf\limits_{i\in I}\int_{X_1} F_{1,i}^-\circ -F _{2,j}^-f\, d\mu-\int_{X_3} f\, d\nu\right\}\\
&=&\inf\limits_{f\in C(X_3)}\inf\limits_{j\in J}\sup\limits_{i\in I}\left\{-\int_{X_1} F_{1,i}^-\circ -F _{2,j}^-f\, d\mu-\int_{X_3} f\, d\nu\right\}.
\end{eqnarray*}
2) If ${\mathcal F}_1$ is a backward $\beta$-entropic transfer on $X_1\times X_2$ with Kantorovich operator $E_1^-$, 
 then use in the above calculation that $({\mathcal F}_{1})_\mu^*(-F _{2,j}^-f)=\beta (\int_{X_1} E_1^-\circ -F_{2,j}^-f\, d\mu)$.

\begin{cor} Let  ${\mathcal F}$ be a  backward convex coupling on $Y_2\times X_2$ with Kantorovich family $(F^-_i)_{i\in I}$ and let ${\mathcal E}$ be a backward $\beta$-entropic transfer on $X_1\times Y_1$ with Kantorovich operator $E^-$. Let ${\cal T}$ be a backward linear transfer  on $Y_1\times Y_2$ with Kantorovich operator $T^-$ and $\lambda>0$. Then, for any fixed pair of probability measures $\mu\in {\cal P}( X_1)$ and $\nu\in {\cal P}(X_2 )$, the following are equivalent:
\begin{enumerate}
\item For all $\sigma \in {\cal P}(Y_2)$, we have ${\mathcal F}(\sigma, \nu) \leq \lambda \, {\mathcal E}\star {\mathcal T}\, (\mu, \sigma)$.
 
\item  For all $g\in C(X_2)$ and $i\in I$, we have $\beta \big( \int_{X_1} E^-\circ T^-\circ \frac{-1}{\lambda}F_i^-(\lambda g)\, d\mu\big)+\int_{X_2} g\, d\nu \leq 0$. 
 \end{enumerate}
\end{cor}
In particular, if we apply the above in the case where $\cal E$ is the logarithmic entropy, that is 
  \begin{equation}
\hbox{%$\tilde{\cal H}(\mu, \nu)=
${\cal H}(\mu, \nu)=\int_X\log (\frac {d\nu}{d\mu})\, d\nu$ if $\nu<<\mu$ and $+\infty$ otherwise,}
\end{equation}
 which is a backward $\beta$-entropic transfer with $\beta (t)=\log t$ and $E ^-f=e^f$ as a backward Kantorovich operator.
 
  \begin{cor} Let  ${\mathcal F}$ be a backward convex coupling on $X_2\times Y_2$ with Kantorovich family $(F^-_i)_{i\in I}$ and let ${\mathcal E}$ be a backward $\beta$-entropic transfer on $X_1\times Y_1$. with Kantorovich operator $E^-$. Let ${\cal T}$ be a backward linear transfer  on $Y_1\times Y_2$ with Kantorovich operator $T^-$ and $\lambda>0$. Then, for any fixed pair of probability measures $\mu\in {\cal P}( X_1)$ and $\nu\in {\cal P}(X_2 )$, the following are equivalent:
 
 \begin{enumerate}
 \item  For all $\sigma \in {\cal P}(Y)$, we have ${\mathcal F}(\sigma, \nu) \leq \lambda\,  {\mathcal H}\star {\mathcal T}\, (\mu, \sigma)$.

\item For all $g\in C(X_2)$, we have  $\sup\limits_{i\in I} \int_{X_1} e^{T^-\circ \frac{-1}{\lambda}F_i^-(\lambda g)}\, d\mu \leq e^{-\int_{X_2} g\, d\nu}$.
\end{enumerate}
In particular, if  ${\cal T}$ is the identity transfer and ${\cal F}$ is a backward linear transfer, then the following are equivalent:
\begin{enumerate}
  %\begin{equation}
\item ${\mathcal F}(\sigma, \nu) \leq \lambda\,   {\mathcal H}\, (\sigma, \mu)$ for all $\sigma \in {\cal P}(Y)$  
\item $ \int_{X_1} e^{-F^-(\lambda g)}\, d\mu \leq e^{-\frac{1}{\lambda}}e^{-\int_{X_2} g\, d\nu}$ \quad  for all $g\in C(X_2)$. 
 
  \end{enumerate}
  \end{cor}

\subsection{Forward convex coupling to backward convex transfer inequalities}

 We are now interested in inequalities such as 
\begin{equation}\label{one.main}	
\hbox{${\mathcal F}_2(\nu, \sigma) \leq {\mathcal F}_1(\mu, \sigma)$ \quad for all $\sigma \in {\cal P}(X)$,} 
\end{equation}
where both ${\mathcal F}_1$ and ${\mathcal F}_2$ are convex backward  transfers, and in particular,  Transport-Entropy inequalities of the form 
\begin{equation}\label{one.revised}	
\hbox{${\mathcal F}(\nu, \sigma) \leq \lambda {\mathcal E}\star {\cal T}(\mu, \sigma)$ \quad for all $\sigma \in {\cal P}(X)$,} 
\end{equation}
where  
${\mathcal E}$ is a $\beta$-entropic transfer and ${\cal T}$ is a backward linear transfer. But we can write (\ref{one.main}) as 
\begin{equation}\label{one.main}	
\hbox{$\tilde {\mathcal F}_2(\sigma, \nu) \leq {\mathcal F}_1(\mu, \sigma)$ \quad for all $\sigma \in {\cal P}(X)$,} 
\end{equation}
where now $\tilde {\mathcal F}_2(\sigma, \nu)={\mathcal F}_2(\nu, \sigma)$ is a convex forward  transfer. So, we need to establish the following type of duality.

\begin{prop} \label{forward-back}  
\label{B-F} Let ${\mathcal F}_1$ be a backward convex   transfer with Kantorovich operator $(F^-_{1,i})_{i\in I}$ on $X_1\times X_2$, and let ${\mathcal F}_2$ be a convex forward coupling on $X_2\times X_3$ with Kantorovich operator $(F^+_{2, j})_{j\in J}$. 
\begin{enumerate}

\item The following duality formula then holds: 
\begin{equation}
{\mathcal F}_1\star -{\mathcal F}_2\, (\mu, \nu)=\inf\limits_{g\in C(X_2)}\inf\limits_{j\in J}\sup\limits_{i\in I}\left\{-\int_{X_1} F_{1,i}^-(-g)\, d\nu-\int_{X_3} F _{2,j}^+(g)\, d\nu\right\}.
\end{equation}
\item If ${\mathcal F}_1$ is a backward $\beta$-entropic transfer on $X_1\times X_2$ with Kantorovich operator $E_1^-$, 
 then
\begin{equation}
{\mathcal F}_1\star -{\mathcal F}_2\, (\mu, \nu)=\inf\limits_{g\in C(X_2)}\inf\limits_{j\in J}\left\{-\beta (\int_{X_1} E_1^-(-g)\, d\mu)-\int_{X_3} F _{j}^+(g)\, d\nu\right\}.
\end{equation}
\item If ${\mathcal F}_1$ is a backward $\beta$-entropic transfer with Kantorovich operator $E ^-_1$, and ${\mathcal F}_2$ is a forward $\alpha$-entropic transfer with Kantorovich operator $E ^+_2$, then
\begin{equation}
{\mathcal F}_1\star -{\mathcal F}_2\, (\mu, \nu)=\inf\limits_{g\in C(X_2)}\left\{ -\beta (\int_{X_1} E_1^-(-g)\, d\mu)-\alpha (\int_{X_3} E_2^+g\, d\nu)\right\}.
\end{equation}
\item In particular, if ${\mathcal E}$ is a backward $\beta$-entropic transfer with Kantorovich operator $E ^-$, and ${\mathcal T}$ is a forward linear transfer with Kantorovich operator $T ^+$, then
\begin{equation}
{\mathcal E}\star -{\mathcal T}\, (\mu, \nu)=\inf\limits_{g\in C(X_2)}\left\{ -\beta (\int_{X_1} E^-(-g)\, d\mu)-\int_{X_3} T^+g\, d\nu\right\}.
\end{equation}

\end{enumerate}
\end{prop}
\noindent{\bf Proof:} 1) Assume ${\mathcal F}_1$ is a backward convex   transfer with Kantorovich operator $F ^-_{1,i}$, and ${\mathcal F}_2$ is a forward convex  coupling with Kantorovich operator $F^+_{2, j}$, then
\begin{eqnarray*}
{\mathcal F}_1\star -{\mathcal F}_2\, (\mu, \nu)&=& \inf\{{\mathcal F}_1(\mu, \sigma)-{\mathcal F}_2(\sigma, \nu);\, \sigma \in {\mathcal P}({X_2}) \}\\
&=&
\inf\limits_{\sigma \in {\mathcal P}({X_2})}\left\{{\mathcal F}_1(\mu, \sigma)- \sup\limits_{g\in C(X_2)}\left\{\sup\limits_{j\in J}(\int_{X_3} F _{2,j}^+g\, d\nu) -\int_{X_2} g\, d\sigma\right\}\right\}\\
&=&
\inf\limits_{\sigma \in {\mathcal P}({X_2})}\inf\limits_{g\in C(X_2)}\inf\limits_{j\in J}\left\{{\mathcal F}_1(\mu, \sigma)- \int_{X_3} F _{2,j}^+g\, d\nu +\int_{X_2} g\, d\sigma \right\}\\
&=&\inf\limits_{g\in C(X_2)}\inf\limits_{j\in J}\left\{-\sup\limits_{\sigma \in {\mathcal P}({X_2})}\{-\int_{X_2} g\, d\sigma   -{\mathcal F}_1(\mu, \sigma)\}- \int_{X_3} F _{2,j}^+g\, d\nu)  \right\}\\
&=&\inf\limits_{g\in C(X_2)}\inf\limits_{j\in J}\left\{-({\mathcal F}_{1})_\mu^*(-g)-\int_{X_3} F _{2,j}^+(g)\, d\nu\right\}\\
&=&\inf\limits_{g\in C(X_2)}\inf\limits_{j\in J}\left\{-(\inf\limits_{i\in I}\int_{X_1} F_{1,i}^-(-g)\, d\nu-\int_{X_3} F _{2,j}^+(g)\, d\nu\right\}\\
&=&\inf\limits_{g\in C(X_2)}\inf\limits_{j\in J}\sup\limits_{i\in I}\left\{-\int_{X_1} F_{1,i}^-(-g)\, d\nu-\int_{X_3} F _{2,j}^+(g)\, d\nu\right\}
\end{eqnarray*}
2) If  ${\mathcal F}_1$ is a backward $\beta$-entropic transfer with Kantorovich operator $E ^-$, it suffices to note in the above proof that $({\mathcal F_1})_\mu^*(g)=\beta (\int_X E_1^-(-g)\, d\mu).$\\
 3) If now  ${\mathcal F}_2$ is a forward $\alpha$-entropic transfer with Kantorovich operator $E ^+_2$, then it suffices to note in the above proof that $({\mathcal F}_2)_\nu^*(g)=\alpha (\int_X E_2^+g\, d\nu).$\\
 4) corresponds to when $\alpha (t)=t$.
 
\begin{cor} Let  ${\mathcal F}$ be a convex backward coupling on $X_2\times Y_2$ with Kantorovich family $(F^-_i)_{i\in I}$ and let ${\mathcal E}$ be a backward $\beta$-entropic transfer on $X_1\times Y_1$ with Kantorovich operator $E^-$.  Let ${\cal T}$ be a backward linear transfer  on $Y_1\times Y_2$ with Kantorovich operator $T^-$ and $\lambda>0$. Then, for any fixed pair of probability measures $\mu\in {\cal P}( X_1)$ and $\nu\in {\cal P}(X_2 )$, the following are equivalent:
\begin{enumerate}
\item For all $\sigma \in {\cal P}(Y_2)$, we have ${\mathcal F}(\nu, \sigma) \leq \lambda \, {\mathcal E}\star {\mathcal T}\, (\mu, \sigma)$. 
 \item  For all $g\in C(X_2)$, we have $\beta  \big( \int_{X_1} E^-\circ T^-g)\, d\mu\big)\leq \inf\limits_{i\in I}\frac{1}{\lambda}\int_{X_2} F_i^-(\lambda g) d\nu$.
 \end{enumerate}
 In particular, if ${\mathcal E}_2$ is a backward $\beta_2$-entropic transfer on $X_2\times Y_2$ with Kantorovich operator $E^-_2$, and ${\mathcal E}_1$ is a backward $\beta_1$-entropic transfer on $X_1\times Y_1$ with Kantorovich operator $E^-_1$, then the following are equivalent:
 \begin{enumerate}
\item  For all $\sigma \in {\cal P}(Y_2)$, we have ${\mathcal E}_2(\nu, \sigma) \leq \lambda \, {\mathcal E}_1\star {\mathcal T}\, (\mu, \sigma)$.
 
\item  For all $g\in C(X_2)$ and $i\in I$, we have $\beta_1  \big( \int_{X_1} E_1^-\circ T^-g)\, d\mu\big)\leq \frac{1}{\lambda}\beta_2(\int_{X_2} E_2^-(\lambda g) d\nu)$.  
\end{enumerate}
 \end{cor}
\noindent{\bf Proof:} Note that here, we need the formula for $({\mathcal E}\star {\mathcal T})\star (- {\tilde {\cal F}})(\mu, \nu)$. Since $ {\tilde {\cal F}}$ is now a convex forward  transfer with Kantorovich operators equal to $\tilde F_i^+(g)=-F_i^-(-g)$, we can apply Part 2) of Proposition \ref{B-F} to ${\cal F}_2=\frac{1}{\lambda} {\tilde{ \mathcal F}}$ and ${\cal F}_1={\mathcal E}\star {\mathcal T}$, which is a backward $\beta$-entropic transfer with Kantorovich operator  $E^-\circ T^-$, to obtain
\[
({\mathcal E}\star {\mathcal T})\star (- {\tilde {\cal F}})(\mu, \nu)=\inf\limits_{g\in C(X_2)}\inf\limits_{j\in J}\left\{-\beta (\int_{X_1} E^-\circ T^-g\, d\mu)+\frac{1}{\lambda}\int_{X_3} F _{j}^-(\lambda g)\, d\nu\right\}.
\]
 A similar argument applies for 2).\\
We now apply the above to the case where $\cal E$ is the backward logarithmic transfer to obtain,
  
\begin{cor} Let  ${\mathcal F}$ be a backward convex   transfer on $X_2\times Y_2$ with Kantorovich family $(F^-_i)_{i\in I}$, and  
 let ${\cal T}$ be a backward linear transfer  on $Y_1\times Y_2$ with Kantorovich operator $T^-$ and $\lambda>0$. Then, for any fixed pair of probability measures $\mu\in {\cal P}( X_1)$ and $\nu\in {\cal P}(X_2 )$, the following are equivalent:
\begin{enumerate}
\item For all $\sigma \in {\cal P}(Y_2)$, we have ${\mathcal F}(\nu, \sigma) \leq \lambda \, {\mathcal H}\star {\mathcal T}\, (\mu, \sigma)$ 
 
 \item  For all $g\in C(X_2)$, we have $\log  \big( \int_{X_1} e^{T^-g}\, d\mu\big)\leq \inf\limits_{i\in I}\frac{1}{\lambda}\int_{X_2} F_i^-(\lambda g) d\nu$.
 \end{enumerate}
\end{cor}

\begin{rmk} \rm An immediate application of (4) in Proposition \ref{forward-back} is the following result in \cite{C-K} 
\begin{equation}
\inf\{{\overline{\mathcal W}}_2(\mu, \sigma)+{\mathcal H}(dx, \sigma); \sigma \in {\cal P}(\R^d)\}=\inf\{-\log \int e^{-f^*}\, dx +\int f\, d\mu; f\in {\cal C}(\R^d)\},
\end{equation} 
where ${\cal C}onv(\R^d)$ is the cone of convex functions on $\R^d$, and ${\overline{\mathcal W}}_2(\mu, \sigma)=-{\mathcal W}_2(\sigma, {\bar \mu})$, the latter being the Brenier transfer of  Example 3.12 and $\bar \mu$ is defined as $\int f(x) d \bar \mu(x)=\int f(-x) d \mu(x)$.  Note that in this case, $T^+f (x)=-f^*(-x)$, $E^-f= e^f$ and $\beta (t)=\log t$, and since $g^{**} \leq g$, 
\begin{eqnarray*}
\inf\{{\overline{\mathcal W}}_2(\mu, \sigma)+{\mathcal H}(dx, \sigma); \sigma \in {\cal P}(\R^d)\}
&=&{\mathcal H}\star (-{\mathcal W}_2)(dx, \bar \mu)\\
&=&\inf\{-\log \int e^{-g}\, dx +\int g^*(x)\, d\mu; g\in C(\R^d)\}\\
&=& \inf\{-\log \int e^{-f^*}\, dx +\int f\, d\mu; f\in {\cal C}onv(\R^d)\}.
\end{eqnarray*} 
What is remarkable in the result of Cordero-Erausquin and Klartag \cite{C-K} is the characterization of those measures $\mu$ (the moment measures) for which there is attainment in both minimization problems. 
\end{rmk}  

\subsection{Maurey-type inequalities}
We are now interested in inequalities of the following type:
For all $\sigma_1\in {\mathcal P}(X_1), \sigma_2\in {\mathcal P}(X_2)$, we have 
\begin{equation}
{\mathcal F}(\sigma_1, \sigma_2) \leq \lambda_1 {\mathcal T}_1\star {\mathcal H}_1(\sigma_1, \mu)+\lambda_2 {\mathcal T}_2\star {\mathcal H}_2( \sigma_2, \nu).
\end{equation}
This will requires a duality formula for the expression
$
\tilde {\mathcal E}_1\star (-{\mathcal F})\star  {\mathcal E}_2,
$
where ${\cal F}$ is a backward convex transfer and $ {\mathcal E}_1$,  ${\mathcal E}_2$ are forward entropic transfers.

\begin{thm} Assume ${\mathcal F}$ is a backward convex  coupling on $Y_1\times Y_2$ with Kantorovich family $(F^-_i)_{i\in I}$, ${\mathcal E}_1$ (resp., ${\mathcal E}_2$) is a forward $\alpha_1$-entropic transfer on $Y_1\times X_1$ (resp., a  forward $\alpha_2$-entropic transfer on $Y_2\times X_2$) with Kantorovich operator  $E_1^+$ (resp., $E_2^+$), then  for any $(\mu, \nu)\in {\cal P}(X_1)\times {\cal P}(X_2)$, we have
\begin{equation}
\tilde {\mathcal E}_1\star (-{\mathcal F}) \star  {\mathcal E}_2 \, (\mu, \nu)=\inf\limits_{i\in I} \inf\limits_{f\in C(X_3)} \left\{\alpha_1 \big(\int_{X_1} E_1^+\circ F_{i}^-f\, d\mu)+ \alpha_2 (\int_{X_2}E_2^+(f)\, d\nu)\right\}.
\end{equation}
 \end{thm}

\noindent{\bf Proof:} If ${\mathcal E}_1$  a forward $\alpha_1$-entropic transfer on $Y_1\times X_1$, then ${\tilde{\mathcal E}}_1$ is a backward $-(\alpha_1^\oplus)^\ominus$-entropic transfer on $X_1\times Y_1$ with Kantorovich operator $\tilde {E_1^-}g=-E_1^+(-g)$. Apply Proposition \ref{back-back} with ${\cal F}_1={\tilde{\mathcal E}}_1$, and ${\cal F}_2={\cal F}$ to get 
\begin{eqnarray*}
\tilde {\mathcal E}_1\star (-{\mathcal F})\, (\mu, \nu)&=&\inf\limits_{f\in C(X_3)}\inf\limits_{i\in I}\left\{(\alpha_1^\oplus)^\ominus \big(\int_{X_1} -E_1^+\circ F_{i}^-f\, d\mu)-\int_{X_3} f\, d\nu \right\}\\
&=&\inf\limits_{f\in C(X_3)}\inf\limits_{i\in I}\left\{\alpha_1 \big(\int_{X_1} E_1^+\circ F_{i}^-f\, d\mu)-\int_{X_3} f\, d\nu \right\}.
\end{eqnarray*}
Write now, 
 \begin{eqnarray*}
\tilde {\mathcal E}_1\star (-{\mathcal F}) \star  {\mathcal E}_2 \, (\mu, \nu)&=&\inf\left\{\tilde {\mathcal E}_1\star (-{\mathcal F}) (\mu, \sigma)+{\mathcal E}_2(\sigma, \nu);\, \sigma \in {\mathcal P}(Y_2)\right\}\\
&=&\inf\limits_{\sigma\in {\mathcal P}(Y_2)} \inf\limits_{f\in C(X_3)}\inf\limits_{i\in I}\left\{\alpha_1 \big(\int_{X_1} E_1^+\circ F_{i}^-f\, d\mu)-\int_{X_3} f\, d\sigma+{\mathcal E}_2(\sigma, \nu)\right\}\\
&=& \inf\limits_{i\in I}\inf\limits_{f\in C(X_3)} \left\{\alpha_1 \big(\int_{X_1} E_1^+\circ F_{i}^-f\, d\mu)-\sup\limits_{\sigma \in {\mathcal P}(Y_2)}\{\int_{Y_2} f\, d\sigma-{\mathcal E}_2(\sigma, \nu)\} \right\}\\
&=&\inf\limits_{i\in I} \inf\limits_{f\in C(X_3)} \left\{\alpha_1 \big(\int_{X_1} E_1^+\circ  F_{i}^-f\, d\mu)+ \alpha_2 (\int_{X_2}E_2^+(-f)\, d\nu)\right\}.
\end{eqnarray*}

 \begin{cor} Assume ${\mathcal E}_1$ (resp., ${\mathcal E}_2$) is a forward $\alpha_1$-entropic transfer on $Z_1\times X_1$ (resp., $\alpha_2$-entropic transfer on $Z_2\times X_2$) with Kantorovich operator $E_1^+$ (resp., $E_2^+$). Let ${\mathcal T}_1$ (resp., ${\mathcal T}_2$) be forward linear transfers on $Y_1\times Z_1$ (resp., $Y_2\times Z_2$) with Kantorovich operator $T_1^+$ (resp., $T_2^+$), and let ${\mathcal F}$ be a  backward convex coupling on $Y_1 \times Y_2$ with Kantorovich family $(F_i^-)_i$. Then, for any given $\lambda_1, \lambda_2\in \R^+$ and $(\mu, \nu)\in  {\mathcal P}(X_1) \times {\mathcal P}(X_2)$, the following are equivalent:
\begin{enumerate}
\item For all $\sigma_1\in {\mathcal P}(Y_1), \sigma_2\in {\mathcal P}(Y_2)$, we have 
\begin{equation}
{\mathcal F}(\sigma_1, \sigma_2) \leq \lambda_1 {\mathcal T}_1\star {\mathcal E}_1(\sigma_1, \mu)+\lambda_2 {\mathcal T}_2\star {\mathcal E}_2( \sigma_2, \nu).
\end{equation}
 \item For all $g\in C(Y_2)$ and all $i\in I$, we have 
\begin{equation}
\lambda_1\alpha_1 \big(\int_{X_1} E_1^+\circ T_1^+\circ (\frac{1}{\lambda_1}F_i^-g)\, d\mu\big)+ \lambda_2\alpha_2 (\int_{X_2} E_2^+\circ T_2^+(\frac{-1}{\lambda_2}g)\, d\nu)\geq 0.
\end{equation}
\end{enumerate}
\end{cor} 
\noindent {\bf Proof:} It suffices to apply the above with the forward $\lambda_i\alpha_i$-transfers ${\cal F}_i:=\lambda_i {\mathcal T}_i\star {\mathcal E}_i$, whose Kantorovich operators are $F_i(g)=E_i^+\circ T_i^+(\frac{g}{\lambda_i})$ for $i=1,2$.\\

By applying the above to ${\cal E}_i(\mu, \nu)=:\cal{H}$ the forward logarithmic entropy where $\alpha_i(t)=-\log(-t)$ and Kantorovich operator $E^+f=e^{-f}$, we get the following extension of a celebrated result of Maurey \cite{Mau}.

\begin{cor} Assume ${\mathcal F}$ is a convex backward coupling on $Y_1\times Y_2$ with Kantorovich family $(F^-_i)_{i\in I}$, and  
let ${\mathcal T}_1$ (resp., ${\mathcal T}_2$) be forward linear transfer on $Y_1\times X_1$ (resp., $Y_2\times X_2$) with Kantorovich operator  $T_1^+$ (resp., $T_2^+$), then  for any given $\lambda_1, \lambda_2\in \R^+$ and $(\mu, \nu)\in {\cal P}(X_1)\times {\cal P}(X_2)$, the following are equivalent:
\begin{enumerate}
\item For all $\sigma_1\in {\mathcal P}(X_1), \sigma_2\in {\mathcal P}(X_2)$, we have 
\begin{equation}
{\mathcal F}(\sigma_1, \sigma_2) \leq \lambda_1 {\mathcal T}_1\star {\mathcal H}(\sigma_1, \mu)+\lambda_2 {\mathcal T}_2\star {\mathcal H}( \sigma_2, \nu).
\end{equation}

\item For all $g\in C(Y_2)$ and all $i\in I$, we have 
\begin{equation}
(\int_{X_1} e^{-T_1^+\circ \frac{1}{\lambda_1}F_i^-g}\, d\mu\big)^{\lambda_1} (\int_{X_2}e^{-T_2^+(\frac{1}{-\lambda_2}g)}\, d\nu)^{\lambda_2} \leq 1.
\end{equation}
\end{enumerate}
If ${\mathcal T}_1={\mathcal T}_2$ are the identity transfer, then the above is equivalent to saying that for all $g\in C(Y_2)$ and all $i\in I$, we have 
\begin{equation}
(\int_{X_1} e^{\frac{-1}{\lambda_1}F_i^-g}\, d\mu\big)^{\lambda_1} (\int_{X_2}e^{\frac{1}{\lambda_2}g}\, d\nu)^{\lambda_2} \leq 1.
\end{equation}

\end{cor}

\end{document}